\documentclass[11pt]{amsart}
\usepackage{amssymb,amsmath,txfonts,amsrefs, amscd}
\usepackage{array,multirow,bigstrut,xcolor}
\usepackage{appendix}
\usepackage[all]{xy}

\newtheorem{thm}{Theorem}[section]
\newtheorem{prop}[thm]{Proposition}
\newtheorem{lemma}[thm]{Lemma}
\newtheorem{remark}[thm]{Remark}

\newtheorem{definition}[thm]{Definition}
\newtheorem{cor}[thm]{Corollary}

\DeclareMathOperator{\ep}{\epsilon}

\numberwithin{equation}{section}
     
\def\D{{\mathcal{D}}} 

\def\tmJ{\tilde{\mathcal{J}}}
\def\tomega{\tilde{\omega}}

\def\pa{\partial}
\def\bpa{\bar{\partial}}
\def\dpp{{\partial_+}}
\def\dpm{{\partial_-}}
\def\dpps{{\partial^*_+ }}
\def\dpms{{\partial^*_-}}  

\def\om{\omega}
\def\w{\wedge}

\def\drho{{d\rho}}
\def\ep{{\epsilon}}

\def\ba{{\beta^1}}
\def\bb{{\beta^2}}
\def\bc{{\beta^3}}
\def\bd{{\beta^4}}

\def\tbeta{\tilde{\beta}}

\def\xa{{\zeta_1}}
\def\xb{{\zeta_2}}
\def\xc{{\zeta_3}}
\def\xas{{\zeta^2_1}}
\def\xbs{{\zeta^2_2}}
\def\xcs{{\zeta^2_3}}
\def\xaf{{\zeta^4_1}}

\def\xcf{{\zeta^4_3}}
\def\wa{{w_1}}
\def\wb{{w_2}}
\def\wc{{w_3}}
\def\wdd{{w_4}}

\def\baa{c^{1}_{n-1}}
\def\bab{c^{2}_{n-1}}
\def\bba{c^{3}_{n-1}}
\def\bbb{c^{4}_{n-1}}
\def\bca{c^{5}_{n-2}}
\def\bcb{c^{6}_{n-2}}

\def\ga{\gamma}
\def\Th{{\Theta_{12}}}
\def\Thp{{\Theta'_{34}}}

\def\Om{{\Omega}}

\def\La{{\Lambda}}
\def\tbeta{{\tilde{\beta}}}
\def\tN{{{N_T}}}
\def\tpM{{(\partial M)_T}}

\def\CP{{\mathcal{P}}}

\def\Th{{\Theta_{12}}}
\def\Thpp{{\Theta_{12}\Theta'_{34}}}

\newcommand\SmallMatrix[1]{{%
  \tiny\arraycolsep=-.3\arraycolsep\ensuremath{\begin{pmatrix}#1\end{pmatrix}}}}

\newcommand\SmallMatrixE[1]{{%
  \tiny\arraycolsep=-.8\arraycolsep\ensuremath{\begin{pmatrix}#1\end{pmatrix}}}}

\begin{document}

\title[]{Symplectic Boundary Conditions and Cohomology}
\author{ Li-Sheng Tseng}
\address{Department of Mathematics, University of California, Irvine, CA 92697}
\email{lstseng@math.uci.edu}
\author{Lihan Wang}
\address{Department of Mathematics \& Statistics, California State University, Long Beach, CA 90840}
\email{lihan.wang@csulb.edu}



\begin{abstract}

We introduce new boundary conditions for differential forms on symplectic manifolds with boundary.  These boundary conditions, dependent on the symplectic structure, allows us to write down elliptic boundary value problems for both second-order and fourth-order symplectic Laplacians and establish Hodge theories for the cohomologies of primitive forms on manifolds with boundary.  We further use these boundary conditions to define a relative version of the primitive cohomologies and to relate primitive cohomologies with Lefschetz maps on manifolds with boundary.  As we show, these cohomologies of primitive forms can distinguish certain K\"ahler structures of  K\"ahler manifolds with boundary.

\bigskip
\end{abstract}

\maketitle
\tableofcontents

\maketitle
\tableofcontents

\section{Introduction}

In this paper, we initiate the search for global invariants of differential forms on symplectic manifolds with boundary.  Manifolds with boundary are important in symplectic geometry as they are central for cobordism theory and have appeared in various contexts such as in the study of symplectic filling and symplectic field theory (see, for example, \cites{Eliashberg, Etnyre, EGH}).   The consideration of differential forms on such spaces also has physical motivations and applications.  For instance, they are involved in a system of differential equations with singular source charges of Type II string theory \cite{Tomasiello, TY3}.  Analyzing the solution space of such a physical system would involve solving for differential forms on symplectic manifolds with certain prescribed boundary conditions along the location of source charges.

We begin our study by analyzing cohomologies on symplectic manifolds with boundary.  Of particular interest here are the primitive cohomologies introduced by Tseng-Yau \cite{TY2}. These cohomologies are defined on the space of primitive differential forms.  Roughly, primitive forms are those that are trivial under the interior product with the symplectic form.  (For a precise definition, see Definition 2.1.)  The primitive cohomologies depend on the symplectic form and have significant differences with other known cohomologies \cite{TY2, TTY}. Of note, the primitive cohomologies have associated elliptic Laplacians, which we shall simply refer to here as symplectic Laplacians.
 
One of the main goals of this paper is to define and analyze the unique harmonic representative for each class of the primitive cohomologies. That is, we are interested in the Hodge theory of the symplectic Laplacians on symplectic manifolds with boundary.  As is well-known, conditions on differential forms (and sometimes also of the boundary) are necessary to establish the Hodge theory of elliptic operators on manifolds with boundary. For instance, in Riemannian geometry, the well-known Dirichlet $(D)$ and Neumann $(N)$ boundary conditions on differential forms are needed for the Hodge theory of the Laplace-de Rham operator \cite{F, M1}. Similarly, in complex geometry, in order to establish the Hodge theory of the Dolbeault Laplacian, the $\bpa$-Neumann boundary condition is usually assumed on differential forms in addition to imposing the strongly pseudoconvex condition on the boundary \cite{M2}. In both cases, the boundary conditions on differential forms have garnered wide interests and applications. (For a general reference, see \cite{S} for the Riemannian case and \cite{M2} for the complex case.) In Table \ref{Tab1}, we summarize the well-known boundary conditions involved in the Hodge theory for these two cases.    
 
Clearly, our first task is to identify the boundary conditions that are natural for differential forms on symplectic manifolds.  Heuristically, boundary conditions that have good analytical properties are typically closely related to the natural differential operators on the manifold. Consider for example the boundary conditions in Table \ref{Tab1}. The Dirichlet $(D)$ and the Neumann $(N)$ boundary conditions are defined using the exterior derivative operator $d$ and its adjoint $d^*$, respectively, while the $\bpa$-Neumann boundary condition uses the Dolbeault operator $\bar{\partial}$. Therefore, we should ask what natural differential operators should we work with in the symplectic case?

For any symplectic manifold $(M^{2n}, \omega)$, it was observed by Tseng-Yau  \cite{TY2} that there are two, first-order, linear differential operators that appear in a symplectic decomposition of the standard exterior derivative operator:
$$ d = \dpp + \om \w \dpm~.$$  
The pair $(\dpp, \dpm)$ are dependent on the symplectic structure $\om$ and have good properties: (i) $(\dpp)^2 =(\dpm)^2 =0$; (ii) $\om\w\dpp\dpm=-\om\w\dpm\dpp$; (iii) $[\om, \dpp]=[\om, \om\w\dpm]=0\,$.  
In addition, Tseng-Yau \cite{TY1,TY2} also identified the second-order differential operator, $\dpp\dpm$, as an important operator to study for symplectic manifolds.  

With respect to this triplet of differential operators, $(\dpp, \dpm,\, \dpp\dpm\,)$,  we will introduce {\it symplectic boundary conditions} on forms that are analogous to the standard Dirichlet and Neumann boundary conditions of Riemannian geometry.  In the case of the two first-order operators $(\dpp, \dpm)$, we can straightforwardly define four new boundary conditions which we denote by $D_+, N_+, $ and $D_-, N_-$, as listed in Table \ref{Tab2}.  The case of the second-order operator $\dpp\dpm$ is much more subtle.  Generally, boundary conditions associated with second-order operators are not well-understood or studied.  We however are led to define two boundary conditions, $D_{++}$ and $N_{--}$, associated with $\dpp\dpm$ given in Table \ref{Tab2.5}.

 \begin{table}[t!]\label{Tab1}
\caption{The standard boundary conditions on manifolds with boundary.  The notation $\sigma_{\D}$ denotes the principal symbol of the differential operator $\D$, and $\rho$ is the boundary defining function. }
\renewcommand{\arraystretch}{1.2}
\begin{tabular}{c | c | c }
  & Riemannian $(M, g)$ & Complex $(M, J, g)$ \\
\hline
Cohomology& de Rham cohomology & Dolbeault cohomology \\
&$H^{\ast}(M)$& $H^{p,q}(M)$\\
\hline
Laplacian & $\Delta_d = d\, d^{\ast} + d^{\ast} d$  & $\Delta_{\bar{\partial}}=\bpa\, \bpa^{\ast}+ \bpa^{\ast}\bpa $ \\
\hline
Boundary & Dirichlet $(D)$: $\sigma_{d}(d\rho)\, \eta~|_{\partial M}=0\,$; & 
$\bpa$-Neumann: $\sigma_{\bpa^{\ast}}(d\rho)\, \eta~|_{\partial M}=0\,$,\\
Conditions & Neumann $(N)$: $\sigma_{d^{\ast}}(d\rho)\, \eta~|_{\partial M}=0\,$. & $\partial M$ strongly pseudoconvex\\\hline
\end{tabular}

\

\

\end{table}

\begin{table}[t]
\caption{Symplectic boundary conditions $\{D_+, N_+, D_-, N_-\}$ associated with $(\dpp, \dpm)$.}\label{Tab2}
\renewcommand{\arraystretch}{1.2}
\begin{tabular}{l|c| c}
&  $\dpp$ & $\dpm$ \\
\hline
Dirichlet-type & $(D_+): \sigma_{\partial_{+}}(\drho)\, \eta~|_{\partial M}=0$ &$(D_-): \sigma_{\partial_{-}}(\drho)\, \eta~|_{\partial M}=0$\\ 
\hline
Neumann-type & $(N_+):  \sigma_{\partial^{\ast}_{+}}(\drho)\, \eta~|_{\partial M}=0$ & $(N_-): \sigma_{\partial_{-}^{\ast}}(\drho)\, \eta~|_{\partial M}=0$\\
\hline
\end{tabular}
\end{table}

\begin{table}[t]
\caption{Symplectic boundary conditions associated with $\dpp\dpm\,$.  Notationally, $\sigma_{\dpp\dpm}$ denotes the principal symbol of $\dpp\dpm$,  $\rho$ is the boundary defining function, and $\mathcal{L}_{\vec{n}}$ is the Lie derivative with respect to the inward normal vector $\vec{n}\,$. 
}\label{Tab2.5}
\renewcommand{\arraystretch}{1.2}
\begin{tabular}{c|c}
Boundary Condition&Definition\\
\hline
$D_{++}$& $(D_{+-}): \sigma_{\dpp\dpm}(\drho)\, \eta~|_{\partial M}=0$\\
& $\left\{2\,\dpp\dpm(\rho\,\eta) - \frac{1}{2}\mathcal{L}_{\vec{n}}\big[\partial_{+}\partial_{-}(\rho^2 \eta)\big] \right\}|_{\partial M} =0$\\
\hline
$N_{--}$& $(N_{+-}):  \sigma_{(\dpp\dpm){}^{\ast}}(\drho)\, \eta~|_{\partial M}=0$\\
 & $\left\{2\,(\dpp\dpm)^*(\rho\,\eta) - \frac{1}{2}\mathcal{L}_{\vec{n}}\big[(\dpp\dpm)^*(\rho^2 \eta)\big] \right\}|_{\partial M} =0$\\
\hline
\end{tabular}
\end{table}

The six symplectic boundary conditions $\{D_+, N_+, D_-, N_-, D_{++}, N_{--}\}$ are in general weaker conditions than the standard Dirichlet and Neumann conditions.  However, they should be thought of as the natural boundary conditions associated with $(\dpp, \dpm, \,\dpp\dpm\,)$.  For one, these symplectic conditions arise when considering the adjoint of the three operators and imposing that any boundary integral contributions vanish. Importantly, they are also preserved under the action of the corresponding differential operator: $\dpp,\, \dpm$, or $\dpp\dpm$. For example, if a form $\eta$ satisfies the $D_+$ boundary condition, then $\dpp \eta$ will also satisfy the $D_+$ condition.  We will describe these and other useful properties of the symplectic boundary conditions in detail in Section 3.

The six symplectic boundary conditions in Tables \ref{Tab2} and \ref{Tab2.5}  turn out to be useful in establishing Hodge decompositions of forms.  With the appropriate pairing of symplectic boundary conditions and symplectic Laplacians, we write down in Section 4 systems of partial differential equations on forms that are elliptic.   Having done so, we can then apply standard elliptic theory on manifolds with boundary for these types of systems of equations, standardly referred to as elliptic boundary value problems, to obtain Hodge-type decompositions of forms involving harmonic fields.  Here,  harmonic fields are forms that are, for example, in the $\dpp$ case, both $\dpp$-closed and $\dpps$-closed. (Note the distinction in the boundary case: a harmonic {\it form}, that is a zero of the Laplacian, is not necessarily a harmonic {\it field}.)  We shall show that the space of these harmonic fields satisfying certain symplectic boundary conditions is finite-dimensional. Moreover, we will apply the obtained Hodge decompositions to prove the existence of solutions for several other types of boundary value problems.

Having studied the relevant partial differential equations and Hodge decompositions, we introduce and analyze both the absolute and relative primitive cohomology on symplectic manifolds with boundary in Section 5. We list their definitions in Table \ref{Tab3}, where $\Omega^{k}$ there denotes the space of differential $k$-forms and $P^k$ the subspace of primitive $k$-forms.  We will use the obtained Hodge decompositions to demonstrate that each class of the primitive cohomologies in Table \ref{Tab3} has a unique harmonic field, that satisfies certain symplectic boundary condition, as its representative.  Such harmonic fields may then be used to demonstrate a natural pairing isomorphism between the absolute primitive cohomology and the relative primitive cohomology. 

Additionally, with the six symplectic boundary conditions, we can study Lefschetz maps on manifolds with boundary and establish relations between relative de Rham cohomology and relative primitive cohomology. As is well-known, on closed K\"ahler manifolds, Lefschetz maps of the form
\begin{align*}
L\!:\quad H^{k}(M)\quad \rightarrow&\quad H^{k+2}(M)\\
\quad[\eta] \quad \quad\rightarrow&\quad [\omega \wedge \eta]~,
\end{align*}
can be easily understood by the Hard Lefschetz Theorem. In \cite{TTY}, Tsai-Tseng-Yau studied Lefschetz maps for general, non-K\"ahler symplectic manifolds and showed that the kernels and cokernels of these Lefschetz maps can be characterized by the primitive cohomologies.  Here, we find similar results for cohomologies defined on symplectic manifolds with boundary and further extend their results to the relative cohomology case. We summarize our Lefschetz maps results in Table \ref{Tab4}.

\begin{table}
\caption{Absolute and relative cohomologies on manifolds with boundary.  
}\label{Tab3}
\renewcommand{\arraystretch}{2.0}
\begin{tabular}{c|c|c}
 & Absolute Cohomology & Relative Cohomology\\
\hline
De Rham &$H^{k}(M)=\dfrac{{\ker d \cap \Omega^k(M)}}{ d \Omega^{k-1}\!(M)}$ & $H^{k}(M, \partial M)=\dfrac{\ker d \cap \Omega_D^k(M) }{ d \Omega_D^{k-1}\!(M)}$
\\
\hline
\multirow{4}*[-10pt]{Primitive}
&  $PH^{k}_+(M)=\dfrac{\ker \partial_{+} \cap P^k(M)}{ \partial_{+} P^{k-1}(M)}$ &$PH^{k}_+(M, \partial M)=\dfrac{\ker \partial_{+} \cap P_{D_{+}}^k\!(M)}{ \partial_{+} P_{D_{+}}^{k-1}(M)}$ \\ 
& $PH^n_+(M)=\dfrac{\ker  \partial_{+}\partial_{-} \cap P^{n}(M)}{\partial_{+} P^{n-1}(M)}$  &$PH^n_+(M, \partial M)=\dfrac{\ker  \partial_{+}\partial_{-} \cap P^{n}_{D_{++}}\!\!(M)}{\partial_{+} P_{D_+}^{n-1}(M)}$

\bigstrut\\
\cline{2-3}
& $PH^{k}_-(M)=\dfrac{\ker \partial_{-} \cap P^k(M)}{ \partial_{-} P^{k+1}(M)}$
 &  $PH^{k}_-(M, \partial M)=\dfrac{\ker \partial_{-} \cap P_{D_{-}}^k\!(M) }{ \partial_{-} P_{D_{-}}^{k+1}(M)}$\\
& $PH^{n}_-(M)=\dfrac{\ker \partial_{-} \cap P^n(M)}{ \partial_{+}\partial_{-} P^{n}(M)}$
& $PH^n_-(M, \partial M)=\dfrac{\ker\partial_{-} \cap P_{D_-}^n\!\!(M)}{ \partial_{+}\partial_{-} P_{D_{++}}^{n}(M)}  $\bigstrut
\\
\hline
\end{tabular}
\end{table}

\begin{table}
\caption{Relations of primitive cohomology with Lefschetz map.}\label{Tab4}
\renewcommand{\arraystretch}{1.2}
\begin{tabular}{l |r l}
Cohomology&  & \qquad\qquad$k \leq n$\\ \hline
&$PH^k_+(M)\cong$ &${\rm coker}[L\!: {H^{k-2}(M) \to H^k(M)}]$ \\ 
~~Absolute& &\quad $\oplus \ker [L\!: {H^{k-1}(M)}\to H^{k+1}(M)]$\\
~~Primitive&$PH^k_-(M)\cong$ &${\rm coker}[L\!: {H^{2n-k-1}(M)\to H^{2n-k+1}(M)}] $ \\
& & \quad $\oplus \ker [L\!: {H^{2n-k}(M)}\to H^{2n-k+2}(M)]$\\
\hline
&$PH^k_+(M, \partial M)\cong$ &${\rm coker}[L\!: {H^{k-2}(M, \partial M)}\to {H^{k}(M, \partial M)}]$\\ 
~~Relative& &$\quad\oplus \ker [L\!: {H^{k-1}(M, \partial M)} \to {H^{k+1}(M, \partial M)} ]$\\
~~Primitive&$PH^k_-(M, \partial M)\cong$ &${\rm coker}[L\!: {H^{2n-k-1}(M, \partial M)}\to {H^{2n-k+1}(M, \partial M)}]$ \\ 
&&\quad$\oplus \ker [L\!: {H^{2n-k}(M, \partial M)} \to {H^{2n-k+2}(M, \partial M)}]$\\
\hline
\end{tabular}
\end{table}

To further demonstrate some of their uses, we explicitly calculate the primitive cohomologies for some examples of K\"ahler manifolds with boundary.  These examples show clearly that primitive cohomologies are very different from the standard de Rham cohomologies on manifolds with boundary.  Interestingly, we find that even on a simple K\"ahler manifold that is the product of a three-ball times a three-torus, $B^3 \times T^3$, two different K\"ahler structures can lead to different primitive cohomologies.  In Section 7, we conclude with a discussion connecting our relative primitive cohomology with the differential topological notion of a relative cohomology.  This allows us to propose a relative primitive cohomology with respect to any submanifold, including lagrangians, embedded within a symplectic manifold.

\
\noindent{\it Acknowledgements.~} 
We would like to thank T.-J. Li, Y. S. Poon, M. Schechter, C.-J. Tsai, J. Wang, and S.-T. Yau for helpful comments and discussions.   Additionally, we are grateful to S.-Y. Li, Z. Lu, C.-L. Terng, Q. Zhang, and especially P. Li for their interest and input in this work.  L.-S. Tseng would like to acknowledge the support of a Simons Collaboration Grant for Mathematicians.

\section{Preliminaries }
In this section, we will gather some basic definitions and properties of differential forms and operators in symplectic geometry.  Further background details and proofs of the lemmas and propositions stated here without elaboration can be found in \cite{TY1, TY2}.
 
\subsection{Primitive structures on symplectic manifolds}\label{defsection}
Given a symplectic manifold $(M^{2n}, \omega)$,  let $\Omega^k$ denote the space of smooth $k$-forms on $M$. In local coordinates, we write the symplectic form as $\omega= \frac{1}{2}\sum \omega_{ij}~ dx^{i}\wedge dx^{j}$. The Lefschetz operator $L$ and its dual operator $\Lambda$ acting on a differential $k$-form $\eta\in \Omega^k$ are then defined by
\begin{align*}
L: \Omega^k \rightarrow \Omega^{k+2}, \quad L(\eta)&=\omega \wedge \eta\,,\\
\Lambda: \Omega^k \rightarrow \Omega^{k-2},\quad \Lambda(\eta)&=\frac{1}{2}(\omega^{-1})^{ij}\,\iota_{\frac{\partial}{\partial{x^i}}}\iota_{\frac{\partial}{\partial{x^j}}}\eta\,,
\end{align*} where $\iota$ denotes the interior product, and $\omega^{-1}$ is the inverse matrix of $\omega$. Define also the degree counting operator 
\begin{equation}\label{defH}
H=\underset{k}{\sum}(n-k) {\prod}^k
\end{equation} where $\prod^k: \Omega^{\ast} \rightarrow \Omega^k$ is the projection operator onto forms of degree $k$. The three operators $(L, \Lambda, H)$ together provide a representation of an $sl(2)$ algebra acting on $\Omega^{\ast}$:
\begin{equation*}
[\Lambda, L]=H, \quad \, [H, \Lambda]=2\Lambda, \quad \, [H,L]=-2L.
\end{equation*}

This $sl(2)$ representation leads to a Lefschetz decomposition of forms in terms of irreducible finite-dimensional $sl(2)$ modules. The highest weight states of these irreducible $sl(2)$ modules are the primitive forms, whose space we denote by $P^{\ast}$.
\begin{definition}
A $k$-form $\beta$ is called primitive (i.e. $\beta \in P^k$) if $\Lambda\, \beta=0\,$. This is equivalent to the condition $L^{n-k+1}\beta=0\,$. 
\end{definition} 
As implied by the definition, the degree of the primitive form is constrained to be $k \leq n$. Note also that $P^k=\Omega^k$ when $k=0, 1\,$.  In terms of primitive forms, the Lefschetz decomposition of a form $\eta \in \Omega^k$ can be expressed as
\[
\eta =\underset{r \geq \max(k-n, 0)}{\sum} \om^r \w\beta_{k-2r}.
\] Here, each $\beta_{k-2r}\in P^{k-2r}$ is uniquely determined by $\eta\,$.  
We see that each term of this decomposition can be labeled by a pair $(r,s)$  corresponding to the space 
 \[
 \mathcal{L}^{r,s}=\left\{\eta\in \Omega^{2r+s}~|~ \eta=\om^r\w\beta_s \,\text{with}\, \beta_s \in P^s\right\}.
 \] where $0\leq s \leq (n-r)$.   
Two other maps will also be used in this paper:
\begin{align}
&\Pi: \Omega^k \rightarrow P^k, \quad \text{the projection map for~}k\leq n\,;\text{and} ~ \label{projmap}\\
&\ast_r: \mathcal{L}^{r,s} \rightarrow \mathcal{L}^{n-r-s,s}, \quad \om^r\w\beta_s \rightarrow \om^{n-r-s}\w \beta_s\,.\label{starmap}
\end{align}
The first map is always surjective and the second one is always bijective. The triple $\{L, \Pi, \ast_r\}$ played an essential role in \cite{TTY} for building a long exact sequence relating primitive cohomologies with Lefschetz maps.

\subsection{Differential operators $\partial_{+}, \partial_{-}$, and $d^{\Lambda}$}
 We consider the action of the exterior derivative operator $d$ on $\mathcal{L}^{r,s}$ \cite{TY2}. 
\begin{prop}
$d$ acting on $\mathcal{L}^{r,s}$ leads to at most two terms:
\begin{equation*}
d: \mathcal{L}^{r,s} \rightarrow \mathcal{L}^{r,s+1}\oplus \mathcal{L}^{r+1,s-1}
\end{equation*} with
\begin{align*}
d(\om^r\w\beta_s)&=\om^r\w (d\,\beta_s)=\om^r\w\beta_{s+1}+\om^{r+1}\w\beta_{s-1}\,.
\end{align*} 
 \end{prop} 
This result is a consequence of the closedness of the symplectic form $\omega$  and the following formulas:
\begin{itemize}
  \item If $s <n$, $d\,\beta_s=\beta_{s+1}+\om\w\beta_{s-1}$;
  \item If $s=n$, $d\,\beta_n=\om\w\beta_{n-1}$.
\end{itemize} By this proposition, Tseng-Yau \cite{TY2} defined the decomposition of $d$ into two linear differential operators $(\partial_{+}, \partial_{-})$. 
\begin{definition}
On a symplectic manifold $(M, \omega^{2n})$, we define the first order differential operators $\partial_{+}, \partial_{-}$ by the property:
\begin{align*}
&\partial_{+}:\mathcal{L}^{r,s} \rightarrow \mathcal{L}^{r,s+1}, \qquad\partial_{+}(\om^{r}\w\beta_{s})=\om^{r}\w\beta_{s+1},\\
&\partial_{-}:  \mathcal{L}^{r,s} \rightarrow \mathcal{L}^{r,s-1}, \qquad \partial_{-}(\om^{r}\w\beta_{s})=\om^{r}\w\beta_{s-1},
\end{align*}  such that 
$$d=\partial_{+}+\om \w \partial_{-}~.$$ 
Here, $\beta_s, \beta_{s+1}, \beta_{s-1} \in P^{\ast}$ and $d\,\beta_s=\beta_{s+1}+\om\w\beta_{s-1}$.
\end{definition} 
When acting on primitive forms, $\partial_{+}$ and $\partial_{-}$ can be equivalently written as follows: 
\begin{lemma}\label{pdpdm}
Acting on primitive differential forms, the operators $(\partial_{+}, \partial_{-})$ have the following expressions:
\begin{align*}
\partial_{+}&=d-L\,H^{-1}\Lambda\, d,\\
\partial_{-}&=H^{-1}\Lambda\, d.
\end{align*}
\end{lemma}
In fact, on $P^*$,  
\begin{align}\label{dpproj}
\partial_{+}=\Pi~ d \,.
\end{align} Moreover, the $\dpp$ and $\dpm$ operators have the following properties on general forms:
\begin{prop}
On $(M^{2n}, \omega)$, the symplectic differential operators $(\partial_{+}, \partial_{-})$ satisfy:
\begin{itemize}
\item $\partial_{+}^2=\partial_{-}^2=0\,$;
\item $L\,\partial_{+}\partial_{-}=-L\,\partial_{-}\partial_{+}\,$;
\item $[L, \partial_{+}]=[L, L\,\partial_{-}]=0\,$. 
\end{itemize}
\end{prop}
Besides $d, \dpp,$ and $\dpm$, there is one other first-order differential operator, $d^{\Lambda}: \Omega^{k} \rightarrow \Omega^{k-1}\,$, that will be of interest in this paper.  It can written as 
\begin{align}\label{dlambda}
d^{\Lambda}=d\,\Lambda - \Lambda\,d\,. 
\end{align}
and is sometimes called the symplectic adjoint operator since it lowers the degree of a form.  Let us point out that in terms of $d$ and $d^{\Lambda}$, the pair $(\partial_{+}, \partial_{-})$ can be expressed as follows.
\begin{lemma}\label{ex}
On a symplectic manifold $(M, \omega)$, $\partial_{+}$ and $\partial_{-} $ can be expressed as 
\begin{align*}
\partial_{+}&=\frac{1}{H+2R+1}\left[ (H+R+1)d+L\,d^{\Lambda}\right],\\
\partial_{-}&=\frac{1}{(H+2R+1)(H+R)}\left[\Lambda d-(H+R)d^{\Lambda}\right].
\end{align*}
where  the operator $R: \mathcal{L}^{r,s} \to  \mathcal{L}^{r,s}$ is the multiplication
$$R\,(L^r\beta_s)=r\,(L^r\beta_s)\,$$
\end{lemma}
In particular, acting on primitive $(r=0)$ forms, $P^*$, the expression for $\dpm$ reduces to
\begin{align}\label{dmdlam}
\partial_{-}= -\dfrac{1}{H}\, d^{\Lambda} = \dfrac{1}{H} \, \Lambda\, d 
\end{align}
which agrees with Lemma \ref{pdpdm}.

\subsection{Conjugate relations}\label{conjsec}
Let $(\omega, J, g)$ be a compatible triple on the symplectic manifold $(M^{2n}, \omega)$ with $J$ being an almost complex structure and $g$ a compatible Riemannian metric on $M$. With respect to the almost complex structure $J$, there is the standard $(p,q)$ decomposition $\Omega^k=\underset{p+q=k}{\oplus}\Omega^{p,q}$. Let us define the operator 
\begin{align}\label{conjop}
\mathcal{J}=\underset{p,q}{\sum}(\sqrt{-1})^{p-q}{\prod}^{p,q}~,
\end{align} 
where ${\prod}^{p,q}$ denotes the projection of a $k$-form onto its $(p,q)$ component.  Notice that $\mathcal{J}^2=(-1)^k$ acting on $k$-forms and also that  $\mathcal{J}$ commutes with both $L$ and $\Lambda$ since
the symplectic form $\omega$ is a $(1,1)$-form with respect to the almost complex structure $J$. Moreover, the operator $\mathcal{J}$ defines the following conjugate relations (\cite{TY1,TY2}) between differential operators:
\begin{lemma}\label{conjugate}
For a compatible triple $(\omega, J, g)$ on a symplectic manifold, let $d^{\ast}, d^{\Lambda \ast}, \partial_{+}^{\ast}$ and $\partial_{-}^{\ast}$ be the adjoint operators of the corresponding differential operators, respectively. Then there are the following conjugate relations:
\begin{itemize}
\item $d^{\Lambda}=\mathcal{J}^{-1}d^{\ast}\,\mathcal{J}$ and $d^{\Lambda \ast}=\mathcal{J}^{-1}d\,\mathcal{J}$;
\item $\mathcal{J}\partial_{+}\,\mathcal{J}^{-1}=\partial_{-}^{\ast}(H+R)$ and 
$\mathcal{J}\partial_{+}^{\ast}\,\mathcal{J}^{-1}=(H+R)\,\partial_{-}\,$.
\end{itemize}
\end{lemma} 
This lemma, together with Lemma \ref{ex}, implies the following expressions for $(\partial^{\ast}_{+}, \partial^{\ast}_{-})$.

\begin{lemma}\label{c2}
On a symplectic manifold $(M^{2n}, \omega)$ with a compatible Riemannian metric $g$, the adjoints $(\partial^{\ast}_{+}, \partial^{\ast}_{-})$ have the form
\begin{align*}
\partial^{\ast}_{+}&=[d^{\ast}(H+R+1)+d^{\Lambda\ast}\Lambda](H+2R+1)^{-1},\\
\partial^{\ast}_{-}&=[d^{\ast}(H+R+1)^{-1}L-d^{\Lambda\ast}](H+2R+1)^{-1}.
\end{align*}
\end{lemma}

\begin{cor}\label{c1}
On $P^k$, the adjoints $(\partial^{\ast}_{+}, \partial^{\ast}_{-})$ have the form
\begin{align*}
\partial^{\ast}_{+}&=d^{\ast},\\
\partial^{\ast}_{-}&=[d^{\ast}, \,LH^{-1}]=(n-k)^{-1}d^{\ast}L-(n-k+1)^{-1}Ld^{\ast}.
\end{align*} 
\end{cor}

\begin{remark}
Throughout the paper, we will always assume that the Riemannian metric $g$ used to define the adjoints $(\partial^{\ast}_{+}, \partial^{\ast}_{-}, d^{\Lambda \ast})$ is compatible with the symplectic form $\om$.
\end{remark}

\subsection{Symplectic elliptic complex and Laplacians}\label{symplap}
For symplectic manifolds, there is an elliptic complex on the space of primitive forms $P^*$  \cite{TY2} (see also \cite{Smith, Eastwood, E1}):
$$\begin{CD}
0@>>>P^0@>\partial_{+}>>P^1@>\partial_{+}>>\cdots @>\partial_{+}>>P^{n-1}@>\partial_{+}>>P^n\\
    @. @. @. @.   @.                                                                                                            @VV{\partial_{+}\partial_{-}}V\\
0@<\partial_{-}<<P^1@<\partial_{-}<<P^2@<\partial_{-}<<\cdots @<\partial_{-}<<P^{n-1}@<\partial_{-}<<P^n
\end{CD}$$
Of note is the presence of the second-order differential operator $\dpp\dpm$ that acts on the middle degree primitive space, $P^n$, in the middle of the complex.  Though this elliptic complex involves differential operators of differing orders, we can still easily associate to each element of the complex an elliptic Laplacian operator.  (See, for example, \cite{AB}.) Hence, we define the following symplectic Laplacians as associated with this elliptic complex:
\begin{align} 
\Delta_{+}&=\partial_{+}\partial_{+}^{\ast}+\partial_{+}^{\ast}\partial_{+}, \, \text{on $P^k$, for $k<n$}\,,\label{sympLp}\\
\Delta_{-}&=\partial_{-}\partial_{-}^{\ast}+\partial_{-}^{\ast}\partial_{-},\, \text{on $P^k$, for $k<n$}\,,\label{sympLm}\\
\Delta_{++}&=(\partial_{+}\partial_{-})^{\ast}(\partial_{+}\partial_{-})+(\partial_{+}\partial_{+}^{\ast})^2, \, \text{on $P^n$}\,,\label{sympLpp}\\
\Delta_{--}&=(\partial_{+}\partial_{-})(\partial_{+}\partial_{-})^{\ast}+(\partial_{-}^{\ast}\partial_{-})^2, \, \text{on $P^n$}\,.\label{sympLmm}
\end{align} 
The principal symbol of these symplectic Laplacian operators $\{-\Delta_{+}, -\Delta_{-}, \Delta_{++}, \Delta_{--}\}$ can be straightforwardly calculated and shown to be positive.  (See for example the calculations in Appendix \ref{AppB}.)

\section{Symplectic boundary conditions}
In this section, we present several intrinsically symplectic boundary conditions for differential forms on compact symplectic manifolds with smooth boundary.   We will briefly review first the standard Dirichlet and Neumann boundary conditions for differential forms on Riemannian manifolds.  Again, let $(M^{2n}, \om)$ be a symplectic manifold with boundary $\partial M$ and $(\om, J, g)$ a compatible triple on it.  We will denote throughout any boundary defining function by $\rho$ (i.e. $\rho=0$ on $\partial M$), the associated induced cotangent 1-form by $\drho$,  and the inward dual normal vector field on the boundary by $\vec{n}$ which satisfies $d\rho = g(\vec{n}, \cdot)$ on $\partial M$.  Furthermore, for any differential operator $\D$, we shall use the notation $\sigma_{\!\D}$ to denote its principal symbol.

\subsection{Dirichlet, Neumann and $\mathcal{J}$-conjugate boundary conditions on forms}
We first recall the standard Dirichlet and Neumann boundary conditions: 
  \begin{definition}
We say a differential $k$-form $\eta$ satisfies
\begin{itemize}
\item the Dirichlet $(D)$ boundary condition, i.e. $\eta \in D$, if $\sigma_d(\drho)\,\eta~|_{\partial M}=0\,$; 
\item the Neumann $(N)$ boundary condition, i.e. $\eta \in N$, if $\sigma_{d^{\ast}}(\drho)\,\eta~|_{\partial M}=0\,$.
\end{itemize}
\end{definition}
We note that the Dirichlet condition for forms is equivalent to the condition that $d\rho \w \eta=0$ on $\pa M$; that is, a form without a component in the normal direction would need to vanish on the boundary.  (In the special case where $\eta$ is a function, i.e. a 0-form, the above Dirichlet condition is equivalent to $\eta$ vanishing identically on the boundary.)  In contrast, the Neumann condition corresponds to $\iota_{\vec{n}}\, \eta=0$ on $\pa M$; that is, any form with a component in the normal direction must vanish on the boundary. Here again, $\vec{n}$ is the inward normal along the boundary, and $\iota_{\vec{n}}\, \eta$ is the interior product by $\vec{n}$ on the form $\eta$.

For calculations, it is often convenient to express the boundary conditions in terms of differential operators, without any principal symbols as follows.
\begin{remark}\label{diffequiv}(See, for example \cite{T})
For any first-order differential operator $\mathcal{P}$ and boundary defining function $\rho$, 
\begin{align}\label{sdiff}
\sigma_{\!\mathcal{P}}(\drho)\, \eta~|_{\partial M} = \mathcal{P}(\rho \,\eta)~|_{\partial M}\,.
\end{align}
For instance, for the standard Dirichlet boundary condition, $\sigma_d(\drho)\,\eta~|_{\partial M}=0$ is equivalent to the condition $d(\rho \,\eta)~|_{\partial M}=0\,$. 
\end{remark}

It is also useful to point out that both the Dirichlet and Neumann boundary conditions arise naturally when integrating by parts the exterior derivative operator, $d$. These boundary conditions can be inferred from the Green's formula which we recall here \cite{T}.
\begin{lemma}[Green's formula for first-order differential operators]\label{green}
If $M$ is a smooth, compact manifold with boundary and $\mathcal{P}$ is a first-order differential operator acting on sections of a vector bundle equipped with a metric, then 
\begin{equation}\label{greenf}
(\mathcal{P}\, \phi, \psi)-(\phi,\mathcal{P}^{\ast} \psi)=\int_{\partial M}\langle \sigma_{\mathcal{P}}(\drho)\, \phi, \psi \rangle \,dS
\end{equation} with $\mathcal{P}^{\ast}$ the adjoint operator of $\mathcal{P}$ and $\langle , \rangle$ denoting the metric on the vector bundle and $dS$ the volume form on the boundary.
\end{lemma}

In particular, for the exterior derivative operator, $d$, the lemma implies for  any $\eta,\, \xi \in \Omega^*$  that
\begin{align*}
(d\eta,\xi)-(\eta, d^{\ast}\xi)=\int_{\partial M}\langle \sigma_{d}(\drho)\, \eta, \xi\rangle \,dS =-\int_{\partial M}\langle \eta, \sigma_{d^{\ast}}(\drho)\,\xi\rangle  \, dS~.
\end{align*}

Another noteworthy property of the Dirichlet and Neumann condition is the following lemma (see for example, \cite{G}). 
\begin{lemma}\label{DN}
The Dirichlet  boundary condition is preserved by $d$ and the Neumann boundary condition is preserved by $d^{\ast}$. That is, for any $\eta \in \Omega^k$, we have 
\begin{align*}
\eta &\in D \Longrightarrow d\eta \in D\, , \\
\eta &\in N \Longrightarrow d^{\ast}\eta \in N\,. 
\end{align*}
\end{lemma}

Besides the Dirichlet and Neumann boundary conditions, let us introduce here two other related boundary conditions which will be useful later on.  Using the conjugate relations in Lemma \ref{conjugate}, we define the following:
\begin{definition}
We say a differential form $\eta$ satisfies 
\begin{itemize}
\item the ${J}$-Dirichlet $(JD)$ boundary condition, i.e. $\eta \in JD$, if $\sigma_{d^{\Lambda \ast}}(\drho)\,\eta~|_{\partial M}\!=\!0\,$;
\item the ${J}$-Neumann $(JN)$ boundary condition, i.e. $\eta \in JN$, if $\sigma_{d^{\Lambda }}(\drho)\,\eta~|_{\partial M}\!=\!0\,$.
\end{itemize}
\end{definition}
The relation between ($JD$, $JN$) and ($D$, $N$) boundary conditions are as follows:
\begin{lemma}With respect to a compatible triple $(\om, J, g)$ on a symplectic manifold $M^{2n}$, any $\eta \in \Omega^k$ satisfies the following:
\begin{align*}
&\eta \in JD \, \Longleftrightarrow\, \mathcal{J}\eta \in D\,,\\
&\eta \in JN \, \Longleftrightarrow\, \mathcal{J}\eta \in N\,.
\end{align*}
\end{lemma}
\begin{proof}
Using the relations $d^{\Lambda \ast}=\mathcal{J}^{-1} d\, \mathcal{J}$ and $d^{\Lambda}=\mathcal{J}^{-1} d^{\ast} \mathcal{J}$ in Lemma \ref{conjugate} and expressing the boundary conditions in terms of differential operators as in \eqref{sdiff},  we have 
\begin{align*}
\eta \in JD&\Leftrightarrow  d^{\Lambda \ast}(\rho\, \eta)~|_{\partial M}\!=\!0 \Leftrightarrow\mathcal{J}^{-1}\!d\mathcal{J}(\rho\, \eta)~|_{\partial M}\!=\!0  
\Leftrightarrow  d(\rho\,\mathcal{J} \eta)~|_{\partial M}\!=\!0  \Leftrightarrow\mathcal{J}\eta \in D\,,\\
\eta \in JN&\Leftrightarrow d^{\Lambda}(\rho\, \eta)~|_{\partial M}\!=\!0 \Leftrightarrow\mathcal{J}^{-1}\!d^{\ast}\!\mathcal{J}(\rho\, \eta)~|_{\partial M}\!=\!0 
\Leftrightarrow  d^{\ast}(\rho\,\mathcal{J} \eta)~|_{\partial M}\!=\!0 \Leftrightarrow\mathcal{J}\eta \in N.
\end{align*}
\end{proof}
Applying Lemma \ref{DN}, we also obtain the following:
\begin{cor}\label{pre1}
The $JD$  boundary condition is preserved by $d^{\Lambda\ast}$ and the $JN$ boundary condition is preserved by $d^{\Lambda}$. That is, for any $\eta \in \Omega^k$, 
\begin{align*}
\eta &\in JD \Longrightarrow d^{\Lambda \ast}\eta \in JD\,, \\
\eta &\in JN \Longrightarrow d^{\Lambda}\eta \in JN\,. 
\end{align*}
\end{cor}
\begin{proof}
Since $\eta \in JD$ is equivalent to $\mathcal{J}\eta \in D$, it follows that $d\mathcal{J}\eta \in D$. Therefore, $\mathcal{J}^{-1}d\mathcal{J}\eta \in JD$, that is, $d^{\Lambda \ast}\eta \in JD$. By similar arguments, $d^{\Lambda}$ preserves the $JN$ boundary condition. 
\end{proof}

We can give an interpretation for the $JD$ and $JN$ boundary conditions as follows. As mentioned, the $D$ and $N$ boundary conditions are defined with respect to the outward normal vector field $\vec{n}$ along the boundary. For $JD$ and $JN$ boundary conditions, they are instead defined with respect to the ${J} \vec{n}$ vector field.   More specifically, around a point $x \in \partial M$, we can choose a local Darboux basis of one-forms, $\{w_j\}$, such that $w_1=d\rho$ and $\omega=\underset{i}{\sum}w_{2i-1}\wedge w_{2i}\,$.     Let us further choose an almost complex structure $J$ such that $\mathcal{J}w_{2i-1} = -w_{2i}$ and  $\mathcal{J}w_{2i} = w_{2i-1}$ for $i=1, 2, \ldots, n$.  We denote the dual basis of tangent vectors by $\{e_j\}$.  The boundary conditions then correspond to the following:
\begin{align*}
\eta \in JD &\Longrightarrow w_2\w \eta~ |_{\partial M}=0,\\
\eta \in JN & \Longrightarrow \iota_{e_2}\eta ~|_{\partial M}=0.
\end{align*}
Moreover, if the boundary is of contact type, then $\partial M$, being a contact space, has a well-known symplectization that can be mapped to the collar neighborhood of $\partial M$.  In this case, ${J} \vec{n}$ can also be identified with the Reeb vector field on the contact boundary.

\subsection{Symplectic boundary conditions on forms}
\subsubsection{Boundary conditions associated with $\dpp$ and $\dpm$ operators}
With two natural linear first-order operators $\dpp$ and $\dpm$ on symplectic manifolds, we are motivated to define the analogous Dirichlet and Neumann boundary conditions with respect to these operators. 
\begin{definition}\label{sbc1def}
We say a differential $k$-form $\eta\in\Omega^k$ satisfies
\begin{itemize}
\item the $\partial_{+}$-Dirichlet $(D_{+})$ boundary condition, i.e.  $\eta \in D_{+}$, if $\sigma_{\partial_{+}}(\drho)\, \eta~|_{\partial M}\!=\!0\,$; 
\item the $\partial_{-}$-Dirichlet $(D_{-})$ boundary condition, i.e. $\eta \in  D_{-}$, if $\sigma_{\partial_{-}}(\drho)\,\eta~|_{\partial M}\!=\!0\,$;
\item the $\partial_{+}$-Neumann $(N_{+})$ boundary condition, i.e. $\!\eta \in N_{+}$, if $\sigma_{\partial_{+}^{\ast}}(\drho)\,\eta~|_{\partial M}\!=\!0\,$; 
\item the $\partial_{-}$-Neumann $(N_{-})$ boundary condition, i.e. $\!\eta \in  N_{-}$, if $\sigma_{\partial_{-}^{\ast}}(\drho)\, \eta~|_{\partial M}\!=\!0\,$. 
\end{itemize}
\end{definition} 

Just as for $D$ and $N$ boundary conditions, it follows from Lemma \ref{green} for $(\partial_{+}, \partial_{-})$ that:
\begin{align*}
(\partial_{+}\eta, \xi)-(\eta, \partial^{\ast}_{+}\xi)&=\int_{\partial M}\langle \sigma_{\partial_{+}}(\drho)\, \eta, \xi\rangle \, dS=-\int_{\partial M}\langle \eta, \sigma_{\partial_{+}^{\ast}}(\drho)\, \xi\rangle \, dS,\\
(\partial_{-}\eta, \xi)-(\eta, \partial_{-}^{\ast}\xi)&=\int_{\partial M}\langle \sigma_{\partial_{-}}(\drho)\, \eta, \xi \rangle\, dS=-\int_{\partial M}\langle \eta, \sigma_{\partial_{-}^{\ast}}(\drho)\, \xi\rangle\, dS.
\end{align*}
These formulas above imply that the $\{D_{+}, N_{+}\}$ and the $\{D_{-}, N_{-}\}$ boundary conditions are natural from the perspective of integration by parts.  

\subsubsection{Boundary condition associated with the $\dpp\dpm$ operator}
As above, we can also introduce Dirichlet and Neumann type boundary conditions for the $\dpp\dpm$ operator.
\begin{definition} We say a differential form $\eta$ satisfies, 
\begin{itemize}
\item the $\dpp\dpm$-Dirichlet boundary condition $(D_{+-})$, i.e. $\eta \in D_{+-}\,$,
\\ if $ \sigma_{\partial_{+}\partial_{-}}(\drho)\,\eta~ |_{\partial M}=0$\,;
\item the $\dpp\dpm$-Neumann boundary condition $(N_{+-})$, i.e.  $\eta \in N_{+-}\,$,\\ if  $\sigma_{(\partial_{+}\partial_{-})^{\ast}}(\drho)\,\eta~ |_{\partial M}=0\,$.
\end{itemize}
\end{definition}
\begin{remark}
Similar to the first-order case in Remark \ref{diffequiv}, the above second-order boundary conditions can be equivalently expressed differentially as follows:  
\begin{align*}
\eta \in D_{+-}&\,\Longleftrightarrow \,\partial_{+}\partial_{-}(\rho^2\eta)~|_{\partial M}=0\,,\\
\eta \in N_{+-}&\,\Longleftrightarrow \,\partial_{-}^{\ast}\partial_{+}^{\ast}(\rho^2\eta)~|_{\partial M}=0\,.
\end{align*}
\end{remark}
The $D_{+-}$ and $N_{+-}$ boundary conditions, however, by themselves are not sufficient to ensure that $(\dpp\dpm \eta, \xi) = (\eta, (\dpp\dpm)^{\ast} \xi)$.  This can be seen from the following lemma:
\begin{lemma}[Green's formula for second-order differential operators]\label{green2}
If $M$ is a smooth, compact manifold with boundary and $\mathcal{P}$ is a second-order differential operator acting on sections of the vector bundle equipped with a metric, then 
\begin{align}\label{greens}
(\mathcal{P}\, &\phi, \psi)-(\phi,\mathcal{P}^{\ast} \psi)  \\
&= -\int_{\partial M}\langle \left\{2\mathcal{P}(\rho\,\phi) - \frac{1}{2}\mathcal{L}_{\vec{n}}\big[\mathcal{P}(\rho^2 \phi)\big]\right\}, \psi\rangle\, dS+\int_{\partial M}\langle \frac{1}{2}\mathcal{P}(\rho^2 \phi) , \mathcal{L}^{\ast}_{\vec{n}}(\psi)\rangle \, dS\nonumber
\end{align} 
with $\mathcal{P}^{\ast}$ the adjoint operator of $\mathcal{P}$, $dS$ the volume form on $\partial M$, and $\langle , \rangle$ denoting a metric on the vector bundle.
\end{lemma}
\begin{proof}
Let $\dim M = m$.  Using a partition of unity, we may assume that $\phi$ and $\psi$ are supported within a coordinate patch $U$ in $M$.  Hence, we only need to consider the case when $U$ intersects with the boundary $\partial M$. So suppose $U$ is in $\mathbb{R}^m_{+}$ and the coordinates are such that $\frac{\partial }{\partial x_m}$ is the unit inward normal at $\partial M$. In $U$, the second-order operator $\mathcal{P}$ has the form
\begin{align}\label{2ndo}
\mathcal{P}=\underset{i \leq j}{\sum} a_{ij}(x)\frac{\partial^2 }{\partial x_i\partial x_j} +\underset{i}{\sum} b_{i}(x)\frac{\partial }{\partial x_i}+c(x)\, .
\end{align}
where here $i, j = 1, 2, \ldots, m$. Then,
\begin{align*}
(\mathcal{P}\phi, \psi)_{U}&=\int_{U}\bigg[\underset{i\leq j}{\sum}\langle a_{ij}\frac{\partial^2 \phi}{\partial x_i\partial x_j}, \psi\rangle+\underset{i}{\sum}\langle b_{i}\frac{\partial \phi}{\partial x_i}, \psi\rangle+\langle c\, \phi, \psi\rangle\bigg]\sqrt{g}dx.
\end{align*}
Integrating by parts, there are boundary integral contributions coming from the terms involving $\frac{\partial}{\partial x_m}$, and we obtain
\begin{align}
(\mathcal{P}\phi, \psi)_{U}&
=(\phi, \mathcal{P}^{\ast}\psi)_{U} - \int_{U\cap\,\mathbb{R}^{m-1}}\langle \underset{i\leq m}{\sum}a_{im}\frac{\partial \phi}{\partial x_m}+b_{m}\phi -\frac{\partial a_{mm}}{\partial x_m}\phi ,\psi\rangle\sqrt{g(x^{\prime},0)}dx^{\prime}\nonumber\\
&+\int_{U\cap\,\mathbb{R}^{m-1}} \left\{\dfrac{\partial}{\partial x_m} \left[ \langle a_{mm}\phi,\psi\rangle \sqrt{g}\,\right]  - \langle \,\dfrac{\partial(a_{mm} \phi)}{\partial x_m}, \psi \,\rangle \sqrt{g}\, \right\} dx'  \label{d2p1}
\end{align} 
where $dx^{\prime}=dx_1\cdots dx_{m-1}$ and $\sqrt{g(x^{\prime},0)}dx^{\prime}$ is the volume element on $\partial M$.  Now, we can write
\begin{align}
\dfrac{\partial}{\partial x_m}& \left[ \langle a_{mm}\phi,\psi\rangle \sqrt{g}\,\right]  - \langle \,\dfrac{\partial(a_{mm} \phi)}{\partial x_m}, \psi \,\rangle \sqrt{g}= \langle a_{mm} \phi, \mathcal{L}^{\ast}_{\vec{n}}(\psi)\rangle\,, \label{d2pp2} \\
\mathcal{P}(\rho\,\phi) &= P(x_m \, \phi) =  \underset{i\leq m}{\sum} a_{im} \frac{\partial \phi}{\partial x_i} + a_{mm} \frac{\partial \phi}{\partial x_m} +  b_m \phi  + \mathcal{O}(x_m)\label{d2p3}\\
\frac{1}{2}\mathcal{P}(\rho^2 \phi)&=\frac{1}{2}\mathcal{P}(x_m^2 \phi)=a_{mm}\phi +  x_m \left[\underset{i\leq m}{\sum} a_{im} \dfrac{\partial \phi}{\partial x_i} + a_{mm} \dfrac{\partial \phi}{\partial x_m} +  b_m \phi  \right] + \mathcal{O}(x^2_m)\label{d2p4}
\end{align}
Using \eqref{d2p3}-\eqref{d2p4}, we find along the boundary (i.e. $x_m\, =0$) that
\begin{align} \label{d2p5}
\left\{2\mathcal{P}(\rho\, \phi)- \dfrac{1}{2}\mathcal{L}_{\vec{n}}\left[\mathcal{P}(\rho^2 \phi)\right]\right\}_{x_m=\,0}=\underset{i\leq m}{\sum}a_{im}\frac{\partial \phi}{\partial x_i}+b_m \phi-\frac{\partial a_{mm}}{\partial x_m}\phi\,.
\end{align}
The statement then follows substituting \eqref{d2pp2} and \eqref{d2p4}-\eqref{d2p5} into \eqref{d2p1}.
\end{proof}

The above lemma leads us to the following definitions:
\begin{definition}\label{sbc2def}
We say a differential form $\eta$ satisfies 
\begin{itemize}
\item the $D_{++}$ boundary condition if 
\begin{enumerate}
\item  $\eta\in D_{+-}\,$, that is, $\partial_{+}\partial_{-}(\rho^2 \eta)\,|_{\partial M}=0\,$, and
\item $\left\{2\dpp\dpm(\rho\,\eta) - \frac{1}{2}\mathcal{L}_{\vec{n}}\big[\partial_{+}\partial_{-}(\rho^2 \eta)\big] \right\}|_{\partial M}=0\,;\medskip$
\end{enumerate}
\item the $N_{--}$ boundary condition if 
\begin{enumerate}
\item  $\eta\in N_{+-}\,$, that is, $\partial_{-}^{\ast}\partial_{+}^{\ast}(\rho^2 \eta)\,|_{\partial M}=0\,$, and
\item $\left\{2\partial_{-}^{\ast}\partial_{+}^{\ast}(\rho\,\eta) - \frac{1}{2}\mathcal{L}_{\vec{n}}\big[\partial_{-}^{\ast}\partial_{+}^{\ast}(\rho^2 \eta)\big] \right\}|_{\partial M}=0\,.$\\
\end{enumerate}
\end{itemize}
\end{definition}

\begin{remark}\label{sbc2rmk}
The $D_{++}$ and $N_{--}$ boundary conditions can be alternatively defined using the principal symbol.  With the convention that the principal symbol of the second-order operator $\mathcal{P}$ in $\eqref{2ndo}$ is $ \sigma_{\mathcal{P}}(d\rho)\,\phi ~|_{\partial M} = a_{mm}\phi$, we have that
\begin{align*}
\left\{\mathcal{P}(\rho\, \phi) - \mathcal{L}_{\vec{n}}\left[\sigma_{\mathcal{P}}(d\rho)\,\phi\right] \right\}|_{\partial M}  =\left\{2\mathcal{P}(\rho \,\phi)- \frac{1}{2}\mathcal{L}_{\vec{n}}\left[\mathcal{P}(\rho^2\phi)\right]\right\}_{\partial M} \,.
\end{align*}
Hence, we can express the boundary conditions in the form of
\begin{align*}
\sigma_{\mathcal{P}}(d\rho)\,\eta ~|_{\partial M}&=0\,,\\
\left\{\mathcal{P}(\rho\, \eta)-\mathcal{L}_{\vec{n}}\left[\sigma_{\mathcal{P}}(d\rho)\,\eta\right]\right\} |_{\partial M}&=0\,.
\end{align*}
setting $\mathcal{P} = \dpp\dpm$ for the $D_{++}$ boundary condition and $\mathcal{P}=(\dpp\dpm)^\ast$ for the $N_{--}$ condition.
\end{remark}

Lemma \ref{green2} immediately implies the following results.

\begin{cor}\label{D++equ}
For a differential $k$-form $\eta$,  $\eta \in D_{++}$ is equivalent to the condition $$(\partial_{+}\partial_{-}\eta, \xi)=(\eta, (\partial_{+}\partial_{-})^{\ast}\xi)$$ for any $\xi \in \Omega^k$.
Similarly, $\eta \in N_{--}$ is equivalent to the condition 
$$((\partial_{+}\partial_{-})^{\ast}\eta, \xi)=(\eta, \partial_{+}\partial_{-}\xi)$$ 
again for any $\xi\in \Omega^k$.

\end{cor}

Clearly, all six of the above boundary conditions -- $\{D_+, D_-, N_+, N_-\}$ in Definition \ref{sbc1def} and $\{D_{++}, N_{--}\}$ in Definition \ref{sbc2def} -- depend on the symplectic structure.  Being so, we will refer to them as {\it symplectic boundary conditions}.  Certainly, these boundary conditions are defined for general differential forms.  To get a better sense of these symplectic boundary conditions, we will focus our discussion in the following to primitive forms and explore the properties of these boundary conditions on them.

\subsubsection{Local description of boundary conditions on primitive forms}\label{locD}

To make clear the differences and infer the properties of the various new boundary conditions presented above, we provide here a local description of the boundary conditions on primitive forms.  For simplicity, we shall describe them in terms of a local Darboux basis $\{w_j=dx_j\}$ of $\Omega^1$ where $w_1=d\rho$ and $\omega=\underset{i}{\sum}w_{2i-1}\wedge w_{2i}$.  As before, we denote the dual basis of tangent vectors by $\{e_j\}$ and choose as the almost complex structure $J$ the standard one where $\mathcal{J}w_{2i-1} = -w_{2i}$ and  $\mathcal{J}w_{2i} = w_{2i-1}$ for $i=1, 2, \ldots, n$.  
In such a basis, any primitive differential $k$-form, $\beta\in P^k$, can be decomposed into four distinct terms \cite{TY2}:
\begin{align}\label{ldecomp}
\beta= w_1 \w \beta^1 + w_2 \w \beta^2 + \Theta_{12}\w\beta^3 + \beta^4
\end{align}
where $\beta^1\!, \,\beta^2\!\in P^{k-1},$ $\,\beta^3\!\in P^{k-2}$, and $\beta^4 \!\in P^k$ are primitive forms that do not contain any components of $w_1$ or $w_2$, and 
$$\Theta_{12}=  w_1 \w w_2 - \dfrac{1}{H+1} \sum_{i=2}^{n} w_{2i-1} \w w_{2i}~,$$
where $H$ is the degree counting operator defined in \eqref{defH}.

Using this decomposition, we can see see explicitly how the different boundary conditions constrain a primitive form $\beta$ along $\partial M$.   To start, consider first the $D$ condition which corresponds to $d\rho \w \beta~|_{\partial M}=w_1 \w \beta~|_{\partial M} =0$.  With $\beta$ expressed in the decomposed form of  \eqref{ldecomp}, the $D$ condition implies that $\,\bb=\bc=\bd=0$ on the boundary, and hence, locally $\beta\,|_{\partial M} = w_1 \w \ba$.  Now, let us consider the symplectic $D_+$ condition.   Recall from \eqref{dpproj} that  $\dpp = \Pi \, d $ when acting on a primitive form.  Thus, the $D_+$ condition corresponds to $\Pi(d\rho \w \beta)~|_{\partial M} =0\,$, which is just the projected form of the $D$ condition.  Applying the decomposition \eqref{ldecomp}, the $D_+$ condition implies only that $\bb=\bd=0$ on the boundary since $\Pi(w_1 \w (\Theta_{12} \w \bc)) = -\Pi[w_1 (1/(H+1)) \w \omega\w \bc] = 0$.  Hence, a primitive form that satisfies the $D_+$ condition takes the form  $\beta~|_{\partial M}= w_1 \w \ba + \Theta_{12} \w \bd$ along the boundary.  Compared to the $D$ condition, we see clearly that $D_+$ is a weaker condition than the $D$ condition.  In Table \ref{Tab5} and Table \ref{Tab6}, we write down the required local form for a general primitive form $\beta$ along $\partial M$ for all the boundary conditions that were discussed above.  Let us point out that for $\beta\in P^n$, the boundary conditions $D_{+}$ and $N_{-}$ are trivial, i.e. they do not impose any conditions on $\beta$.

\begin{table}[t!]
\caption{First-order boundary conditions and their constraints on a primitive form $\beta$ as expressed in the local basis of \eqref{ldecomp} with $w_1 = d\rho\,$.}
\renewcommand{\arraystretch}{1.2}
\begin{tabular}{r | c | l }
  & Condition on $\partial M$ &  Local Form on $\partial M$\\
\hline
$D$ & $w_1 \w \beta =0 $ & $\beta = w_1 \w \beta^1 $\\\hline
$N$ & $\iota_{e_1} \beta = 0$ & $\beta = w_2 \w \beta^2  + \beta^4 $\\\hline
$JD$ & $w_2 \w \beta =0$ & $\beta = w_2 \w \beta^2 $\\\hline
$JN$& $\iota_{e_2}\beta=0$ &$ \beta = w_1 \w \beta^1 + \beta^4$ \\\hline
$D_+$ & $\Pi(w_1 \w \beta) =0 $ & $\beta = w_1 \w \beta^1 + \Theta_{12} \w \beta^3$\\\hline
$N_+$ & $\iota_{e_1} \beta = 0$ & $\beta = w_2 \w \beta^2  + \beta^4$\\\hline
$D_-$ & $\iota_{e_2}\beta=0$   & $\beta = w_1 \w \beta^1 + \beta^4$\\\hline
$N_-$& $\Pi(w_2 \w \beta) =0$&$ \beta = w_2 \w \beta^2 + \Theta_{12} \w \beta^3$ \\\hline
\end{tabular}
\label{Tab5}
\end{table}

\begin{table}[t!]
\caption{Second-order symplectic boundary conditions and their constraints on a primitive form $\beta$ as expressed in the local basis of \eqref{ldecomp} with $w_1=d\rho\,$.  The primed operators $(\partial_{+}', \partial_{-}')$ are defined on the $(2n-2)$-dimension symplectic subspace spanned by $\{e_3, e_4, \ldots, e_{2n}\}$.}
\renewcommand{\arraystretch}{1.3}
\begin{tabular}{r | c }
   &  Conditions on Local Form on $\partial M$\\
\hline
$D_{++}$ & $\beta^2=0  \qquad(D_{+-} ~condition)$\\
&  $\partial_1 \beta^2 - \partial_2 \beta^1 + \dfrac{H+1}{H}\partial_{+}'\beta^3 + (H-1) \partial_{-}'\beta^4 = 0$\\\hline
$N_{--}$ &  $\beta^1=0\qquad(N_{+-}~ condition)$\\
&  $\partial_1 \beta^1 + \partial_2\beta^2 +(H+1) \partial_{-}'^\ast \beta^3 - \partial_{+}'^\ast\beta^4=0$\\\hline
\end{tabular}\label{Tab6}
\end{table}

The derivation for the case of $D_{++}$ and $N_{--}$ boundary conditions requires a bit more calculations.  For instance, for the $D_{++}$ condition, which consists of two conditions as in Definition \ref{sbc2def}, the first condition $D_{+-}$ imposes on $\partial M$  
\begin{align}\label{lobc1}
\sigma_{\dpp\dpm}(d\rho)\,\beta =w_1(\Lambda(w_1 \w \beta))=0\quad  \Longrightarrow \quad \beta^2=0\,.
\end{align}
In the form expressed in Remark \ref{sbc2rmk}, the second condition imposes on $\partial M$
\begin{align}\label{lobc2}
0&=\dpp\dpm(\rho\, \beta)-\mathcal{L}_{\vec{n}}\left[\sigma_{\dpp\dpm}(d\rho)\,\eta\right]\\
&=\dfrac{1}{H+1} w_1 \w \left[\partial_1 \beta^2 - \partial_2 \beta^1 + \dfrac{H+1}{H}\partial_{+}'\beta^3 + (H-1) \partial_{-}'\beta^4\right]  \nonumber\\
&\quad + \dfrac{1}{H+1} w_2 \w \partial_2 \beta^2 - \Theta_{12}\partial_{-}'\beta^2 + \dfrac{1}{H+1} \partial_{+}'\beta^2 \nonumber
\end{align}
where $(\partial_{+}', \partial_{-}')$ refers to the $(\dpp, \dpm)$ operators on the symplectic subspace spanned by $\{e_j\}$ for $j=3, 4, \ldots, 2n\,$.  Since this subspace is within $\partial M$ it is clear that the third line of \eqref{lobc2} vanishes if \eqref{lobc1} is imposed.  This results in the second condition for $D_{++}$ in Table \ref{Tab6}.  For $N_{--}$, one finds  
\begin{align}\label{lobc3}
\sigma_{(\dpp\dpm)^\ast}(d\rho)\,\beta = \iota_{e_1}(\omega\w(\iota_{e_1} \beta))=0\quad  \Longrightarrow \quad \beta^1=0\,.
\end{align}
and for the second differential condition on $\partial M$
\begin{align}\label{lobc4}
0&=(\dpp\dpm)^\ast(\rho\, \beta)-\mathcal{L}_{\vec{n}}\left[\sigma_{(\dpp\dpm)^\ast}(d\rho)\,\beta\right]\\
&=\dfrac{1}{H+1} w_2 \w \left[\partial_1 \beta^1 + \partial_2 \beta^2 + (H+1) \partial_{-}'^\ast \beta^3 - \partial_{+}'^\ast\beta^4\right]  \nonumber\\
&\quad - \dfrac{1}{H+1} w_1 \w \partial_2 \beta^1 - \dfrac{1}{H+1} \Theta_{12}\partial_{+}'^\ast\beta^1 - \dfrac{H}{H+1} \partial_{-}'^\ast\beta^1 \nonumber
\end{align}
where again the adjoint primed operators are defined on the co-dimension two symplectic subspace orthogonal to $\{e_1, e_2\}$.  Since $\beta^1=0$ on the boundary, the last line of \eqref{lobc4} vanishes and this gives the second condition for $N_{--}$ in Table \ref{Tab6}.

From these local characterizations and definitions, we can quickly find a number of relations relating the different boundary conditions.  For instance a primitive form $\beta \in P^k$ that satisfies $D$ automatically satisfies both $D_{+}$ and $D_{-}$, i.e. 
$$\beta \in D ~\Longrightarrow ~\bigg\{~\begin{matrix}\beta \in D_{+} \\ \beta \in D_{-} \end{matrix} ~.$$ 
From Tables \ref{Tab5} and \ref{Tab6},  we also obtain the following relations between the boundary conditions for primitive forms.
\begin{lemma} \label{rel}
With respect to a compatible triple $(\om, J, g)$ on a symplectic manifold $M^{2n}$, there are the following equivalent conditions for a primitive form $\beta \in P^k$,:
\begin{align*}
\beta  \in D_{+} &\Longleftrightarrow ~\mathcal{J}\beta \in N_{-}\,,\qquad\quad\beta  \in N_{+} \Longleftrightarrow ~ \beta \in N\,,\\
\beta  \in D_{-} &\Longleftrightarrow ~\mathcal{J}\beta\in N_{+}\,,\qquad\quad\beta  \in D_{-} \Longleftrightarrow  ~\beta \in JN\,,\\
\beta \in D_{++} &\Longleftrightarrow ~\mathcal{J}\beta \in N_{--}\,.\end{align*}  
\end{lemma}

An important feature of these symplectic boundary conditions is that they can be preserved when acted upon by one of symplectic differential operators: $(\partial_{+}, \partial_{-}, \partial_{+}\partial_{-})$ and their adjoints.  The standard Dirichlet and Neumann boundary conditions do not have these properties.

\begin{lemma}\label{pre}
For $\beta \in P^k$,
\begin{align*}
\beta \in D_{+} &\Longrightarrow \partial_{+}\beta \in D_{+}~,~ \qquad\quad \beta \in N_{+} \Longrightarrow \partial_{+}^{\ast} \beta \in N_{+}\,,\\
\beta \in D_{-} &\Longrightarrow \partial_{-}\beta \in D_{-}~, ~\qquad\quad \beta \in N_{-} \Longrightarrow \partial_{-}^{\ast} \beta \in N_{-}\,.
\end{align*}
\end{lemma}
\begin{proof}
The lemma can be proven by direct computation.  We here instead give a simple, quick proof which makes use of the inner product that comes with a compatible metric on $(M^{2n}, \om)$. 

In order to prove that $\partial_{+}\beta \in D_{+}$ for $\beta \in D_{+}$, it is enough to show that 
\begin{align}\label{316eq}
\int_{\partial M}\langle\partial_{+}(\rho\,\partial_{+}\beta), \alpha\rangle\,dS=0\,,
\end{align}
for any $\alpha \in P^{k+2}$. Now, since $\beta \in D_{+}$, we have $(\partial_{+}\beta, \partial_{+}^{\ast}\alpha)=(\beta, \partial^{\ast}_{+}\partial^{\ast}_{+}\alpha)=0$. On the other hand, 
$$(\partial_{+}\beta, \partial_{+}^{\ast}\alpha)=(\partial_{+}\partial_{+}\beta, \alpha)-\int_{\partial M}\langle\partial_{+}(\rho \partial_{+}\beta), \alpha\rangle\,dS=-\int_{\partial M}\langle\partial_{+}(\rho \partial_{+}\beta), \alpha\rangle\,dS~$$ 
which immediately implies \eqref{316eq} for any $\alpha \in P^{k+2}$.  The other three statements can be proved similarly.
\end{proof}
\begin{lemma}\label{pree}
For $k\leq n\,$, 
\begin{itemize}
\item if $\beta \in P^k_{D_{++}}\,$, then $\dpp\dpm \beta \in P^k_{D_-}\,$;
\item if $\beta \in P^{k-1}_{D_+}$, then $\partial_{+}\beta \in P^k_{D_{++}}\,$.  
\end{itemize} 
\end{lemma}
\begin{proof}
Again, the quickest method of proof is similar to that given for Lemma \ref{pre}.  Let $\beta \in P^k_{D_{++}}$.   To show that $\dpp\dpm \beta \in P^k_{D_-}$, it suffices to prove that 
$$\int_{\partial M}\langle \dpm (\rho\,\dpp\dpm \beta) , \alpha\rangle \,dS=0\,,$$
for any $\alpha \in P^{k-1}$.  Since $\beta \in P^k_{D_{++}}$, it follows from Corollary \ref{D++equ} above that $(\dpp\dpm \beta, \partial_{-}^{\ast}\alpha)=(\beta, \partial_{-}^{\ast}\partial_{+}^{\ast}\partial_{-}^{\ast}\alpha)=0$. On the other hand, we have
\begin{align*}
0=(\dpp\dpm \beta, \partial_{-}^{\ast}\alpha)&=(\partial_{-}\dpp\dpm \beta, \alpha)-\int_{\partial {M}}\langle\partial_{-}(\rho\, \dpp\dpm \beta),\alpha\rangle\,dS\\
&=-\int_{\partial {M}}\langle\partial_{-}(\rho\, \dpp\dpm \beta),\alpha\rangle\,dS\,,
\end{align*} 
for any $\alpha \in P^{k-1}$ as desired.   

As for the second statement, let $\beta \in P^{k-1}_{D_+}$.  By Corollary \ref{D++equ}, it is enough to show that $((\partial_{+}\partial_{-})\partial_{+}\beta, \alpha)=(\partial_{+}\beta, (\partial_{+}\partial_{-})^{\ast}\alpha)$ for any $\alpha \in P^k$. Clearly, $((\partial_{+}\partial_{-})\partial_{+}\beta, \alpha)=0$. Furthermore, $(\partial_{+}\beta, (\partial_{+}\partial_{-})^{\ast}\alpha)=(\beta, \partial^{\ast}_{+}\partial_{-}^{\ast}\partial_{+}^{\ast}\alpha)=0$ since $\beta \in P^{k-1}_{D_{+}}$. Hence, the statement follows.
\end{proof}

Similar arguments give the following:
\begin{lemma}
For $k\leq n$, 
\begin{itemize}
\item if $\beta \in P^k_{N_{--}}\,$, then $(\dpp\dpm)^{\ast} \beta \in P^k_{N_{+}}\,$;
\item if $\beta \in P^{k-1}_{N_{-}}$, then $\partial_{-}^{\ast}\beta \in P^k_{N_{--}}\,$.  
\end{itemize} 
\end{lemma}

Lemmas \ref{pre} and \ref{pree} will turn out to be essential later in Section \ref{RelSec} to define the relative primitive cohomologies.

\subsubsection{Boundary conditions under maps}\label{bmap}

The two maps $\{\Pi, \ast_r\}$ on symplectic manifolds defined in \eqref{projmap}-\eqref{starmap},
\begin{align*}
\Pi&: \Omega^k \rightarrow P^k~,\\
\ast_ r&: P^k \rightarrow \om^{n-k}\w P^k \in \Omega^{2n-k}~,
\end{align*}
have particularly interesting properties when the forms that are mapped have specified boundary conditions.  It turns out that these two maps can relate forms with symplectic boundary conditions $D_{+}, D_{-}$ and $D_{+-}$  to those with the usual $D$ boundary condition. In the following, we will denote forms with a specified boundary conditions by a subscript.  For example, the notation $\Omega^k_D$ will denote the space of differential $k$-forms that satisfy the standard Dirichlet boundary condition $D$. 
\begin{prop}\label{prop4}
Under the $\Pi$ and $\ast_r$ maps, we have the following relations between forms with specified boundary conditions:
 \begin{align*}
 \Pi&: \Omega^k_D \longrightarrow \Bigg\{~~ \begin{matrix}P^k_{D_{+}} \qquad {\rm for~} \quad k<n\,, \\P^n_{D_{+-}}\qquad \!\!{\rm for~}\quad k=n\,, \end{matrix} \\
\ast_r&: P^k_{D_{-}}\! \longrightarrow ~~~~\Omega^{2n-k}_{D}~,\quad\quad  k \leq n\,.
\end{align*} Moreover, the first map is surjective and the second is injective.
\end{prop}
\begin{proof} 
 Let $\eta \in \Omega_D^k$  for $k\leq n\,$.  We can express $\eta$ in terms of the following:
 \[
 \eta=\beta +\omega \wedge \xi
 \] with $\beta  \in P^k$ and $\xi \in \Omega^{k-2}$, and hence, $\Pi (\eta)=\beta $.  Around $\partial M$, we choose to work in the local Darboux basis $\{w_j\}$ as above. Since $\eta \in D$, this implies that
\begin{align}
 0=w_1 \wedge \eta~|_{\partial M} =[w_1\wedge \beta +\omega \wedge (w_1 \wedge \xi)]~{\big |}_{\partial M} ~.
\end{align} 
Therefore, $\Pi(w_1 \wedge \eta)~|_{\partial M}=\Pi (w_1 \wedge \beta )~|_{\partial M}=0$, and so we find for $k<n$, $\beta  \in D_{+}$ which gives the first map.  

Note that when $k=n$, $\Pi (w_1 \wedge \beta )=0$ is a trivially condition.  (Recall that the $D_{+}$ condition is an empty condition on primitive $n$-form.)  We want to show instead that $\beta \in D_{+-}$ when $k=n$.  This is the condition that $\sigma_{\partial_{+}\partial_{-}}(\drho)\,\beta ~|_{\partial M}=0\,$, or equivalently, $w_1\Lambda (w_1\wedge \beta )~|_{\partial M}=0\,$.  But since $\dpp\dpm$ maps primitive forms to primitive forms and non-primitive forms to non-primitive forms, it follows that
\begin{align}
0=\Pi(w_1\Lambda (w_1\wedge \eta))~|_{\partial M}=w_1\Lambda (w_1\wedge \beta )~|_{\partial M} ~,
\end{align}
where we have also noted $(w_1\wedge \eta)~|_{\partial M}=0\,$.  This thus proves that $\beta \in D_{+-}$ when $k=n$.

To see that the map $\Pi$ is surjective, consider first the case $k<n$ and $\beta  \in P_{D_{+}}^{k}\,$. Locally around $\partial M$, we again express $\beta$ in terms of the decomposition of \eqref{ldecomp}:
 \[
 \beta =w_1 \wedge \beta^1+ w_2\wedge \beta^2+ \Theta_{12} \wedge \beta^3+\beta^4.
 \]
We note that $\beta  \in D_{+}$ implies that at the boundary, $\beta^2~|_{\partial M}=\beta^4~|_{\partial M}=0\,$. Let us therefore define $\eta= w_1\wedge \beta^1+w_2 \wedge \beta^2+\frac{n-k+2}{n-k+1}w_1\wedge w_2 \wedge \beta^3+\beta^4$.  It can be straightforwardly checked that $\eta \in D$ since $w_1\wedge \eta~|_{\partial M}=0\,$, and moreover, $\Pi (\eta)=\beta $. Using the partition of unity, this leads to a well-defined global form with the desired properties. 
 
For the case of $k=n$, let $\beta  \in P_{D_{+-}}^{n}\,$. The local decomposition of \eqref{ldecomp} near the boundary becomes the following:
  \[
 \beta =w_1 \wedge \beta^1+ w_2\wedge \beta^2+ \Theta_{12}\wedge \beta^3~,
 \]
with $\beta^4=0$ since there are no primitive $n$-form without a component in either $w_1$ or $w_2$.  The condition $\beta \in D_{+-}$ further implies that  $\beta^2~|_{\partial M}=0\,$.   This leads us to define $\eta= w_1\wedge\beta^1+w_2\wedge \beta^2+2 \,w_1\wedge w_2 \wedge \beta^3$ which satisfies both $\eta \in D$ and $\Pi(\eta)=\beta$.

Finally, we consider the $\ast_r$ map. Let $\beta \in P_{D_{-}}^k$ for $k\leq n\,$.  We want to show that $\ast_r \,\beta = \om^{n-k} \w \beta$ satisfies the Dirichlet condition.  In local Darboux coordinates $\{w_j\}$ near the boundary, we find
\begin{align*}
w_1 \w (\ast_r \,\beta)~|_{\partial M}= &=\omega^{n-k}\wedge( w_1\wedge \beta)~|_{\partial M}\\
&=\omega^{n-k}\w \left(\Pi(w_1\wedge\beta)+ \om\w  [H^{-1}\Lambda(w_1\w \beta)]\right)|_{\partial M}\\
&=\omega^{n-k+1}H^{-1}\Lambda(w_1\wedge\beta)~|_{\partial M}~=0~.
\end{align*} 
Above, in the second line, we have Lefschetz decomposed $w_1\w \beta$ into two terms, $\beta_{k+1} + \om \w \beta_{k-1}$.  In the third line, we have noted that $\om^{n-k} \w \beta_{k+1} =0 $ by primitivity and also that $\beta \in D_{-}$ implies $\Lambda (w_1 \wedge \beta)|_{\partial M}=0\,$, which allow us to conclude that $\ast_r \, \beta \in D\,$.  Lastly, the injectiveness of this $\ast_r$  map follows from the injectiveness of the map $\ast_r: P^k \rightarrow \Omega^{2n-k}$ without any boundary conditions as mentioned right below \eqref{starmap}.  
  \end{proof}
 
Composing Proposition \ref{prop4} with the $\mathcal{J}$ map, we immediately obtain the following corollary relating $JD$ boundary condition with the $N_{-}, N_{+-}\,$, and $N_{+}$ boundary conditions.
 \begin{cor}
 Under the $\Pi$ and $\ast_r$ maps, we have the following relations between forms with specified boundary conditions:
 \begin{align*}
 \Pi&:  \Omega^k_{JD} \longrightarrow \Bigg\{~~ \begin{matrix}P^k_{N_{-}} \qquad {\rm for~} \quad k<n\,, \\P^n_{N_{+-}}\qquad \!\!{\rm for~}\quad k=n\,, \end{matrix} \\
\ast_r&: P^k_{N_{+}}\! \longrightarrow ~~~~\Omega^{2n-k}_{JD}~,\quad\quad  k \leq n\,.
  \end{align*} Moreover, the first map is surjective and the second is injective.
 \end{cor}
 \begin{proof}
 Let $\eta \in \Omega^k_{JD}$. Then $\mathcal{J}\eta \in \Omega^k_{D}$.  By the lemma above, it follows that $\Pi(\mathcal{J}\eta)$ is either an element of $P^k_{D_{+}}$ when $k<n\,$, or $P^n_{D_{+-}}$ when $k=n$.  Since $\Pi(\mathcal{J}\eta)=\mathcal{J}(\Pi(\eta))$ and applying Lemma \ref{rel}, we obtain $\Pi(\eta) \in P^k_{N_{-}}$ for $k<n\,$ and $\Pi(\eta) \in P^n_{N_{+-}}$ for $k=n$. A similar argument applies for the $\ast_r$ map.
 \end{proof}

 \section{Hodge theory for symplectic Laplacians }
In this section, we will work out the Hodge theory for the symplectic Laplacians \eqref{sympLp}-\eqref{sympLmm} in Section \ref{symplap}.  To do so, we will introduce certain boundary value problems (BVPs) that will be shown to be elliptic.  We begin by first recalling the definition and some basic results of elliptic BVP. 

{\it Notation}: In the following Section 4.1 and Section 4.2 only, we will use the standard notation $H$ and $L$ to denote Hilbert and Lebesgue spaces.  They should not be confused with the $sl(2)$ representation operators defined in Sec. \ref{defsection} as their meaning should be clear from the context.

\subsection{Elliptic boundary value problems} 
We give the definition of an elliptic boundary value problem (BVP) following H\"ormander \cite[Sec.~20.1]{H3}  (see also Schwarz \cite[Sec.~1.6]{S} or Agranovich \cite[Sec. 7.1]{A}).
\begin{definition}[Elliptic BVP]\label{BVPdef}Let $M$ be a compact manifold with a smooth boundary $\partial M$.  Let $E$ and $F$ be vector bundles over $M$, and let $G_j$ for $j=1, \ldots, J$, be vector bundles over $\partial M$. For the differential operators, 
\begin{align*}
\left\{\begin{aligned}
&\mathcal{P}: C^{\infty}(M,E) \rightarrow C^{\infty}(M,F)\\
&\mathcal{B}_j: C^{\infty}(M,E) \rightarrow C^{\infty}(\partial M, G_j)\,, \quad j=1, \ldots, J\,,
\end{aligned}\right.
\end{align*} 
where $\mathcal{P}$ is an operator of order $2m$ and each $\mathcal{B}_j$ is a boundary differential operator of order $m_j$, we consider the boundary value problem for $u\in C^{\infty}(M, E)$ solving
\begin{align}\label{BVPeq}
\mathcal{P}u&=f  \qquad \;  on ~ M\\
\mathcal{B}_j u&=g_j \qquad on~\partial M \quad j=1, \ldots, J\nonumber
\end{align}
for some given $f\in C^{\infty}(M,F)$ and $g_j\in C^{\infty}(\partial M, G_j)$.    Denote by $p$ and $b_j$ the corresponding principal symbol of $\mathcal{P}$ and $\mathcal{B}_j$, respectively.  We say \eqref{BVPeq} is an elliptic BVP or simply that $\{\mathcal{P}, \mathcal{B}_j\}$ is elliptic if the following conditions are satisfied:
\begin{enumerate}
\item $\mathcal{P}$ is elliptic for all $x\in M$;
\item For every $x\in \partial M$ and $\xi \in T^{\ast}_{x}M\!\setminus\!\{0\}$ not proportional to the interior conormal $d\rho$, the map
\begin{align*}
\Psi_{x,\, \xi}: ~\mathcal{M}^{+}_{x,\, \xi}  &\rightarrow  \quad \bigoplus_{1\leq j \leq J} G_{j, x}\\
\phi\quad &\mapsto \quad \left(b_1(\xi+ d\rho \,D_t)\phi\big|_{t=0}\,,\, 
\ldots\,, \,b_J(\xi+d\rho\, D_t)\phi\big|_{t=0}\right)
\end{align*} 
is bijective, where 
\begin{align*}\mathcal{M}^{+}_{x,\, \xi}=\left\{\phi(t) \in C^{\infty}(\mathbb{R},E_x)\, \big|\, p(\xi+ d\rho\, D_t)\phi(t)=0, \text{bounded on $\mathbb{R}_{+}$}\right\}
\end{align*}
is the space of $\mathbb{R}_+$-bounded functions taking values on $E_x$,  
and $D_t=-i\frac{d}{d t}$.
\end{enumerate}
\end{definition}
\begin{remark}
As noted in \cite{H3}, it is sufficient to verify condition (2) above for $\xi \in T^{\ast}_{x}M\!\setminus\!\{0\}$ modulo $\mathbb{R} \,d\rho$.  Modulo this equivalence, $\xi$ can be taken to be orthogonal to $d\rho$ as we shall assume below.  The argument $(\xi + d\rho\, D_t  )$ in condition (2) above then can be thought of as applying the Fourier transform only in the directions tangential to the boundary $\partial M$ and the $t$ coordinate parametrizes the interior normal direction.  Condition (2) is often referred to as the Lopatinski or Shapiro-Lopatinski condition.  
\end{remark}

When a BVP satisfies the above elliptic conditions, the combined operator $\hat{\mathcal{P}}=\{\mathcal{P}, \mathcal{B}_j\}$, i.e. 
\begin{equation*}
\hat{\mathcal{P}}: H^s(M,E)\rightarrow H^{s-2m}(M,F)\oplus H^{s-m_1-\frac{1}{2}}(\partial M, G_1) \oplus \cdots \oplus H^{s-m_J-\frac{1}{2}}(\partial M, G_J)\,.
\end{equation*} 
is well-known to be Fredholm. The boundary value problems that we consider below will also be self-adjoint. 
\begin{definition}[Self-Adjoint BVP]When the bundle $F=E$, we say $\{\mathcal{P}, \mathcal{B}_j\}$ is self-adjoint if $\mathcal{P}$ is self-adjoint and the following holds: for any $u,v \in C^{\infty}(M,E)$,
\begin{itemize}
\item if $\mathcal{B}_j(u)=\mathcal{B}_j(v)=0$ for every $j$, then $(\mathcal{P}u, v)=(u, Pv)$;
\item if $\mathcal{B}_j(u)=0$ for every $j$, and $(\mathcal{P}u, v)=(u, \mathcal{P}v)$, then $\mathcal{B}_j(v)=0$ for every $j$.
\end{itemize}
\end{definition}
In the next lemma we give some general properties of self-adjoint elliptic BVPs. (For reference, see \cite{H3, SM}).
\begin{lemma}\label{ell}
For an elliptic BVP, $\hat{\mathcal{P}}=\{\mathcal{P}, \mathcal{B}_j\}$, that is self-adjoint,  the following holds: 
\begin{itemize}
\item the kernel of $\hat{\mathcal{P}}$, denoted by $\ker \hat{\mathcal{P}}$,  is finite and smooth;
\item for any $\chi \in H^s(M,F)$ which is orthogonal to $\ker \hat{\mathcal{P}}$, there exists a unique $\phi \in H^{s+2m}(M,E)$ and $\phi \bot \ker\hat{\mathcal{P}}$ such that $\mathcal{P}\phi=\chi$ and $B_j(\phi)=0\,$ for all $j\,$;
\item if $\chi \in H^s(M,F)$ and $\mathcal{P}\phi=\chi\,$ and $\mathcal{B}_j(\phi)=0\,$ for all $j\,$, then $\phi \in H^{s+2m}(M, E)\,$.
\end{itemize}
\end{lemma}

With this lemma, we can show that the weak solutions of self-adjoint, elliptic BVPs are actually strong solutions.
\begin{lemma}\label{ws}
Given $\chi \in L^2(M,E)$.  Let $\phi \in L^2(M,E)$ satisfy the following:
\[
(\phi, \mathcal{P}\psi)=(\chi, \psi)
\]
for any $\psi \in C^{\infty}(M,E)$ satisfying $\mathcal{B}_j(\psi)=0\,$, for $j=1, \ldots, J$. Then $\phi \in H^{2m}(M,E)$ and 
\[ 
\mathcal{P}\phi=\chi\,,\quad \mathcal{B}_j(\phi)=0\,, ~~ {\rm for}~ j=1, \ldots, J\,.
\]
\end{lemma}
When $\chi=0$, Lemma \ref{ws} implies immediately the following:
\begin{cor}\label{ker}
If $\phi \in L^2(M,E)$ satisfies $(\phi, \mathcal{P}\psi)=0$ for any $\psi\in C^{\infty}(M,E)$ with $\mathcal{B}_j(\psi)=0\,$, for $j=1, \cdots, J$, then $\phi \in \ker \hat{\mathcal{P}}$. In particular, $\phi$ is smooth and $\mathcal{B}_j(\phi)=0\,$, for $j=1, \cdots, J$.
\end{cor}
We will give a proof of Lemma \ref{ws} based on the arguments of Schechter \cite{SM}, where the case for functions is proved.
\begin{proof}[Proof of Lemma \ref{ws}]
Since the space $\ker \hat{\mathcal{P}}$ is finite-dimensional, we can write $\chi= \chi^1+\chi^2$ with $\chi^1\!\in \ker \hat{\mathcal{P}}$ and $\chi^2 \bot \ker\hat{\mathcal{P}}$. By Lemma \ref{ell}, there exists a $\varphi \in H^{2m}(M, E)$ such that $\mathcal{P} \varphi=\chi^2$ and $\mathcal{B}_j(\varphi)=0\,$, for $j=1, \cdots, J$. Then 
\[
(\phi-\varphi, \mathcal{P}\psi)=(\chi^1, \psi)
\] for any $\psi \in C^{\infty}(M,E)$ satisfying the boundary conditions $\mathcal{B}_j(\psi)=0\,$, for $j=1, \cdots, J\,$. There exists a sequence $\varphi_i \in C^{\infty}(M,E)$ such that $\varphi_i \rightarrow \phi-\varphi$ in $L^2$ norm, as $i \rightarrow \infty$. Let $\varphi_i=\varphi_i^1+\varphi_i^2$ with $\varphi_i^1 \in \ker \hat{\mathcal{P}}$ as the projection and $\varphi_i^2 \bot \ker\hat{\mathcal{P}}$. Then there exist $\upsilon_i \in H^{2m}(M,E)$ with $\upsilon_i \bot \ker \hat{\mathcal{P}}$ such that $\mathcal{P}\upsilon_i=\varphi_i^2$ and $\mathcal{B}_j(\upsilon_i)=0$ for every $i$ and $j$. Therefore, 
\begin{align*}
(\phi-\varphi, \varphi_i)&=(\phi-\varphi, \varphi_i^1)+(\phi-\varphi, \varphi_i^2)=(\phi-\phi, \phi_i^1)+(\phi-\varphi, P\upsilon_i)\\
&=(\phi-\varphi, \varphi_i^1)+(\chi^1, \upsilon_i)=(\phi-\varphi, \varphi_i^1)\,.
\end{align*} 
As $i\rightarrow \infty$, we get $\varphi_i^1 \rightarrow \phi-\varphi$. Since $\ker \hat{\mathcal{P}}$ is closed and $\phi-\varphi \in \ker \hat{\mathcal{P}}$, they imply that $ \phi \in H^{2m}(M,E)$ and $\mathcal{B}_j(\phi)=0\,$, for $j=1, \cdots, J\,$.
\end{proof}

\subsection{Hodge decompositions}\label{hodgede}
\begin{definition}
We call the following spaces 
\begin{align*}
P\mathcal{H}_{+}^k&=\{\beta \in H^1\!P^k~ {\big |}\, \partial_{+} \beta=\partial_{+}^{\ast} \beta=0\}\,,\quad  P\mathcal{H}_{-}^k=\{\beta \in H^1\!P^k~ {\big |}\,\partial_{-} \beta=\partial_{-}^{\ast} \beta=0\}\,, 
\end{align*} 
where $k=0,1, \ldots, n-1\,$, and
\begin{align*}
P\mathcal{H}_{+}^n=\{\beta \in H^2P^n|\,\partial_{+} \partial_{-}\beta=\partial_{+}^{\ast} \beta=0\}, \quad P\mathcal{H}_{-}^n=\{\beta \in H^2P^n|\, \partial_{-} \beta=\partial_{-}^{\ast} \partial_{+}^{\ast}\beta=0\}\,,
\end{align*} 
the space of {\it harmonic fields} for $\Delta_+$, $\Delta_-$, $\Delta_{++}$, and $\Delta_{--}$ Laplacians, respectively.  
\end{definition} 
\begin{remark} For a manifold with boundary, the notion of a harmonic field is different from that of a harmonic form.  For instance, a primitive $k$-form $\beta \in P^k$ is a harmonic form of $\Delta_{+}$ if $\Delta_{+}\, \beta=0$ on $M$. However, this does not imply that $\beta$ is also a harmonic field (i.e. $\partial_{+}\, \beta=\partial_{+}^{\ast}\, \beta=0$) when $\partial M$ is non-empty. 
\end{remark} 
Below, we shall use the theory of elliptic BVPs to obtain Hodge decompositions of primitive forms on symplectic manifolds with boundary.  We begin first with the decompositions associated with the second-order Laplacians, $(\Delta_+, \Delta_-)$, and then proceed to describe the case of the fourth-order Laplacians, $(\Delta_{++}, \Delta_{--})$.   

\subsubsection{Second-order symplectic Laplacians}
\begin{thm}[Hodge decomposition for $\Delta_{+}$]\label{Hdp}
For $k<n\,$, 
\begin{itemize}
\item[1.] $P\mathcal{H}^k_{+, D_{+}}$ and $P\mathcal{H}^k_{+, N_{+}}$ are finite-dimensional and smooth;
\item[2.]  The following decompositions hold:
\begin{align*} 
 {\rm (i)}~~& L^2P^k =P\mathcal{H}_{+,D_{+}}^k\oplus \partial_{+}\,H^1\!P_{D_{+}}^{k-1}\oplus \partial_{+}^{\ast}\,H^1\!P^{k+1};\\ 
{\rm (ii)}~~& L^2P^k =P\mathcal{H}_{+,N_{+}}^k\oplus \partial_{+}\,H^1\!P^{k-1}\oplus \partial_{+}^{\ast}\,H^1\!P_{N_{+}}^{k+1};\\
{\rm (iii)}~~& L^2 P^k =L^2P\mathcal{H}_{+}^k \oplus \partial_{+}\,H^1\!P_{D_{+}}^{k-1}\oplus \partial_{+}^{\ast}\, H^1 P^{k+1}_{N_{+}}.
\end{align*} 
\end{itemize}
\end{thm} 
Note the presence of an additional subscript when we would like to restrict consideration to differential forms that satisfy a particular boundary condition.  For instance, $P_{D_{+}}^k$ denotes the space of primitive $k$-forms that satisfy the $D_+$ boundary condition.  Applying the above results to $\mathcal{J}\beta$, we obtain analogous Hodge decompositions for $\Delta_{-}$.
\begin{thm}[Hodge decomposition for $\Delta_{-}$]\label{Hdm}
 For $k<n\,$,
 \begin{itemize}
  \item [1.] $P\mathcal{H}^k_{-, D_{-}}$ and $P\mathcal{H}^k_{-, N_{-}}$ are finite-dimensional and smooth.
\item [2.] The following decompositions hold:
\begin{align*}
{\rm (i)}~~&L^2P^k=P\mathcal{H}_{-, D_{-}}^k \oplus \partial_{-}\, H^1\!P_{D_{-}}^{k+1}\oplus  \partial_{-}^{\ast}\,H^1\!P^{k-1};\\
{\rm (ii)}~~&L^2P^k=P\mathcal{H}_{-, N_{-}}^k \oplus  \partial_{-}\,H^1\!P^{k+1}\oplus  \partial_{-}^{\ast}\,H^1\!P_{N_{-}}^{k-1};\\
{\rm (iii)}~~&L^2 P^k=L^2P\mathcal{H}_{-}^k \oplus  \partial_{-}\,H^1\!P^{k+1}_{D_{-}}\oplus  \partial_{-}^{\ast}\,H^1\!P_{N_{-}}^{k-1}.
\end{align*}
\end{itemize}  
\end{thm} 
To prove Theorem \ref{Hdp}, we will introduce two natural, elliptic BVPs, which are self-adjoint. Consider first the following symplectic BVP.
\begin{prop}\label{2ndE} For $k<n$, the following boundary value problem is self-adjoint and elliptic for any $\beta, \lambda \in P^k$:
\begin{align}\label{BVP1}
&\quad\Delta_{+}\, \beta=\lambda\,,  \quad\qquad\text{on $M$}\\
&\left\{  \begin{aligned}
& \partial_{+}(\rho\,\beta)=0\,, \\
&\partial_{+}(\rho\, \partial_{+}^{\ast} \beta)=0\,,
\end{aligned}~~\quad \text {on $\partial M$}\,. \right. \nonumber
\end{align}
\end{prop}
\begin{proof}
That this BVP is self-adjoint can be easily checked using the Green's formula of Lemma \ref{green}.  For ellipticity of the BVP, we note first that $\Delta_{+}$ is elliptic acting on $P^k(M)$ for $k<n$ as mentioned in Section \ref{symplap}.  This satisfies the first condition of Definition \ref{BVPdef}.  To check the second condition of Definition \ref{BVPdef}, we need to first solve for the space $\mathcal{M}^{+}_{x,\xi}$ of $\mathbb{R_+}$-bounded solutions of the system of ordinary differential equations:
\begin{align}\label{ode2}
\sigma_{\Delta_{+}}(\xi - i \, d\rho \tfrac{d}{dt})\,\phi(t)=0\,,
\end{align}
for every $x \in \partial M$ and $\xi \in T^{\ast}_{x}M\!\setminus\!\{0\}$ orthogonal to $d\rho$.  Here, $\phi(t)\in C^{\infty}(\mathbb{R},P^k(T_x^*M))$ where $P^k(T_x^*M)$ denotes the primitive exterior cotangent space at $x$ of degree $k$.  After finding the solutions, we then need to show the bijectivity of the linear map
\begin{align}\label{ode2s}
\Psi_{x, \xi}: \mathcal{M}^{+}_{x,\xi}&\rightarrow G_{1, x}\oplus G_{2, x}\nonumber \\
\phi(t) &\rightarrow \left(\sigma_{\dpp}(d\rho)\,\phi(t)\big|_{t=0}\,,\, \sigma_{\dpp}(d\rho)\left[\sigma_{\dpps}(\xi - i \, d\rho \tfrac{d}{dt})\,\phi(t)\right]\big|_{t=0}\right)
\end{align}  
where in the second line, we have used \eqref{sdiff} for the boundary principal symbols.  Note that the dimension of the codomain space $G_{1, x}\oplus G_{2, x}$ represents the number of boundary conditions.  As will be clear from the calculations below, $\dim \mathcal{M}^{+}_{x,\xi} = \dim (G_{1, x}\oplus G_{2, x}) = \dim P^k(T_x^*M)$.   So for bijectivity, we only need to show the injectivity of $\Psi_{x,\, \xi}$ which will be the main task in the following.  

For a fixed $x \in \partial M$,  the covector $\xi \in T^{\ast}_xM$ can be of two types: (A) $\La(d\rho\w\xi)\neq 0$; or (B) $\La(d\rho\w\xi)=0$. 
To simplify our calculations, we note that $T^{\ast}_xM$ as a symplectic vector space can be modeled by $\mathbb{R}^{2n}$, and without lost of generality, we can choose to work in a Darboux basis $\{w_i\}$ on $T^*_xM$ with $\om= w_1 \w w_2 + w_3 \w w_4 + \ldots,$ such that $d\rho=w_1$  and the compatible metric $g_{ij}=\delta_{ij}$, for $i, j = 1, \ldots, 2n$.  For case (A), it is sufficient by the residual local symplectomorphism to consider only $\xi = p\,w_2+q\,w_3$ where $p\neq 0$.  We will give the calculations for the $q=0$ case (A1) here and give those for the $q\neq0$ case (A2) in Appendix \ref{AppC}.  For case (B), it is sufficient by symplectomorphism to only consider $\xi= q\,w_3$ with $q>0$.  The solutions of \eqref{ode2} take different forms for each of these three cases and so we need to consider them separately.

{\bf Case (A1):}  Let $\xi = \ep\, p\, w_2$, where $p>0$ and $\ep = \pm 1$.  Since we are singling out the $\{w_1, w_2\}$ components here, it is advantageous to express $\phi(t)\in \mathcal{M}^{+}_{x,\xi=\ep p w_2}$ in terms of the decomposition of  \eqref{ldecomp}:
\begin{align}\label{ldecomp12} 
\phi_k(t) =w_1\wedge \phi_{k-1}^1(t)+w_2\wedge \phi_{k-1}^2(t)+\Theta_{12}\wedge\phi_{k-2}^3(t) +\phi_{k}^4(t)\,,
\end{align} 
where $\{\phi_{k-1}^1, \phi_{k-1}^2, \phi_{k-2}^3, \phi_k^4\}$ above are primitive forms without $w_1$ and $w_2$ components.  As usual, the subscript denotes the degree of the form.

To write out the ordinary differential system \eqref{ode2}, we use the calculation of the principal symbol $\sigma_{\Delta_+}(\zeta_1 w_1 + \zeta_2 w_2)$ in   Appendix \ref{AppB}.  Explicitly, we substitute $\zeta_1 = -i\,\partial_t$ and $\zeta_2=\ep \,p$ in \eqref{Dpsyma}.  With the differential system in hand, the general $\mathbb{R}_+$-bounded solution 
of  $\sigma_{\Delta_+} ( -i \, w_1 \partial_t + \ep\, p \,w_2)\, \phi_k(t)=0$ can be solved straightforwardly and shown to take the following decomposed form
\begin{align}\label{bs2p}
\phi_k(t)&=  w_1 \w \left[\left(1+\tfrac{p\,t}{2h+1}\right) c^1_{k-1} +\left(\tfrac{-i\,\ep\,p\,t}{2h+1}\right) c^2_{k-1} \right] e^{-pt} \nonumber \\
& \qquad  + w_2 \w \left[\left(\tfrac{-i\,\ep\,p\,t}{2h+1}\right) c^1_{k-1} + \left(1-\tfrac{p\,t}{2h+1}\right) c^2_{k-1} \right]e^{-pt}  +\Th \w c^3_{k-2} e^{-pt} + c^4_k e^{-pt}
\end{align} 
where $h=n-k$ and $\{c^1_{k-1}, c^2_{k-1}, c^3_{k-2}, c^4_k\}$ are constant primitive forms of $P^*(\mathbb{R}^{2n})$ without components in $\{w_1, w_2\}$.      

For injectivity of $\Psi_{x, \xi=\ep p w_2}$, we consider its kernel.  With \eqref{ode2s}, 
this imposes the following conditions on the solutions in \eqref{bs2p}:
\begin{align*}
b_1(-i\, w_1 \partial_t + \ep p w_2)\phi_k(t)\,|_{t=0}=0
&\Longrightarrow ~ \sigma_{\dpp}(d\rho)\,\phi_k(0)= 0\,, \\
&\Longrightarrow~  \phi_{k-1}^2(0)=0\,, \quad  \phi_k^4(0) = 0\, , \\
&\Longrightarrow ~ c^2_{k-1} =0\,, \quad c^4_k =0\,,\\
b_2(-i\,w_1 \partial_t + \ep p w_2)\phi_k(t)\,|_{t=0}=0
&\Longrightarrow~ \sigma_{\dpp}(d\rho)\left[\sigma_{\dpps}(-i\, w_1 \partial_t + \ep p w_2)\,\phi(t)\right]\big|_{t=0} =0\,,\\
&\Longrightarrow ~ \partial_t\phi^3_{k-2}(0) =0 \,, ~~  i\,\partial_t\phi^1_{k-1}(0) - \ep p\, \phi^2_{k-1}(0) =0\,, \\
&\Longrightarrow ~ c^3_{k-2} = 0\,,\quad i\,c^1_{k-1}+\ep\, c^2_{k-1}=0 \,.
\end{align*}
Above give four boundary condition equations on $\{\phi_{k-1}^1, \phi_{k-1}^2, \phi_{k-2}^3, \phi_k^4\}$.  Together, they imply that the four constant forms $\{c^1_{k-1}, c^2_{k-1}, c^3_{k-2}, c^4_k\}$ in \eqref{bs2p} are identically zero, and therefore, proving injectivity when $\xi= \ep\, p\, w_2$.

{\bf {Case (A2):}} For this case where $\xi= p\, w_2 + q\, w_3$ with both $p\neq0$ and $q\neq0$, the proof of injectivity is given in Appendix \ref{AppC1}.

{\bf {Case (B):}} Let $\xi = q\, w_3$ (for some $q>0$).  Because the components $\{w_1, w_3\}$ are now picked out, it is useful to decompose further each of the forms $\{\phi^1, \phi^2, \phi^3, \phi^3\}$ in the decomposition of \eqref{ldecomp12} to extract out the dependence on $\{w_3, w_4\}$ as well.  This result in the following decomposition of $\phi_k(t)$ into 16 terms:
\begin{align}\label{bexpand2}
\phi_k(t) &=w_1\wedge \phi_{k-1}^1(t)+w_2\wedge \phi_{k-1}^2(t)+\Theta_{12}\wedge\phi_{k-2}^3(t) +\phi_{k}^4(t)\nonumber\\
&= w_1 \w \left[w_3 \w \ga^{13}_{k-2}(t) + w_4\w \ga^{14}_{k-2}(t) + \Thp \w\ga^{134}_{k-3}(t) + \ga^1_{k-1}(t)\right] \nonumber\\
& ~~ +w_2\w \left[w_3  \w \ga^{23}_{k-2}(t) + w_4\w \ga^{24}_{k-2}(t) + \Thp \w\ga^{234}_{k-3}(t) + \ga^2_{k-1}(t)\right]\nonumber \\
& \quad +\Th \w \left[w_3 \w\ga^{123}_{k-3}(t) + w_4 \w\ga^{124}_{k-3}(t) + \Thp \w\ga^{1234}_{k-4}(t) + \ga^{12}_{k-2}(t)\right] \\
&\qquad \ + \left[w_3 \w\ga^3_{k-1}(t) + w_4 \w\ga^4_{k-1}(t) + \Thp \w\ga^{34}_{k-2}(t) + \ga^0_k(t)\right] \nonumber
\end{align}
where the 16 $\gamma(t)$'s are forms without components in $\{w_1, w_2, w_3, w_4\}$ (with the degree labeled by the subscript) and   
\begin{align}\label{Thpdef4}
\Theta'_{34}=  w_3 \w w_4 - \dfrac{1}{H} \sum_{i=3}^{n} w_{2i-1} \w w_{2i}\,.
\end{align} 
To find the the general  $\mathbb{R}_+$-bounded solution $\phi_k(t) \in \mathcal{M}^{+}_{x, \xi = q w_3}$ for  $\sigma_{\Delta_{+}} ( -i \, w_1 \partial_t + q \,w_3)\, \phi_k(t)=0$, we use the calculation of the principal symbol $\sigma_{\Delta_{+}}(\zeta_1 w_1 + \zeta_2 w_2 + \zeta_3 w_3)$ in Appendix \ref{AppB} and  substitute in $(\zeta_1, \zeta_2, \zeta_3)=( -i\,\partial_t, 0, q)$.  This gives us a system of ordinary differential equation involving the 16 $\ga$'s.  The $\mathbb{R}_+$-bounded solutions for the $\ga(t)$'s in \eqref{bexpand2} can be solved and expressed in terms of 16 constant primitive forms, which we will label by $c^l$ for $l=1, \ldots, 16$, that also have no components in $\{\wa, \wb, \wc, \wdd\}$:
\begin{align*}
&\ga_k^0(t)= c_k^1 e^{-qt} \,,\quad  \ga_{k-1}^1(t)=  c_{k-1}^2 e^{-qt}\,,\quad  \ga^3_{k-1}(t)=  c_{k-1}^3 e^{-qt}, \\ 
&\ga_{k-2}^{13}(t)= c_{k-2}^4 e^{-qt}\,,\quad \ga_{k-2}^{24}(t)= c_{k-2}^5 e^{-qt}\,,\quad 
\ga_{k-4}^{1234}(t)= c_{k-4}^6 e^{-qt}\,,
\end{align*}
\begingroup
\renewcommand*{\arraystretch}{1.19}
\begin{align*}
\begin{pmatrix} 
\ga_{k-1}^{2} \\
 \ga_{k-1}^{4} 
\end{pmatrix}&=
c_{k-1}^7\begin{pmatrix}1-\frac{q\,t}{2h+1}\\ \frac{i\,q\,t}{2h+1}\end{pmatrix} e^{-qt} + c_{k-1}^8 \begin{pmatrix} \frac{i\,q\,t}{2h+1}\\ 1+\frac{q\,t}{2h+1}\end{pmatrix}e^{-qt},\\
%
\begin{pmatrix} 
 \ga_{k-3}^{134} \\
 \ga_{k-3}^{123}
\end{pmatrix}&=
c_{k-3}^9\begin{pmatrix}1+\frac{q\,t}{2h+3}\\ \frac{-i\,q\,t}{2h+3}\end{pmatrix} e^{-qt} 
+ c_{k-3}^{10} \begin{pmatrix} \frac{-i\,q\,t}{2h+3}\\ 1-\frac{q\,t}{2h+3}\end{pmatrix}e^{-qt},
\\
%
\begin{pmatrix} 
 \ga_{k-3}^{234} \\
 \ga_{k-3}^{124}
\end{pmatrix}
&=
c_{k-3}^{11}\begin{pmatrix}1-\frac{q\,t}{2h^2+4h+1}\\ \frac{i\,q\,t}{2h^2+4h+1}\end{pmatrix} e^{-qt} 
+ c_{k-3}^{12} \begin{pmatrix} \frac{i\,q\,t}{2h^2+4h+1}\\ 1+\frac{q\,t}{2h^2+4h+1}\end{pmatrix}e^{-qt},
\\
%
\begin{pmatrix} 
\ga_{k-2}^{14}\\ \ga_{k-2}^{23}\\  \ga_{k-2}^{34}\\ \ga_{k-2}^{12}
\end{pmatrix}&=c_{k-2}^{13}\begin{pmatrix}0 \\ 1\\ i \\ 0 \end{pmatrix}e^{-qt}
+c_{k-2}^{14}\begin{pmatrix}1 \\ \frac{-1}{h+1}\\ 0 \\ i \end{pmatrix}e^{-qt}
+c_{k-2}^{15}\begin{pmatrix}\frac{2h+1}{q} + t \\ \frac{2h+1}{q}- t \\ \frac{-i\, h}{h+1}  t\\ i\,t \end{pmatrix}e^{-qt}\\
&\qquad+c_{k-2}^{16}\begin{pmatrix}\frac{2}{q^2}(4h+1)(h+1) + \frac{2}{q}(2h+1)t +t^2 \\ -\frac{2h}{q^2}(4h^2+6h+3) +  \frac{2}{q}(2h^2+2h+1)t -t^2\\ \frac{i\, 4h^2}{q}t -\frac{i\,h}{h+1} t^2\\ i\, t^2 \end{pmatrix}e^{-qt}.
\end{align*}  
For injectivity, we look at the kernel of $\Psi_{x, \xi=qw_3}$ which imposes the following conditions:
\begin{align*}
b_1(x, -i\, w_1 \partial_t +qw_3)\phi_k(t)\big|_{t=0}\,=0
&\Longrightarrow ~ \sigma_{\dpp}(d\rho)\,\phi_k(t)\,|_{t=0} = 0\,, \\
&\Longrightarrow ~ \phi_{k-1}^2(0)=0\,, \quad  \phi_k^4(0) = 0\, , \\
&\Longrightarrow  \left\{\begin{aligned}
&\ga^{23}(0) =0, \,\ga^{24}(0) =0, \,\ga^{234}(0) =0, \, \ga^{2}(0) =0\,, \\
&\ga^{3}(0) =0, \,\ga^{4}(0) =0, \,\ga^{34}(0) =0, \, \ga^{0}(0) =0\,, 
\end{aligned}
\right.\\
b_2(x, -i\, w_1\partial_t + q w_3)\phi_k(t)\big|_{t=0}\,=0
&\Longrightarrow~ \sigma_{\dpp}(d\rho)\left[\sigma_{\dpps}(-i\, w_1 \partial_t + q\,w_3)\,\phi(t)\right]\big|_{t=0} =0\,,\\
&\Longrightarrow\left\{
\begin{aligned}
&i\pa_t\phi^3_{k-2}(0) + q\, i_{e_3} \phi^2_{k-1}(0) = 0,\\
& i\pa_t\phi^1_{k-1}(0) - q\, i_{e_3}\phi^4_k(0) - \tfrac{q}{h+2} \Thp \w \ga_{k-3}^{123}(0) +\tfrac{q}{h+1}w_4 \w\ga_{k-2}^{12}(0) =0, 
\end{aligned}
\right. \\
&\Longrightarrow\left\{
\begin{aligned}
&\pa_t \ga^{13}(0)=0, ~ \pa_t \ga^{123}(0) = 0, ~ \pa_t\ga^{1234}(0)=0,\\
&i\pa_t \ga^{1}(0) - q \ga^3(0) =0,\\
&i\pa_t \ga^{12}(0) + q \ga^{23}(0) =0,\\
&i \pa_t \ga^{124}(0) + q \ga^{234}(0) =0,\\
&i \pa_t \ga^{134}(0) - \tfrac{q}{h+2} \ga^{123}(0) =0, \\
&i \pa_t \ga^{14}(0) - q \ga^{34}(0) + \tfrac{q}{h+1}\ga^{12}(0) =0,
\end{aligned}
\right. 
\end{align*}
\endgroup
where $i_{e_3}$ is the interior product with respect to the tangent vector dual to $w_3$.  
It can then be straightforwardly checked that imposing the above 16 boundary condition equations on the $\ga(t)$'s for the general $\mathbb{R_+}$-bounded solutions requires that all 16 constant forms $c^{l}$ for $l=1,\ldots, 16$ are identically zero. Therefore we conclude that the map $\Psi_{x,\, \xi=q w_3}$ is injective. 
\end{proof}

Likewise, as can be checked by similar arguments, the following is also true.
\begin{prop}
For $k<n$, the following boundary value problem is self-adjoint and elliptic for any $\beta, \lambda \in P^k$:
\begin{align}\label{BVP2}
&\quad\Delta_{+}\, \beta=\lambda\,, \, \qquad\quad\text{on $M$}\\
&\left\{  \begin{aligned}
& \partial_{+}^{\ast}(\rho\,\beta)=0\,, \\
&\partial_{+}^{\ast}(\rho\, \partial_{+} \beta)=0\,,
\end{aligned}~~\quad \text {on $\partial M$}\,.\right.\nonumber
\end{align}
\end{prop}
The elliptic property of the above two BVPs forms the basis of the proof of Theorem \ref{Hdp}.
\begin{proof}[Proof of Theorem \ref{Hdp}]
We first show that $P\mathcal{H}^k_{+, D_{+}}$ is the kernel of the BVP \eqref{BVP1}. First, it is clear that the kernel of BVP \eqref{BVP1} lies within a subset of $P\mathcal{H}^k_{+, D_{+}}$.  Let $\gamma \in P\mathcal{H}^k_{+, D_{+}}$ and also let $\beta \in P^k$ satisfies the boundary conditions of \eqref{BVP1}, i.e. both $\beta$ and $\partial_{+}^{\ast}\beta$ satisfy the $D_{+}$ condition. By Green's formula, we have
\[
0=(\partial_{+}\gamma, \partial_{+}\beta)+(\partial_{+}^{\ast}\gamma, \partial_{+}^{\ast}\beta)=(\gamma, \Delta_{+}\beta)
\]
By Corollary \ref{ker}, this implies that $\gamma$ must belong to the kernel of the BVP \eqref{BVP1}. Thus, we conclude that $P\mathcal{H}^k_{+, D_{+}}$ is the kernel of BVP \eqref{BVP1}. By Lemma \ref{ell}, we can conclude that $P\mathcal{H}^k_{+, D_{+}}$  is finite-dimensional and smooth.  Similar arguments using the BVP in \eqref{BVP2} will give the analogous result that $P\mathcal{H}^k_{+, N_{+}}$ is finite-dimensional and smooth.

To prove the Hodge decomposition 2.(i) in Theorem \ref{Hdp}, we first write
\[
L^2P^k=P\mathcal{H}^k_{+, D_{+}}\oplus P\mathcal{H}_{+, D_{+}}^{k,\bot}
\] where $P\mathcal{H}_{+, D_{+}}^{k,\bot}$ denotes the orthogonal complement. For any $\beta \in L^2P^k$, let $\gamma$ be its projection to $P\mathcal{H}^k_{+, D_{+}}$. By Lemma \ref{ell}, there exists a unique $\varphi \in H^2\!P^k \cap P\mathcal{H}_{+, D_{+}}^{k,\bot}$ that solves the BVP of \eqref{BVP1}, i.e. $\Delta_{+} \varphi = \lambda$ with $\lambda=\beta - \gamma$.
Therefore, we can write 
\[
\beta=\gamma+\partial_{+}(\partial_{+}^{\ast}\varphi)+\partial_{+}^{\ast}(\partial_{+}\varphi)
\] with $\gamma \in P\mathcal{H}^k_{+, D_{+}}$ and $\partial_{+}^{\ast}\varphi\in H^1\!P^{k-1}_{D_{+}}$. This proves the decomposition. The $L^2$-closedness of $\partial_{+}H^1\!P^k_{D_{+}}$ is implied by this decomposition using standard functional analysis arguments.

The proof for the Hodge decomposition 2.(ii) is analogous to that for 2.(i) but makes use of the BVP \eqref{BVP2} instead.  It remains to prove the decomposition 2.(iii).  Our arguments will be similar to those in \cite{S} to prove a similar-type decomposition with respect to the Laplace-de Rham Laplacian $\Delta_d$. 

By the decompositions of 2.(i) and 2.(ii), we can express any $\beta \in L^2P^k$ as follows:  
\begin{align*}
\beta&=\gamma_1+\partial_{+}\varphi_1+\partial_{+}^{\ast}\sigma_1\\
\beta&=\gamma_2+\partial_{+}\varphi_2+\partial_{+}^{\ast}\sigma_2
\end{align*}
where $\gamma_1 \in P\mathcal{H}^k_{+, D_{+}}, \varphi_1 \in H^1\!P^{k-1}_{D_{+}}, \sigma_1 \in H^1\!P^{k+1}, \gamma_2\in P\mathcal{H}^{k}_{+, N_{+}}, \varphi_2 \in H^1\!P^{k-1}$ and $\sigma_2 \in H^1\!P^{k+1}_{N_{+}}$. Now, we define $\varphi=\beta-\partial_{+}\varphi_1-\partial_{+}^{\ast}\sigma_2$. We will show that $\varphi \in P\mathcal{H}^k_{+}$ when $\beta \in H^1\!P^k$. This is because
\begin{align*}
(\varphi, \partial_{+}\nu)&=(\beta-\partial_{+}\varphi_1, \partial_{+}\nu) = (\beta-\gamma_1-\partial_{+}\varphi_1, \partial_{+}\nu)=0\,, \quad  \text{for}\,\nu\in H^1\!P^{k-1}_{D_{+}},\\
(\varphi, \partial_{+}^{\ast}\nu)&=(\beta-\partial_{+}^{\ast}\sigma_2, \partial_{+}^{\ast}\nu)=(\beta-\gamma_2-\partial_{+}^{\ast}\sigma_2, \partial_{+}^{\ast}\nu)=0\,,\!\quad \text{for}\, \nu\in H^1\!P^{k+1}_{N_{+}}, 
\end{align*} and $H^1\!P^{k}_{D_{+}}$ and $H^1\!P^{k}_{N_{+}}$ are dense in $H^1\!P^k$. Therefore, we obtain 
\[
H^1\!P^k=P\mathcal{H}^k_{+}\oplus \partial_{+}H^1\!P^{k-1}_{D_{+}}\oplus \partial_{+}^{\ast}H^1\!P_{N_{+}}^{k+1}.
\] Since $ \partial_{+}H^1\!P^{k-1}_{D_{+}}$ and $\partial_{+}^{\ast}H^1\!P_{N_{+}}^{k+1}$ are closed in the $L^2$-topology, the $L^2$-decomposition then follows by means of a completion argument.
\end{proof}

\subsubsection{Fourth-order symplectic Laplacians}

The fourth-order Laplacian $\Delta_{++}$ has the following Hodge decomposition. 
\begin{thm}[Hodge decomposition for $\Delta_{++}$]\label{Hdpp} 
\
\begin{itemize}
\item[1.] $P\mathcal{H}^n_{+, D_{++}}$ and $P\mathcal{H}^n_{+, N_{+}}$ are finite-dimensional and smooth;
\item[2.]  The following decompositions hold:
\begin{align*} 
 {\rm (i)}~~& L^2P^n =P\mathcal{H}_{+,D_{++}}^n\oplus \partial_{+}\,H^1\!P_{D_{+}}^{n-1}\oplus (\partial_{+}\partial_{-})^{\ast}H^2\!P^{n};\\ 
{\rm (ii)}~~& L^2P^n =P\mathcal{H}_{+,N_{+}}^n\oplus \partial_{+}\,H^1\!P^{n-1}\oplus (\partial_{+}\partial_{-})^{\ast}H^2\!P_{N_{--}}^{n};\\
{\rm (iii)}~~& L^2 P^n =L^2P\mathcal{H}_{+}^n \oplus \partial_{+}\,H^1\!P_{D_{+}}^{n-1}\oplus (\partial_{+}\partial_{-})^{\ast} H^2\!P^{n}_{N_{--}}.
\end{align*} 
\end{itemize}
\end{thm} 

Applying this to $\mathcal{J} \beta$ gives the Hodge decomposition for $\Delta_{--}$. 

\begin{thm}[Hodge decomposition for $\Delta_{--}$]\label{Hdmm} 
\
\begin{itemize}
\item[1.] $P\mathcal{H}^n_{-, D_{-}}$ and $P\mathcal{H}^n_{-, N_{--}}$ are finite-dimensional and smooth;
\item[2.]  The following decompositions hold:
\begin{align*} 
 {\rm (i)}~~& L^2P^n =P\mathcal{H}_{-,D_{-}}^n\oplus (\partial_{+}\partial_{-})H^2\!P_{D_{++}}^{n}\oplus \partial_{-}^{\ast}\,H^1\!P^{n-1};\\ 
{\rm (ii)}~~& L^2P^n =P\mathcal{H}_{-,N_{--}}^n\oplus (\partial_{+}\partial_{-})H^2\!P^{n}\oplus \partial_{-}^{\ast}\,H^1\!P_{N_{-}}^{n-1};\\
{\rm (iii)}~~& L^2 P^n =L^2P\mathcal{H}_{-}^n \oplus (\partial_{+}\partial_{-})H^2\!P_{D_{++}}^{n}\oplus \partial_{-}^{\ast}\, H^1\!P^{n-1}_{N_{-}}.
\end{align*} 
\end{itemize}
\end{thm}

Similar to the proof of the second-order case, we will introduce two BVPs to prove Theorem \ref{Hdpp}.

\begin{prop}\label{4thE}
The following boundary value problem is self-adjoint and elliptic for any $\beta, \lambda \in P^n$:
\begin{align}\label{BVP3}
&\quad\Delta_{++}\, \beta=\left[(\partial_{+}\partial_{-})^{\ast}(\partial_{+}\partial_{-})+(\partial_{+}\partial_{+}^{\ast})^2\right] \beta  = \lambda\,, \, \qquad\quad\text{on $M$}\\
&\left\{  \begin{aligned}
& \beta \in D_{++}\,,
\\
& \partial_{+}(\rho\;\partial_{+}^{\ast}\beta)=0\,, \\
&\partial_{+}(\rho\; \partial_{+}^{\ast}\partial_{+}\partial_{+}^{\ast} \beta)=0\,,
\end{aligned}~\qquad \text {on $\partial M$}\,.\right.\nonumber
\end{align}
\end{prop}
\begin{proof}
The self-adjoint property is implied by the Green's formulas -- Lemma \ref{green} and Lemma \ref{green2} -- and the property of the boundary condition $D_{++}$, Corollary \ref{D++equ}.  Further, the Laplacian $\Delta_{++}$, is also elliptic (see Sec. \ref{symplap}).  We thus need to proof  that the linear map of condition (2) of Definition \ref{BVPdef}
\begin{align}\label{ode4s}
\Psi_{x, \xi}: \mathcal{M}^{+}_{x,\xi}&\rightarrow G_{1, x}\oplus G_{2, x}\oplus G_{3, x} \oplus G_{4, x}
\end{align}  
is bijective for every $x \in \partial M$ and $\xi \in T^{\ast}_{x}M\!\setminus\!\{0\}$ orthogonal to $d\rho$. 
Here, $\mathcal{M}^{+}_{x,\xi}$ is the space of $\mathbb{R}_+$-bounded solutions of 
\begin{align}\label{ode4}
\sigma_{\Delta_{++}} (\xi - i \, d\rho \tfrac{d}{dt})\,\phi(t)=0\,
\end{align}
and $G_{1, x}\oplus G_{2, x}$ represents the codomain associated with the $D_{++}$ boundary condition which actually consists of two sets of boundary conditions as described in Table \ref{Tab6}.  As will be clear from the calculations below, $\dim  \mathcal{M}^{+}_{x,\xi} = \dim (G_{1, x}\oplus G_{2, x}\oplus G_{3, x} \oplus G_{4, x}) = 2 \dim P^n(T_x^*M)$.  Hence for bijectivity, we just need to show that $\Psi_{x, \xi}$ is injective which we will do so in the following.  

Note that the method of proof here is identical to that for Proposition \ref{2ndE}.  So again, without lost of generality, at the point $x\in \partial M$, we can work in a local Darboux basis $\{w_i\}$ with metric $g_{ij}=\delta_{ij}$, and take $d\rho=w_1$.  And likewise, by local symplectomorphism, it is sufficient for the proof of injectivity to just consider the following three cases: (A1) $\xi = p\, w_2$ for $p\neq 0$; (A2) $\xi = p\, w_2+ q \, w_3$ for both $p\neq 0$ and $q\neq 0$; (B) $\xi= q\, w_3$ for $q>0$.  

{\bf Case (A1):}   Let $\xi = \ep\, p\, w_2$, with $p>0$ and $\ep = \pm 1$.   We decompose the primitive $n$-form $\phi_n(t)$ following \eqref{ldecomp}:
 \begin{align}\label{ldecomp22}
\phi_n(t) =w_1\wedge \phi_{n-1}^1(t)+w_2\wedge \phi_{n-1}^2(t)+\Theta_{12}\wedge\phi_{n-2}^3(t) 
\end{align}
where $\{\phi_{n-1}^1, \phi_{n-1}^2, \phi_{n-2}^3\}$ are primitive forms without $w_1$ and $w_2$ components.

Using the calculation of the principal symbol of $\Delta_{++}$ in \eqref{Dppsyma} of Appendix \ref{AppB}, the general $\mathbb{R}_+$-bounded solution $\phi_n(t) \in \mathcal{M}^{+}_{x, \xi = \ep p w_2}$ for  $\sigma_{\Delta_{++}} ( -i \, w_1 \partial_t + \ep\, p \,w_2)\, \phi_n(t)=0$ can be solved straightforwardly and expressed as 
\begin{align}\label{bs2pp}
\phi_n(t) & = w_1 \w \left(\baa + \bab t \right) e^{-pt} + w_2 \w \left(\bba + \bbb t \right) e^{-pt} \\
 & \quad\quad + \Theta_{12} \w \left(\bca + \bcb t \right) e^{-pt}\,\nonumber 
\end{align} 
where all six $c$'s are constant primitive forms without components in $\{w_1, w_2\}$.

To check injectivity, we look at the kernel of the $\Psi_{x,\xi = \ep p w_2}$ map.  First,
the $D_{++}$ boundary condition, expressed explicitly in local coordinates in Table \ref{Tab6}, give us two sets of constraints:
 \begin{align*}
b_1(-i \,w_1\pa_t+ \ep p w_2)\phi_n(0)=0
&\Longrightarrow \phi^2_{n-1}(0)=0\\
&\Longrightarrow \bba = 0\, \\
b_2(-i \,w_1\pa_t+ \ep p w_2)\phi_n(0)=0&\Longrightarrow -i\,\pa_t\phi^2_{n-1}(0) - \, \ep p \,\phi^1_{n-1}(0) =0\\
&\Longrightarrow  p\bba- \bbb = -i\,\ep p \,\baa\,\\
\end{align*}
The other two boundary conditions in \eqref{BVP3} impose the following:
\begin{align*}
b_3(-i \,w_1\pa_t+ \ep p w_2)\phi_n(0)=0 & \Longrightarrow \sigma_{\partial_{+}}(d\rho)\left[\sigma_{\partial_{+}^{\ast}}(-i \,w_1\pa_t+ \ep p w_2)\,\phi_n(0)\right]\big|_{t=0}=0\\
&\Longrightarrow 
\left\{
\begin{aligned}
&\pa_t\phi^3_{n-2}(0) =0\\
& i\, \pa_t\phi^1_{n-1}(0) -\, \ep p\, \phi^2_{n-1}(0)  =0\\
\end{aligned}
\right.\\
&\Longrightarrow 
\left\{
\begin{aligned}
&p\, \bca - \bcb =0\, \\
& - p\, \baa  + \bab + i\,  \ep p\, \bba =0\,\\
\end{aligned}
\right.\\
b_4(-i \,w_1\pa_t+ \ep p w_2)\phi_n(0)=0&\Longrightarrow \sigma_{\partial_{+}}(d\rho)\left[\sigma_{\partial_{+}^{\ast}\partial_{+}\partial_{+}^{\ast}}(-i \,w_1\pa_t+ \ep p w_2)\,\phi_n(t)\right]=0\\
&\Longrightarrow 
\left\{
\begin{aligned}
&\pa_t\left(-\pa^2_t +p^2\right)\phi^3_{n-2}\big|_{t=0} =
0 \\
&\left[-i\pa_t\left(-\pa_t^2 + p^2\right)\phi^1_{n-1} + \ep p \left(p^2- \pa^2_t\right) \phi^2_{n-1}\right]\big|_{t=0}=0 \\
\end{aligned}
\right.\\
&\Longrightarrow 
\left\{
\begin{aligned}
&\bcb =0\\\
& i\, \bab + \ep  \bbb =0 \\
\end{aligned}
\right.
\end{align*} 
Altogether, the six boundary condition equations for $\{\phi_{n-1}^1, \phi_{n-1}^2, \phi_{n-2}^3\}$ above requires the vanishing of all six constant $c$'s, and therefore, $\Psi_{x,\xi=\ep p w_2}$ is injective.

{\bf {Case (A2):}} For this case where $\xi= p\, w_2 + q\, w_3$ with both $p\neq0$ and $q\neq0$, the proof of injectivity is given in Appendix \ref{AppC2}.

{\bf {Case (B):}}  Let $\xi = q\, w_3$ (for some $q>0$).  The relevant decomposition which extracts out the dependence on $\{w_1, w_2, w_3, w_4\}$ has 10 terms:
\begin{align}\label{bexpand4}
\phi_n(t) &=w_1\wedge \phi_{n-1}^1(t)+w_2\wedge \phi_{n-1}^2(t)+\Theta_{12}\wedge\phi_{n-2}^3(t) \nonumber\\
&= w_1 \w \left[w_3 \w \ga^{13}_{n-2}(t) + w_4\w \ga^{14}_{n-2}(t) + \Thp \w\ga^{134}_{n-3}(t) \right] \nonumber\\
& ~~ +w_2\w \left[w_3  \w \ga^{23}_{n-2}(t) + w_4\w \ga^{24}_{n-2}(t) + \Thp \w\ga^{234}_{n-3}(t) \right] \\
& \quad +\Th \w \left[w_3 \w\ga^{123}_{n-3}(t) + w_4 \w\ga^{124}_{n-3}(t) + \Thp \w\ga^{1234}_{n-4}(t) + \ga^{12}_{n-2}(t)\right] \nonumber
\end{align}
where the 10 $\gamma(t)$'s are forms without components in $\{w_1, w_2, w_3, w_4\}$  and   $\Thp$ is as before the two-form in \eqref{Thpdef4}.
To find the the general  $\mathbb{R}_+$-bounded solution $\phi_n(t) \in \mathcal{M}^{+}_{x, \xi = q w_3}$, we write down the system of ordinary differential equations $\sigma_{\Delta_{++}} ( -i \, w_1 \partial_t + q \,w_3)\, \phi_n(t)=0$  for the 10 $\ga$'s using the calculation of the principal symbol $\sigma_{\Delta_{++}}(\zeta_1 w_1 + \zeta_2 w_2 + \zeta_3 w_3)$ in Appendix \ref{AppB}, replacing $(\zeta_1, \zeta_3, \zeta_3)=(-i \pa_t, 0, q)$.  We express the general solution of  $\mathcal{M}^{+}_{x, \xi = q w_3}$ below in terms of 20 constant primitive forms, labelled by $c^l$ for $l=1, \ldots, 20$, that have no components in $\{\wa, \wb, \wc, \wdd\}$:
\begingroup
\renewcommand*{\arraystretch}{1.25}
\begin{align*}
\ga_{n-2}^{13}&= (c_{n-2}^{1}+c_{n-2}^{2}t) e^{-qt},\,
\ga_{n-2}^{24}=(c_{n-2}^{3}+c_{n-2}^{4}t)e^{-qt} ,\,
\ga_{n-4}^{1234}= (c_{n-4}^{5}+c_{n-4}^{6}t) e^{-qt}\\
& \\
%
\begin{pmatrix} 
 \ga_{n-3}^{134} \\
\ga_{n-3}^{123}
\end{pmatrix}&=c_{n-3}^7\begin{pmatrix}1 \\ 0 \end{pmatrix}e^{-qt}
+c_{n-3}^8\begin{pmatrix}0 \\ 1 \end{pmatrix}e^{-qt}
+c_{n-3}^9\begin{pmatrix} \frac{13}{20}\frac{t}{q} +\frac{3}{20} t^2\\  -\frac{3i}{20}\frac{t}{q} -\frac{3i}{20} t^2\end{pmatrix}e^{-qt}+
c_{n-3}^{10}\begin{pmatrix}-\frac{3i}{20} t-\frac{3i}{20} t^2\\  \frac{7}{20}\frac{t}{q}-\frac{3}{20} t^2\end{pmatrix}e^{-qt}\\
%
\begin{pmatrix} 
 \ga_{n-3}^{234}\\
 \ga_{n-3}^{124}
\end{pmatrix}&=c_{n-3}^{11}\begin{pmatrix}1 \\ 0 \end{pmatrix}e^{-qt}
+c_{n-3}^{12}\begin{pmatrix}0 \\ 1 \end{pmatrix}e^{-qt}
+c_{n-3}^{13}\begin{pmatrix} \frac{13}{20}\frac{t}{q} +\frac{3}{20} t^2\\  -\frac{3i}{20}\frac{t}{q} -\frac{3i}{20} t^2\end{pmatrix}e^{-qt}
+c_{n-3}^{14}\begin{pmatrix}-\frac{3i}{20}\frac{t}{q} -\frac{3i}{20} t^2\\  \frac{7}{20}\frac{t}{q} -\frac{3}{20} t^2\end{pmatrix}e^{-qt}\\
%
\begin{pmatrix} 
 \ga_{n-2}^{14}\\ \ga_{n-2}^{23}\\ \ga_{n-2}^{12}
\end{pmatrix}&=(c_{n-2}^{15}+c_{n-2}^{16}t)\begin{pmatrix}1 \\ 1\\ 0 \end{pmatrix}e^{-qt} 
+(c_{n-2}^{17}+c_{n-2}^{18}t)\begin{pmatrix}-1 \\ 1\\ -i \end{pmatrix}e^{-qt}\\
&\qquad +c_{n-2}^{19}\begin{pmatrix}\frac{5}{3q^2} - t^2 \\ -\frac{5}{3q^2} + t^2\\ -it^2 \end{pmatrix}e^{-qt}
+c_{n-2}^{20}\begin{pmatrix}\frac{5}{q^3}+\frac{5t}{q^2}- t^3 \\ -\frac{5}{q^3} - \frac{5t}{q^2}+ t^3\\ -it^3 \end{pmatrix}e^{-qt}.
\end{align*}
\endgroup

For injectivity, the image of the $\Psi_{x, \xi=qw_3}$ map vanishes imposes the following conditions: 
\begin{align*}
b_1(-i\, w_1\partial_t + q w_3)\phi_n(t)\big|_{t=0}\,=0&\Longrightarrow \phi^2_n(0)=0,\\
&\Longrightarrow \ga^{23}(0)=0,\  \ga^{24}(0)=0,\ \ga^{234}(0)=0, \\
b_2(-i\, w_1\partial_t + q w_3)\phi_n(t)\big|_{t=0}\,=0
&\Longrightarrow  \pa_t \phi_{n-1}^2(0) + 2\,i\, \Pi'\left(qw_3 \w \phi_{n-2}^3(0)\right) =0,\\
&\Longrightarrow \left\{\begin{aligned}
&\pa_t\ga^{23}(0) + 2\,i\,q\, \ga^{12}(0) =0,\\
&\pa_t \ga^{24}(0)=0,\\
&\pa_t \ga^{234}(0) + i\,q\, \ga^{124}(0) =0,
\end{aligned}
\right.
\end{align*}
\begin{align*}
b_3(-i\, w_1\partial_t + q w_3)\phi_n(t)\big|_{t=0}\,=0
&\Longrightarrow \left\{
\begin{aligned}
&\pa_t\ga^{13}(0)=0,\ \pa_t\ga^{123}(0)=0,\ \pa_t\ga^{1234}(0)=0,\\
&i\pa_t\ga^{124}(0)+ q \ga^{234}(0) =0,\\
&i\pa_t\ga^{12}(0)+q\ga^{23}(0)=0,\\
&i\pa_t\ga^{14}(0)+q\ga^{12}(0)=0,\\
&i\pa_t\ga^{134}(0)-\tfrac{q}{2}\ga^{123}(0)=0
\end{aligned}
\right.\\
b_4(-i\, w_1\partial_t + q w_3)\phi_n(t)\big|_{t=0}\,=0
&\Longrightarrow \left\{
\begin{aligned}
&\pa_t(\pa_t^2 - q^2)\ga^{1234}\big|_{t=0}=0,\\
&\pa_t(\pa_t^2-q^2)\ga^{13}\big|_{t=0}=0,\\
&\left\{\tfrac{1}{2}q(\pa_t^2-q^2)\ga^{234} + \tfrac{i}{2}\pa_t(\pa_t^2 -q^2)\ga^{124}\right\}\big|_{t=0}=0,\\
&\left\{\tfrac{1}{2}q\pa_t^2 \ga^{134} - i(\tfrac{1}{2}\pa_t^3 -q^2 \pa_t)\ga^{123}
\right\}\big|_{t=0}=0,\\
&\left\{(i\pa_t^3 - \tfrac{3}{4} q^2\pa_t)\ga^{134} - (\tfrac{3}{4}q\pa_t^2 - \tfrac{1}{2}q^3)\ga^{123}
\right\}\big|_{t=0}=0,\\
&\left\{(i\pa_t^3 -\tfrac{q^2}{2}\pa_t)\ga^{14} -\tfrac{i}{2}q^2\pa_t\ga^{23} + (\tfrac{3}{2}q \pa_t^2 -\tfrac{1}{2} q^3)\ga^{12}
\right\}\big|_{t=0}=0,\\
&\left\{\tfrac{1}{2}q\pa_t^2 \ga^{14} +(\tfrac{q}{2}\pa_t^2 - q^3)\ga^{23} +\tfrac{i}{2}(\pa_t^3 - 3q^2 \pa_t)\ga^{12}
\right\}\big|_{t=0}=0.\\
\end{aligned}
\right.
\end{align*}
It can then be straightforwardly checked that the above 20 boundary condition equations on the 10 $\ga$'s are only satisfied if the twenty primitive constants, $c^l$ for $l=1, \dots, 20$, of the general $\mathbb{R}_+$-bounded solution all vanish, thus proving that the map $\Psi_{x, \xi=qw_3}$ is injective.
\end{proof}

By a similar proof, the following is also true.

\begin{prop}\label{4thF}
The following boundary value problem is self-adjoint and elliptic for any $\beta, \lambda \in P^n$:
\begin{align}\label{BVP4}
&\quad\Delta_{++}\, \beta=\left[(\partial_{+}\partial_{-})^{\ast}(\partial_{+}\partial_{-})+(\partial_{+}\partial_{+}^{\ast})^2\right] \beta  = \lambda\,, \, \qquad\quad\text{on $M$}\\
&\left\{  \begin{aligned}
& \partial_{+}^\ast(\rho\; \beta)=0\,, \\
& \partial_{+}^\ast(\rho\; \partial_{+}\partial_{+}^{\ast} \beta)=0\,,\\
& \dpp\dpm\beta\in N_{--}\,,
\end{aligned}~\qquad \text {on $\partial M$}\,.\right.\nonumber
\end{align}
\end{prop}

With the help of the above two BVPs, we can derive the decompositions in Theorem \ref{Hdpp} following very similar arguments as that in the proof of Theorem \ref{Hdp}. For brevity, we will not write out the details.  The key here is that the BVP in \eqref{BVP3} implies the first decomposition and the BVP in \eqref{BVP4} implies the second one in Theorem \ref{Hdpp}. The third decomposition follows by combining the first two decompositions.

\subsection{Harmonic fields and boundary value problems}  The  Hodge decompositions in Section \ref{hodgede} can be applied to solve various boundary value problems.  We begin first with the Poincar\'e lemmas.

\begin{lemma}[Poincar\'e lemma for $\partial_{+}$]
Let $(\omega, J, g)$ be a compatible triple on a compact symplectic manifold with boundary.  Given a primitive form, $\lambda \in P^k$ with $k<n\,$, there exists a solution $\beta \in P^{k-1}$ to the equation  
\begin{equation*}
\partial_{+}\beta=\lambda
\end{equation*}
if and only if $\lambda$ satisfies the integrability conditions:
\begin{equation} \label{integ1}
\partial_{+}\lambda=0\, \qquad  \text{and}\qquad (\lambda, \gamma)=0\, ~\text{ for all}~~ \gamma \in P\mathcal{H}_{+, N_{+}}^k\,.
\end{equation} 
\end{lemma}
\begin{proof}
For any $\lambda \in P^k$ with $k<n$, if $\lambda= \partial_{+} \beta$, then clearly $\lambda$ satisfies the integrability conditions of \eqref{integ1}.  For the converse statement, we make use of the decomposition 2.(ii) of Theorem \ref{Hdp} to express
\[
\lambda=\gamma'+\partial_{+}\beta+\partial_{+}^{\ast}\varphi.
\] 
for some $\gamma' \in P\mathcal{H}^k_{+, N_{+}},\, \beta \in P^{k-1}\,$, and $\varphi \in P_{N_{+}}^{k+1}\,$.
The first integrability condition $\dpp \lambda =0$ implies $\partial_{+}^{\ast}\varphi=0$ since
\begin{equation*}
0=(\dpp \lambda, \varphi) = (\partial_{+}\partial_{+}^{\ast}\varphi, \varphi)=(\partial_{+}^{\ast}\varphi, \partial_{+}^{\ast}\varphi)~.
\end{equation*}
The condition $(\lambda, \gamma)=0$ for any $\gamma\in P\mathcal{H}^k_{+, N_{+}}$ implies that $\gamma'=0$ since we can just set $\gamma = \gamma'$ and this would result in $(\lambda, \gamma')=(\gamma', \gamma')=0$. Therefore, $\lambda=\partial_{+}\beta$. 
\end{proof}
Similarly, we have the following Poincar\'e lemmas for the other symplectic differential operators, which we write down here for completeness.
\begin{lemma}[Poincar\'e lemma for $\partial^{\ast}_{+}$]
Given a $\lambda \in P^k$ with $k<n\,$, there exists a solution $\beta \in P^{k+1}$ to the equation
\begin{equation*}
\partial_{+}^{\ast}\,\beta=\lambda
\end{equation*}  
if and only if $\lambda$ obeys the integrability conditions:
\begin{equation*}
\partial_{+}^{\ast}\lambda=0\, \qquad \text{and}\qquad (\lambda, \gamma)=0\,~ \text{ for all}~~ \gamma \in P\mathcal{H}_{+, D_{+}}^k\,.
\end{equation*} 
\end{lemma}
\begin{lemma}[Poincar\'e lemma for $\partial_{-}$]
Given a $\lambda \in P^k$ and $k<n$, there exists a solution $\beta\in P^{k+1}$ to the equation 
\begin{equation*}\partial_{-}\,\beta=\lambda
\end{equation*}  if and only if $\lambda$ obeys the integrability conditions:
\begin{equation*}
\partial_{-}\lambda=0\, \qquad \text{and}\qquad (\lambda, \gamma)=0\, ~\text{ for all}~~ \gamma \in P\mathcal{H}_{-, N_{-}}^k\,.
\end{equation*} 
\end{lemma}
\begin{lemma}[Poincar\'e lemma for $\partial_{-}^{\ast}$]
Given a $\lambda \in P^k$ and $k<n$, there exists a solution $\beta\in P^{k-1}$ to the equation 
\begin{equation*}
\partial_{-}^{\ast}\,\beta=\lambda
\end{equation*}  if and only if $\lambda$ obeys the integrability conditions: 
\begin{equation*}
\partial_{-}^{\ast}\lambda=0\, \qquad \text{and}\qquad (\lambda, \gamma)=0\,~ \text{ for all}~~ \gamma \in P\mathcal{H}_{-, D_{-}}^k\,.
\end{equation*} 

\end{lemma}

\begin{lemma}[Poincar\'e lemma for $\partial_{+}\partial_{-}$]
Given a $\lambda \in P^n$ , there exists a solution $\beta\in P^{n}$ to the equation 
\begin{equation*}
\partial_{+}\partial_{-}\,\beta=\lambda
\end{equation*}  if and only if $\lambda$ obeys the integrability conditions: 
\begin{equation*}
\partial_{-}\lambda=0\, \qquad \text{and}\qquad (\lambda, \gamma)=0\,~ \text{ for all}~~ \gamma \in P\mathcal{H}_{-, N_{--}}^n\,.
\end{equation*} 

\end{lemma}
\begin{lemma}[Poincar\'e lemma for $(\partial_{+}\partial_{-})^{\ast}$]
Given a $\lambda \in P^n$ , there exists a solution $\beta\in P^{n}$ to the equation 
\begin{equation*}
(\partial_{+}\partial_{-})^{\ast}\,\beta=\lambda
\end{equation*}  if and only if $\lambda$ obeys the integrability conditions: 
\begin{equation*}
\partial_{+}^{\ast}\lambda=0\, \qquad \text{and}\qquad (\lambda, \gamma)=0\,~ \text{ for all}~~ \gamma \in P\mathcal{H}_{+, D_{++}}^n\,.
\end{equation*} 

\end{lemma}

Another application of the Hodge decompositions in Section 4.2 is to show by studying certain BVPs that the spaces of harmonic fields, $P\mathcal{H}_{+}^k$ and $P\mathcal{H}_{-}^k$ , are infinite-dimensional if no boundary condition is imposed.  For simplicity, we will just describe the $k<n$ case below.

\begin{prop}\label{bvpi}
Given a pair of primitive forms, $\lambda \in P^k$  and $\psi \in P^{k-1}$, with $k<n$, there exists a solution $\beta \in P^{k-1}$of the boundary value problem 
\begin{align*}
\partial_{+}\beta&=\lambda\, \quad\quad \quad \quad \text{~on~} M \\
\partial_{+}(\rho\, \beta)&=\partial_{+}(\rho\, \psi) \qquad \!\text{on~}\partial M
\end{align*}
if and only if $\lambda$ and $ \psi$ obey the integrability conditions:  
\begin{equation}
\partial_{+} \lambda=0\, \qquad \text{and}\qquad (\lambda, \gamma)=\int_{\partial M}\langle \partial_{+}(\rho\,  \psi), \gamma\rangle\,dS ~ \text{ for all}~~ \gamma \in P\mathcal{H}_{+}^k\,.
\end{equation} 
Moreover, the solution $\beta$ can be chosen to satisfy $\partial_{+}^{\ast} \beta=0$.
\end{prop}
\begin{proof}
If there exists a solution $\beta\in P^{k-1}$ to the above BVP, then clearly $\lambda$ and $\psi$ satisfy the integrability conditions.  Conversely, for $k<n$, we first decompose $\lambda$ by the Hodge decomposition 2.(iii) of Theorem \ref{Hdp} and write
\begin{equation*}
\lambda=\nu+\partial_{+}\varphi+\partial_{+}^{\ast}\sigma.
\end{equation*} 
where $\nu \in P\mathcal{H}^k_{+}\,,\, \varphi \in P^{k-1}_{D_{+}}\,$, and $\sigma \in P^{k+1}_{N_{+}}\,$.  The first integrability condition $\partial_{+}\lambda=0$ gives the condition that $\partial_{+}\partial_{+}^{\ast}\sigma=0\,$, which implies $\partial_{+}^{\ast}\sigma=0$ since
\[0=(\partial_{+}\partial_{+}^{\ast}\sigma, \sigma)=(\partial_{+}^{\ast}\sigma, \partial_{+}^{\ast}\sigma)\,.\] 
The second integrability condition with the presence of $\psi$ does not imply $\nu=0$.  Let us introduce another primitive form $\widetilde{\psi}\in P^{k-1}$ with the property that 
\begin{equation}\label{tpsi}
\partial_{+}(\rho\,\widetilde{\psi})~|_{\partial M}=\partial_{+}(\rho\, \psi)~|_{\partial M} \qquad  \text{and~} \qquad \partial_{+}^{\ast}\widetilde{\psi}=0\,.
\end{equation} 
This is possible since by the Hodge decomposition 2.(i) of Theorem \ref{Hdp}, we can write 
\[ \psi  = \nu_{\psi} + \dpp \varphi_{\psi} + \partial_{+}^{\ast}\sigma_{\psi} \]
where $\nu_{\psi} \in P\mathcal{H}^{k-1}_{+, D_+},\, \varphi_{\psi} \in P^{k-2}_{D_{+}}$, and $\sigma_{\psi} \in P^{k}$.  Since $\dpp \varphi_{\psi} \in D_+\,$ by Proposition \ref{pre}, we can simply set $\widetilde{\psi}= \psi - \dpp \varphi_{\psi}$ which then satisfies the two conditions in \eqref{tpsi}.

Let $\widetilde{\lambda}=\partial_{+}\widetilde{\psi}$ and again Hodge decompose $\widetilde{\lambda}$ as we did above for $\lambda$:
\[
\widetilde{\lambda}=\widetilde{\nu}+\partial_{+}\widetilde{\varphi}
\] where $\widetilde{\nu} \in P\mathcal{H}_{+}^k$ and $\widetilde{\varphi} \in P^{k-1}_{D_{+}}$.  We can now define  $\beta=\varphi+\widetilde{\psi}-\widetilde{\varphi}\,$ which satisfies
\begin{align*}
&\partial_{+}\beta=\lambda+\widetilde{\nu}-\nu\,,\\
&\partial_{+}(\rho\, \beta)~|_{\partial{M}}=\partial_{+}(\rho\, \psi)~|_{\partial{M}}\,.
\end{align*}
The second integrability condition that for any $\gamma \in P\mathcal{H}_{+}^k\,$,  $(\lambda, \gamma) = ( \partial_{+}\beta - ( \widetilde{\nu}-\nu), \gamma) = \int_{\partial M}  \langle \partial_{+}(\rho\,  \psi), \gamma\rangle\,$ further implies
$\widetilde{\nu}-\nu = 0$.  Hence, $\beta$ is the solution for the boundary value problem.  Furthermore, $\varphi$ and $\widetilde{\varphi}$ can be chosen to be $\partial_{+}^{\ast}$-closed just as we argued for the existence of $\widetilde{\psi}$ above.  Therefore, $\beta$ can satisfy  $\partial_{+}^{\ast} \beta =0\,$ as well.
\end{proof}
The BVP of Proposition \ref{bvpi} can be easily modified to consider the $\dpm$ operator instead of $\dpp$, and also, the dual operators $\dpps$ and $\dpms$ as well.   For instance, the statement for the dual $\dpps$ would be as follows:  
\begin{cor}\label{bvpii}
Given a pair of primitive forms, $\lambda \in P^{k-1}$  and $\psi \in P^{k}$, with $0<k<n$, there exists a solution $\beta \in P^{k}$of the boundary value problem 
\begin{align*}
\partial_{+}^{\ast}\beta&=\lambda\, \quad\quad \quad \quad \text{~on~} M \\
\partial_{+}^{\ast}(\rho\, \beta)&=\partial_{+}^{\ast}(\rho\, \psi) \qquad \!\text{on~}\partial M
\end{align*}
if and only if $\lambda$ and $ \psi$ obey the integrability conditions:  
\begin{equation} \label{ibvpi}
\partial_{+}^{\ast} \lambda=0\, \qquad \text{and}\qquad (\lambda, \gamma)=\int_{\partial M}\langle \partial_{+}^{\ast}(\rho\,  \psi), \gamma\rangle\,dS ~ \text{ for all}~~ \gamma \in P\mathcal{H}_{+}^{k-1}\,.
\end{equation} 
Moreover, the solution $\beta$ can be chosen to satisfy $\partial_{+} \beta=0$.
\end{cor}

We now use Corollary \ref{bvpii} to prove that the space of harmonic fields without imposing any boundary condition is infinite-dimensional.
\begin{thm}\label{inf}
On a compact symplectic manifold $(M, \omega, J, g)$ with smooth boundary, the space $P\mathcal{H}^k_{+}$ and $P\mathcal{H}^k_{-}$ are infinite-dimensional for $0<k<n$.
\end{thm}
\begin{proof} 
For $0<k< n$, let us consider the boundary map
\begin{align*}
\mathcal{B}:~~ P\mathcal{H}^k_{+} &\longrightarrow \quad\Omega^{k-1}~|_{\partial M}\\ 
\beta \quad &\longrightarrow \quad\partial_{+}^{\ast}(\rho\, \beta)~|_{\partial M}.
\end{align*}  
By the definition of $N_{+}$ in Definition \ref{sbc1def} (see also Remark \ref{diffequiv}), it is clear that $\mathcal{B}(\beta)=0$ if and only if $\beta\in N_{+}\,$.  Therefore, $\ker \mathcal{B} = P\mathcal{H}^k_{+, N_{+}}$, which is finite-dimensional as stated in Theorem \ref{Hdp}.

Further, we can show that the map $\mathcal{B}$ is surjective to the space $\partial_{+}^{\ast}(\rho\,\partial_{+}^{\ast}P^{k+1})~|_{\partial M}$. That is, for any $\psi \in  \partial_{+}^{\ast}P^{k+1}$, there is a $\beta \in P\mathcal{H}^k_{+}$ such that
\begin{align*}
&\partial_{+}\beta=0\,, \quad\partial_{+}^{\ast}\beta=0\,, \quad\text{on $M$}\\
&\partial_{+}^{\ast}(\rho\, \beta)=\partial_{+}^{\ast}(\rho\, \psi)\,, \quad \text{ on $\partial M$}.
\end{align*} 
From Corollary \ref{bvpii}, such a $\beta$ exists as long as the two integrability conditions in \eqref{ibvpi} are satisfied.  The first trivially holds since we are only interested in the $\lambda=0$ case.  The second gives the condition
\begin{equation}\label{2integ}
(\lambda, \gamma)=\int_{\partial M}\langle\partial_{+}^{\ast}(\rho\, \psi),\gamma \rangle\,dS = \int_{M}\langle\partial_{+}^{\ast}\psi,\gamma \rangle\,dS \,,
\end{equation}  
for any $\gamma \in P\mathcal{H}^{k-1}_{+}$ when $0<k< n$.  Clearly, this holds as well since here $\psi \in \partial_{+}^{\ast}P^{k+1}$ which thus results in a zero on both sides of \eqref{2integ}.  With the kernel  of $\mathcal{B}$ being finite-dimensional while $\partial_{+}^{\ast}(\rho\,\partial_{+}^{\ast}P^{k+1})~|_{\partial M}$ is infinite-dimensional, we therefore conclude that $P\mathcal{H}^k_{+}$ for $0<k < n$ must be infinite-dimensional. 

Concerning $P\mathcal{H}^k_{-}$, we can make use of the operator $\mathcal{J}$ defined in Section \ref{conjsec}.  By Lemma \ref{conjugate}, $\mathcal{J}$ maps the conditions of  $P\mathcal{H}^k_{+}$ into the conditions of $P\mathcal{H}^k_{-}$, and hence, it is an isomorphism between the two spaces.  This implies that $P\mathcal{H}^k_{-}$ for $0 < k < n$ is infinite-dimensional.
\end{proof}

\section{Symplectic cohomology }
In this section, we study absolute and relative primitive cohomologies on compact symplectic manifolds with boundary.

\subsection{Absolute primitive cohomologies }Recall the symplectic elliptic complex reviewed in Section 2:
\begin{align}\label{pecomp}\begin{CD}
0@>\partial_{+}>>P^0@>\partial_{+}>>P^1@>\partial_{+}>>\cdots @>\partial_{+}>>P^{n-1}@>\partial_{+}>>P^n\\
    @. @. @. @.   @.                                                                                                            @VV{\partial_{+}\partial_{-}}V\\
0@<\partial_{-}<<P^0@<\partial_{-}<<P^1@<\partial_{-}<<\cdots @<\partial_{-}<<P^{n-1}@<\partial_{-}<<P^n
\end{CD}
\end{align} Tseng and Yau studied the cohomologies of this complex in \cite{TY2}, which we shall write as follows:
\begin{align*}
PH^k_+(M)&=\dfrac{\ker \partial_{+} \cap P^k(M)}{ \partial_{+} P^{k-1}(M)}\,, \quad \text{for $k=0, 1, 2, \ldots, n-1 \,$},\\
PH^n_+(M)&=\frac{\ker \partial_{+}\partial_{-}\cap P^n(M)}{\partial_{+}P^{n-1}(M)}\,,\\
PH^n_-(M)&=\frac{\ker \partial_{-}\cap P^{n}(M)}{\partial_{+}\partial_{-}P^{n}(M)}\,,\\
PH^k_-(M)&=\frac{\ker \partial_{-}\cap P^{k}(M)}{\partial_{-} P^{k+1}(M)}\,, \quad\text{for $k=0, 1, 2 , \ldots, n-1 \,$}.
\end{align*} 
On closed manifolds, the ellipticity of the complex \eqref{pecomp} implies that the above cohomologies are finite-dimensional. (For their properties in the closed manifold case, see \cite{TY2,TTY}.)  In fact, the finite-dimensionality extends to the case of manifolds with boundary as we explained in the below proposition, where we also give a simple algebraic proof that the index of the elliptic complex \eqref{pecomp} is always zero. 
\begin{prop}\label{fv}
On a compact symplectic manifold with boundary, the corresponding cohomologies of primitive elliptic complex of \eqref{pecomp} are finite-dimensional and the index of the complex is zero.
\end{prop}
\begin{proof}
We recall the following isomorphisms from \cite{TTY} which hold on symplectic manifolds with boundary:
\begin{align}
PH^{k}_{+}(M)&\cong {\rm coker}[L\!: {H^{k-2}(M) \to H^k(M)}] \oplus \ker [L\!: {H^{k-1}(M)}\to H^{k+1}(M)]\label{fva}\\
PH^{k}_{-}(M)&\cong {\rm coker}[L\!: {H^{2n-k-1}(M)\to H^{2n-k+1}(M)}]  \label{fvb}\\
&\qquad\quad\oplus \ker [L\!: {H^{2n-k}(M)} \to H^{2n-k+2}(M)]\nonumber
\end{align}
Since the de Rham cohomology $H^{\ast}(M)$ is finite-dimensional for a manifold with boundary, the kernels and the cokernels of $L\!: H^{\ast}(M) \rightarrow H^{\ast}(M)$ are also finite-dimensional.  Therefore, the isomorphisms \eqref{fva}-\eqref{fvb} above imply that $PH^{k}_+(M)$ and $PH^{k}_-(M)$ are both finite-dimensional, for $0 \leq k \leq n$.

Consider the index of this complex:
\[
{\rm index}= \overset{n}{\underset{k=0}{\sum}}(-1)^k {\rm dim}\, PH^{k}_+(M)-\overset{n}{\underset{k=0}{\sum}}(-1)^k{\rm dim}\,PH^k_-(M).
\]
Since the Lefschetz map is a linear map on $H^*(M)$, we have the linear relation
\begin{align}\label{linrel}
\dim {\rm coker}\, L|_{H^j} - \dim {\rm ker}\, L|_{H^j} = \dim H^{j+2} - \dim H^j~.
\end{align}
Together with the isomorphism \eqref{fva} above, this imply
\begin{align*}
{\rm dim}\, PH^k_+& = {\rm dim}\, {\rm coker}\, L|_{H^{k-2}} + {\rm dim} \ker L|_{H^{k-1}} \\
& = {\rm dim} \,H^{k}- {\rm dim}\,H^{k-2}+{\rm dim} \ker \,L|_{H^{k-2}}+ {\rm dim} \ker L|_{H^{k-1}}
\end{align*}
Note that for $k=0,1$, this gives
 \begin{align*}
 {\rm dim}\,PH^0_+&={\rm dim} \,H^0\,, 
\\ {\rm dim}\,PH^1_+&={\rm dim}\,H^1+{\rm dim}\ker L|_{H^{0}}~.
 \end{align*}
The alternating sum of $\dim PH_+^k$ results in 
 \begin{align}\label{ppsum}
 \overset{n}{\underset{k=0}{\sum}}(-1)^k {\rm dim}\, PH^{k}_+=(-1)^{n}\left({\rm dim}\,H^n+ {\rm dim}\,\ker L|_{H^{n-1}} \right) + (-1)^{n-1}{\rm dim}\,H^{n-1}.
 \end{align} 
Similarly, for $PH^k_-$, we have
\begin{align*}
\dim PH^k_-& = {\rm dim}\, {\rm coker}\, L|_{H^{2n-k-1}} + {\rm dim} \ker L|_{H^{2n-k}} \\
& = {\rm dim}\, {\rm coker}\, L|_{H^{2n-k-1}} +{\rm dim}\, {\rm coker}\, L|_{H^{2n-k}} + {\rm dim} \,H^{2n-k}- {\rm dim}\,H^{2n-k+2}
\end{align*}
with 
\begin{align*}
{\rm dim}\,PH^0_-&={\rm dim} \,H^{2n}\,,\\
{\rm dim}\,PH^1_-&={\rm dim}\,H^{2n-1}+{\rm dim}\,{\rm coker}\, L|_{H^{2n-2}}~. \end{align*} 
This results in the alternating sum
\begin{align}\label{pmsum}
\overset{n}{\underset{k=0}{\sum}}(-1)^k {\rm dim} \,PH^{k}_-=(-1)^{n}\left({\rm dim}\,{\rm coker} \, L|_{H^{n-1}} + {\rm dim}\,H^{n}\right) + (-1)^{n-1}{\rm dim}\,H^{n+1}.
\end{align} 
Subtracting \eqref{pmsum} from \eqref{ppsum} and then applying again the relation \eqref{linrel}, we obtain that the index is zero.
\end{proof}

Now for each primitive absolute cohomology, we can identify a unique harmonic field representative for each cohomology class.   This follows immediately from the following Hodge decompositions for $k<n\,$,
\begin{align*}
P^k&=P\mathcal{H}_{+, N_{+}}^k\!\oplus \,\partial_{+}P^{k-1}\oplus \,\partial_{+}^{\ast}P_{N_{+}}^{k+1}~,\\
P^k&=P\mathcal{H}_{-,N_{-}}^k\!\oplus\, \partial_{-}P^{k+1}\oplus\, \partial_{-}^{\ast}P_{N_{-}}^{k-1}~,
\end{align*}
from Theorems \ref{Hdp}.2.(ii) and \ref{Hdm}.2.(ii), respectively, and in the case of $k=n\,$,
\begin{align*}
P^n &=P\mathcal{H}_{+,N_{+}}^n\oplus \partial_{+}P^{n-1}\oplus (\partial_{+}\partial_{-})^{\ast}P_{N_{--}}^{n}~,\\
P^n &=P\mathcal{H}_{-,N_{--}}^n\!\!\oplus (\partial_{+}\partial_{-})P^{n}\oplus \partial_{-}^{\ast}P_{N_{-}}^{n-1}~,
\end{align*}
from Theorems \ref{Hdpp}.2.(ii) and \ref{Hdmm}.2.(ii).  These four decompositions immediately gives an isomorphism between absolute primitive cohomology and the space of harmonic fields with $\{N_+, N_-, N_{--}\}$ boundary conditions.
\begin{thm}\label{absoluteN}
Let $(M, \omega)$ be a compact symplectic manifold with a smooth boundary. Let $(\omega, J, g)$ be a compatible triple on $M$. Then there are isomorphisms:  
 \begin{align}\label{abhar}
PH^k_+(M)&\cong P\mathcal{H}^k_{+, N_{+}}\!(M)\,,\quad PH^k_-(M)\cong P\mathcal{H}^k_{-,N_{-}}\!(M)\,,
\end{align}
for $k<n$ and 
\begin{align}\label{abhar2}
PH^n_+(M)&\cong P\mathcal{H}^n_{+, N_{+}}\!(M)\,,\quad PH^n_-(M)\cong P\mathcal{H}^n_{-,N_{--}}\!(M)\,.
\end{align}
\end{thm}
Note that Theorem \ref{absoluteN} also implies the finiteness of the absolute primitive cohomologies since the spaces of harmonic fields on the right hand side of the isomorphisms in \eqref{abhar}-\eqref{abhar2}  are all finite-dimensional following Theorems \ref{Hdp}, \ref{Hdm}, \ref{Hdpp}, \ref{Hdmm}. More noteworthily, the above isomorphisms demonstrate that the dimensions of $P\mathcal{H}^k_{+, N_{+}}\!(M)$, $P\mathcal{H}^k_{-, N_{-}}\!(M)$, for $k< n$, and the dimensions of $P\mathcal{H}^n_{+, N_{+}}\!(M)$, $P\mathcal{H}^n_{-,N_{--}}\!(M)$ are all symplectic invariants and independent of the metric needed to define harmonic fields. In fact, the dimensions of the primitive harmonic fields with Dirichlet-type boundary conditions are also symplectic invariants.  This follows from Lemmas \ref{conjugate} and \ref{rel} which imply that the operator $\mathcal{J}$ induces the following isomorphisms on harmonic fields:
\begin{align}\label{harma}
P\mathcal{H}^{k}_{+, D_{+}}\!(M)\cong P\mathcal{H}^{k}_{-, N_{-}}\!(M)\,, \quad P\mathcal{H}^{k}_{-, D_{-}}\!(M)\cong P\mathcal{H}^{k}_{+, N_{+}}\!(M)\,,
\end{align}
for degree $k<n$ and 
\begin{align}\label{harmb}
P\mathcal{H}^{n}_{+, D_{++}}\!(M)\cong P\mathcal{H}^{n}_{-, N_{--}}\!(M)\,, \quad P\mathcal{H}^{n}_{-, D_{-}}\!(M)\cong P\mathcal{H}^{n}_{+, N_{+}}\!(M)\,.
\end{align}
Therefore, the space of harmonic fields with symplectic boundary conditions, i.e. $D_{\pm}, N_{\pm}, D_{++}$, and $N_{--}$, represent symplectic invariants.

\subsection{Relative primitive cohomologies}\label{RelSec}

For manifolds with boundary, the de Rham complex can be restricted to forms that satisfy the Dirichlet boundary condition
\begin{align*}
\begin{CD}
0@>>>\Omega_{D}^0@>d>>\Omega_{D}^1@>d>>\Omega_{D}^{2}@>d>>~~\cdots   
\end{CD}
\end{align*} 
The cohomology associated with this elliptic complex, 
$$H^k(M, \partial M) = \dfrac{\ker d \cap \Omega_D^k}{d\Omega_D^{k-1}}\,, \quad \text{for}~~ k=0, 1, \ldots, 2n\,,  $$
is called the relative cohomology with respect to the boundary since $\Omega^*_D$ consists of forms that vanish when pulled-back to the boundary manifold $\partial M$.

For primitive forms with boundary conditions, we can write down the following differential complex: 
\begin{align}\label{rc}
\begin{CD}
0@>>>P_{D_{+}}^0@>\partial_{+}>>P_{D_{+}}^1@>\partial_{+}>>\cdots @>\partial_{+}>>P_{D_{+}}^{n-1}@>\partial_{+}>>P_{D_{++}}^n\\
    @. @. @. @.   @.                                                                                                            @VV{\partial_{+}\partial_{-}}V\\
0@<\partial_{-}<<P^0_{D_{-}}@<\partial_{-}<<P^1_{D_{-}}@<\partial_{-}<<\cdots @<\partial_{-}<<P_{D_{-}}^{n-1}@<\partial_{-}<<P^n_{D_{-}}.
\end{CD}
\end{align} 
By Lemmas \ref{pre} and \ref{pree}, this complex is well-defined.  For instance, $\partial_{+}$ preserves the boundary condition $D_{+}$,  $\partial_{-}$ preserves $D_{-}$, and $\dpp\dpm$ maps a primitive form with $D_{++}$ condition into one with $D_-$ condition.  In analogy with the relative de Rham complex which imposes the Dirichlet boundary condition on forms, we call the cohomologies corresponding to the complex \eqref{rc} {\it relative primitive cohomologies} and denote them by 
\begin{align*}
PH^k_+(M, \partial M) &=\dfrac{\ker \partial_{+} \cap P^k_{D_+}(M)}{ \partial_{+}P_{D_{+}}^{k-1}(M)}\,, \quad \text{for $k=0, 1, 2, \ldots, n-1\,$},\\
PH^n_+(M, \partial M)&=\frac{\ker \partial_{+}\partial_{-}\cap P^n_{D_{++}}(M)}{\partial_{+}P_{D_+}^{n-1}(M)}\,,\\
PH^n_-(M, \partial M)&=\frac{\ker \partial_{-}\cap P_{D_-}^{n}(M)}{\partial_{+}\partial_{-}P_{D_{++}}^{n}(M)}\,,\\
PH^k_-(M, \partial M)&=\frac{\ker \partial_{-}\cap P_{D_-}^{k}(M)}{\partial_{-}P_{D_-}^{k+1}(M)}\,, \quad\text{for $k=0,1, 2, \ldots, n-1\,$}.
\end{align*} 
We emphasize that the standard Dirichlet and Neumann boundary conditions are not suitable here since they are not preserved by the differential operators $(\dpp, \dpm)$ in this complex.

Using the decompositions we obtained in Section \ref{hodgede}, we can immediately show that the relative cohomologies are isomorphic to the spaces of harmonic fields with $D_{+}, D_{-}$, or $D_{++}$ boundary conditions.
\begin{thm}\label{relD}
Let $(M, \omega)$ be a compact symplectic manifold with a smooth bound- ary. Let $(\omega, J, g)$ be a compatible triple on $M$.  We have the following isomorphisms: 
\begin{align}\label{relhar}
PH^k_+(M, \partial M) \cong P\mathcal{H}^k_{+, D_{+}}\!(M)\,, \quad PH^k_-(M, \partial M) \cong P\mathcal{H}^k_{-, D_{-}}\!(M)\,,
\end{align}
for $k<n$ and
\begin{align}\label{relhar2}
PH^n_+(M, \partial M) \cong P\mathcal{H}^n_{+, D_{++}}\!(M)\,, \quad PH^n_-(M, \partial M) \cong P\mathcal{H}^n_{-, D_{-}}\!(M)\,.
\end{align}
\end{thm}
\begin{proof} The isomorphisms follows directly from the following Hodge decompositions: 
\begin{align*}
P^k&=P\mathcal{H}^k_{D_{+}} \oplus \partial_{+}P^{k-1}_{D_{+}} \oplus \partial_{+}^{\ast}P^{k+1}~,\\
P^k&=P\mathcal{H}^k_{D_{-}} \oplus \partial_{-}P^{k+1}_{D_{-}} \oplus \partial_{-}^{\ast}P^{k-1}~,
\end{align*}
of Theorem \ref{Hdp}.2.(i) and Theorem \ref{Hdm}.2.(i), respectively, in the case of $k< n$, and for $k=n$ 
\begin{align*}
P^n &=P\mathcal{H}_{+,D_{++}}^n\!\!\oplus \partial_{+}P_{D_{+}}^{n-1}\oplus (\partial_{+}\partial_{-})^{\ast}P^{n}~,\\
P^n &=P\mathcal{H}_{-,D_{-}}^n\oplus (\partial_{+}\partial_{-})P_{D_{++}}^{n}\oplus \partial_{-}^{\ast}P^{n-1}~,
\end{align*}
of Theorem \ref{Hdpp}.2.(i) and Theorem \ref{Hdmm}.2.(i)
\end{proof}

Interestingly, the relative primitive cohomology is naturally paired with the absolute primitive cohomology.
\begin{thm}\label{pairing}
On a compact symplectic manifold $(M, \omega)$ with smooth boundary $\partial M$, we have the following for $k=0, 1, \ldots, n$,
\begin{align}\label{pai}
PH^{k}_+(M) \cong PH^k_-(M, \partial M) \, , \qquad PH^k_-(M) \cong PH^{k}_+(M, \partial M)\,,
\end{align}
and the corresponding non-degenerate pairings
\begin{alignat}{4}\label{paa}
& PH^{k}_+(M) &&\otimes\, PH^k_-(M, \partial M) &&\longrightarrow ~~\quad\mathbb{R}\\
&~\,\quad [\beta]\quad &&\otimes\quad\quad [\lambda] &&\longrightarrow ~~(-1)^{\frac{k(k+1)}{2}}\!\!\int_M  \frac{\omega^{n-k}}{(n-k)!}\wedge \beta \wedge \lambda  \nonumber
\end{alignat}
\begin{alignat}{4}\label{pab}
&PH^k_-(M) &&\otimes\, PH^{k}_+(M, \partial M) &&\longrightarrow ~~\quad\mathbb{R} \\
&~\,\quad[\beta]\quad &&\otimes\quad \quad[\lambda] &&\longrightarrow ~~(-1)^{\frac{k(k+1)}{2}}\!\!\int_M  \frac{\omega^{n-k}}{(n-k)!}\wedge \beta \wedge \lambda\nonumber
\end{alignat} 
\end{thm}
\begin{proof}
The isomorphisms between absolute and relative primitive cohomologies are obtained by the following: (i) isomorphisms of the cohomologies with the corresponding harmonic field spaces given in Theorems \ref{absoluteN} and \ref{relD}; (ii) the isomorphisms between the harmonic fields \eqref{harma}-\eqref{harmb}.

Regarding the pairing, we shall give the arguments for the first pairing \eqref{paa} as that for the second pairing \eqref{pab} are similar.  Let $(\om, J, g)$ be a compatible triple. We recall first the relation for primitive forms under the action of the Hodge star operator $\ast$ with respect to the metric $g$ (see e.g. \cite{TY2}):
\[
 \ast ~ \lambda_k= (-1)^{\frac{k(k+1)}{2}}\frac{\omega^{n-k}}{(n-k)!}\w \mathcal{J}(\lambda_k) \,,
\]
where $\lambda_k \in P^k$ and $\mathcal{J}$ is the conjugate operator defined in \eqref{conjop} with respect to $J$.  Using this, we can re-write the integral in \eqref{paa} as
\[
(-1)^{\frac{k(k+1)}{2}}\!\!\int_M \beta \w \frac{\omega^{n-k}}{(n-k)!}  \wedge \lambda=  \int_M \beta \w \ast ~ \mathcal{J}^{-1}(\lambda)= ( \mathcal{J}\beta, \lambda)\,.
\]

We show that the pairing \eqref{paa} is well-defined, that is, the integral only depends on the cohomology classes. Consider first taking $\beta + \dpp \varphi$ as the representative of $PH^{k}_+(M)$ with 
$\varphi \in P^{k-1}$. The additional $\dpp$- exact term has no contribution since
\begin{align*}
\left(\mathcal{J}\dpp\varphi, \lambda\right) &= (\mathcal{J} \dpp \mathcal{J}^{-1}(\mathcal{J}\varphi), \lambda) = \left(\partial_{-}^{\ast} (n-k+1) \mathcal{J}\varphi, \lambda\right) \\
& = (n-k+1) \left[ (\mathcal{J}\varphi, \dpm \lambda) - \int_{\partial M} \langle \mathcal{J}\varphi, \sigma_{\dpm}(d\rho)\lambda \rangle\,dS \right]~~ = 0~,
\end{align*} 
where, in the first line, the conjugate relation between $\dpp$ and $\partial_{-}^{\ast}$ of Lemma \ref{conjugate} was used, and the second line vanishes since $\lambda \in D_{-}$ and also $\dpm$-closed.  Alternatively, if we consider instead the representative $\lambda + \dpm \sigma$ for $PH^k_-(M, \partial M)$ with $\sigma \in P^{k+1}_{D_{-}}\,$ and $k<n$, or  $\lambda + \partial_{+}\partial_{-} \sigma$ for $PH^n_-(M, \partial M)$ with $\sigma \in P^{n}_{D_{++}}\,$, then the additional contribution would be
\begin{align*}
(\mathcal{J}\beta,\partial_{-}\sigma)
=(\partial_{-}^{\ast}(\mathcal{J}\beta),\sigma) = 0~,
\end{align*} or
\begin{align*}
(\mathcal{J}\beta, \partial_{+}\partial_{-}\sigma)
=(\partial_{+}^{\ast}\partial_{-}^{\ast}(\mathcal{J}\beta),\sigma)  = 0~,
\end{align*}
which similarly vanishes since $\partial_{+}\beta =0$ implies that $\partial_{-}^{\ast} (\mathcal{J}\beta)=0$ (again using Lemma \ref{conjugate}) and the boundary condition on $\sigma$.  Clearly, the exact terms do not contribute to the integral, and therefore, the pairing only depends on the cohomology classes.

To show non-degeneracy, we use the isomorphisms in \eqref{abhar}-\eqref{abhar2} and \eqref{relhar}-\eqref{relhar2} to choose $\beta\in PH^{k}_+(M)$ and $\lambda \in PH^k_-(M, \partial M)$ to be the harmonic representatives of their respective cohomology classes, i.e. $\beta\in P\mathcal{H}^{k}_{+, N_{+}}\!(M)$ and $\lambda \in P\mathcal{H}^{k}_{-, D_{-}}\!(M)$. Further, if we take $\lambda= \mathcal{J} \beta\,$, then the pairing becomes   
\[
\beta \otimes \lambda \rightarrow 
(\mathcal{J} \beta, \mathcal{J} \beta)~~ = ~~\|\mathcal{J}\beta\|^2~,
\] 
which is non-zero as long as $ \beta \neq 0$. 
\end{proof}

\subsection{Relative Lefschetz maps}\label{exactseq}
Recall that the kernels and cokernels of the Lefschetz maps 
\[
L: H^{k}(M) \rightarrow H^{k+2}(M) 
\]
can be characterized by various primitive cohomologies as in \eqref{fva}-\eqref{fvb}.  But with $\partial M$ not vanishing, we can additionally consider studying Lefschetz maps on forms with boundary conditions.  In fact, Lefschetz maps on $\Omega^*_D$, i.e. forms with the Dirichlet boundary condition, are well-defined since  
$$L: \Omega_D^k \rightarrow \Omega_D^{k+2}~.$$
To see this, suppose $\eta \in \Omega_D^k\,$, that is $w_1\wedge \eta~|_{\partial M}=0$ where locally $w_1 = d\rho$.  Then, clearly $L(\eta)=\omega\w \eta  \in \Omega_D^{k+2}$ since   
\begin{align*}
w_1\wedge L(\eta)~|_{\partial M}= \omega\wedge (w_1\wedge \eta)~|_{\partial M}=0~.
\end{align*}

With this property, we can ask whether the short exact sequences of Lefschetz maps on $\Omega^*$ without any boundary condition in \cite{TTY}
\begin{align}\label{ses}
\begin{CD}
@.0@>>>\Omega^{k-2}@>L>>\Omega^{k}@>\Pi>>P^{k}@>>>0~,\\
@.0@>>>\Omega^{n-1}@>L>>\Omega^{n+1}@>>> 0 \\
0@>>>P^{k}@>\ast_r>>\Omega^{2n-k}@>L>> \Omega^{2n-k+2}@>>>0@.
\end{CD}
\end{align}
for $k=0, 1, \ldots, n\,$, have analogues when the Dirichlet boundary condition is imposed.  It turns out that most but not all of the exact sequences above can be extended to the Dirichlet boundary condition case.  Let us first describe when Lefschetz maps on $\Omega^*_D$ are injective or surjective.
\begin{lemma}\label{rl} On a symplectic manifold $(M^{2n}, \om)$ with non-trivial boundary, the Lefschetz maps have the following properties: 
\begin{itemize}
\item $L\!:\, \Omega_D^{k-2} \rightarrow \Omega_D^{k}~$ is injective for $\,2 \leq k \leq n+1\,$;\\
\item $L\!:\, \Omega_D^{2n-k}\rightarrow \Omega_D^{2n-k+2}~$ is surjective for $\,2 \leq k \leq n\,$.
\end{itemize}
\end{lemma}
\begin{proof}
The injective property follows from the first two exact sequences of \eqref{ses} and that $L\!: \Omega_D^{k-2} \rightarrow \Omega_D^{k}$ is well-defined.  For the surjective property, we need to show that for any $\eta \in \Omega^{2n-k+2}_D$ and  $2 \leq k \leq n\,$, there is an $u \in \Omega_D^{2n-k}$ such that $L(u)=\eta$.  But already, the third sequence of \eqref{ses} gives surjectivity when no boundary condition is imposed.  Hence, we only need to demonstrate surjectivity of the Lefschetz map at local neighborhoods of the boundary $\partial M$ with the Dirichlet boundary condition added.  For this near boundary analysis, it suffice to work in the local Darboux  basis $\{w_j\}$ of one-forms from Section \ref{locD}.

First note that we can decompose a $(2n-k+2)$-form, $\eta$, in the following way: 
\begin{align}\label{Dc0}
\eta=\omega^{n-k+2}\wedge \left(\beta_{k-2}+\omega\wedge \xi_{k-4}\right)
\end{align}
where $\beta_{k-2} \in P^{k-2}$ and $\xi_{k-4} \in \Omega^{k-4}$. That $\eta\in \Omega^{2n-k+2}_D$ imposes the condition
\begin{align}\label{Dc1}
0= w_1 \w \eta~ |_{\partial M} =\om^{n-k+2}\w \left( w_1 \w \beta_{k-2} + \om \w w_1 \w \xi_{k-4}\right)~|_{\partial M} ~.
\end{align}
Let us focus on the $w_1 \w \beta_{k-2}~|_{\partial M}$ term in \eqref{Dc1}.  We apply the local decomposition of \eqref{ldecomp} to $\beta_{k-2}$:
\begin{align}\label{Dc3}
\beta_{k-2}= w_1 \w \tbeta_{k-3}^1 + w_2 \w \tbeta_{k-3}^2 + \Theta_{12}\w\tbeta_{k-4}^3 + \tbeta_{k-2}^4
\end{align}
where the primitive forms $\tbeta^i$'s here do not have any components in $w_1$ or $w_2$.  Then 
\begin{align*}
w_1 \w \beta_{k-2}~|_{\partial M}  &= \left(w_1 \w w_2 \w \tbeta_{k-3}^2 + w_1 \w \Theta_{12}\w\tbeta_{k-4}^3 + w_1 \w \tbeta_{k-2}^4\right)|_{\partial M} \\
&= \left(w_1 \w \tbeta_{k-2}^4 + \left[\dfrac{H+1}{H+2} \Theta_{12} + \dfrac{1}{H+2}\om \right]\w \tbeta_{k-3}^2  + w_1 \w \Theta_{12}\w\tbeta_{k-4}^3\right)|_{\partial M}
\end{align*}
Substituting the above expression into \eqref{Dc1}, implies that $\tbeta_{k-3}^2~|_{\partial M} = 0$, since a non-vanishing $\tbeta_{k-3}^2$ would lead to terms that can not be cancelled out by the second term in \eqref{Dc1} which must contain a $w_1$. 
Therefore, if we write
\begin{align}\label{Dc2}
w_1 \w \beta_{k-2}~|_{\partial M} =(\varphi_{k-1} + \om \w \varphi_{k-3})~|_{\partial M}
\end{align}
where $\varphi_{k-1}, \varphi_{k-3}$ are primitive forms,  
then
\begin{align*}
\varphi_{k-1}~|_{\partial M}& = w_1 \w \tbeta_{k-2}^4~|_{\partial M} \\
\om \w \varphi_{k-3}~|_{\partial M} &=  w_1 \w \Theta_{12}\w\tbeta^3_{k-4}~|_{\partial M} ~.
\end{align*}
Note that \eqref{Dc1} imposes no condition on $\tbeta^4_{k-2}$ along $\partial M$, since by primitivity, $\om^{n-k+2} \w \varphi_{k-1} =0\,$.  On the other hand, for $\tbeta^3_{k-4}$, \eqref{Dc1} implies
\begin{align}\label{Dc4}
\left( w_1 \w \Theta_{12}\w\tbeta^3_{k-4} + \om \w w_1 \w \xi_{k-4}\right)~|_{\partial M} =0~.
\end{align}

We can now write down a $u\in \Omega^{2n-k}_D$ such that $L(u)=\eta\,$.  Define  
$$u=\omega^{n-k}\wedge \left(\beta_{k}+\omega\wedge \beta_{k-2}+\omega^2\wedge \xi_{k-4}\right)~,$$
where $\beta_{k-2}$ and $\xi_{k-4}$ are those in \eqref{Dc0} and $\beta_k \in P^k$ is a primitive $k$-form with its value on the boundary specified by $\beta_{k-2}$:  
\begin{align}\label{Dc5}
 \beta_k~|_{\partial M}&=(H+2)\, \sigma(\partial_{+}\partial_{-}^{\ast})(d\rho)\,\beta_{k-2}~|_{\partial M}\nonumber\\
 &= (H+2)\, \Pi(w_1\wedge w_2\wedge \beta_{k-2})~|_{\partial M}\nonumber \\
&= (H+1)\, \Theta_{12} \w \tbeta^4_{k-2}~|_{\partial M}
\end{align}
where in the second line, we have noted that $\sigma(\partial_{+}\partial_{-}^{\ast})(d\rho)\,\beta_{k-2}=\Pi(w_1\wedge w_2\wedge \beta_{k-2})$, and in the third line, we have substituted in the decomposition of \eqref{Dc3}. 
Clearly,
\begin{align*}
L(u)=\omega^{n-k+1}\wedge \left(\beta_{k}+\omega\wedge \beta_{k-2}+\omega^2\wedge \xi_{k-4}\right)
=\omega^{n-k+2}\wedge \left( \beta_{k-2}+\omega\wedge \xi_{k-4}\right)=\eta~.
\end{align*}
Moreover, we can check that $u$ also satisfies the Dirichlet boundary condition:
\begin{align*}
w_1 \wedge u~|_{\partial M}&=\omega^{n-k}\wedge \left(w_1\wedge\beta_{k}+\omega\wedge [w_1\wedge\beta_{k-2}]+ \omega^2\wedge w_1\wedge \xi_{k-4}\right)|_{\partial M}\\
&=\omega^{n-k}\wedge \Big(-w_1\w \om \w \tbeta^4_{k-2} + \om \w \left[w_1 \w \tbeta^4_{k-2} + w_1 \w \Theta_{12} \w \tbeta^3_{k-4}\right]\\
&\qquad\qquad\qquad\qquad \om^2 \w w_1 \w \xi_{k-4}\Big)~|_{\partial M}\\&=0~,
\end{align*}
having applied \eqref{Dc3}-\eqref{Dc5}.
\end{proof}
The injectivity and surjectivity of the Lefschetz maps on $\Omega^*_D$ can be incorporated into the following exact sequences. 
\begin{prop} \label{pres}The following sequences are exact for $0\leq k<n$:
\begin{align*}
\begin{CD}
 @.0@>>>\Omega^{k-2}_D@>L>>\Omega_D^{k}@>\Pi>>P^{k}_{D_{+}}@>>>0\\
@.0@>>>\Omega^{n-2}_D@>L>>\Omega_D^{n}@>\Pi>>P^{n}_{D_{+-}}@>>>0\\
0@>>>P^{n}_{D_{-}}@>\ast_r>>\Omega_D^{n}@>L>> \Omega^{n+2}_D@>>>0@.\\
0@>>>P^{k}_{D_{-}}@>\ast_r>>\Omega_D^{2n-k}@>L>> \Omega^{2n-k+2}_D@>>>0@. 
\end{CD}
\end{align*}
\end{prop}
\begin{proof}
By Proposition \ref{prop4} and Lemma \ref{rl}, these sequences are well-defined. To see the exactness of the first two set of sequences, we only need to show that $\ker \Pi~ |_{\Omega^{k}_D}\subset L(\Omega^{k-2}_D)$ for $k \leq n$. In this case, consider for any $\eta \in \Omega^{k}_D$ such that $\Pi \,\eta=0\,$.  Then we can write $\eta= \omega\wedge \xi$ for some  $\xi \in \Omega^{k-2}$.  Since $\eta \in D$, this gives the condition
\begin{align}\label{Dc7}
w_1\wedge\eta~|_{\partial M}=\omega\wedge(w_1\wedge \xi)~|_{\partial M} =0
\end{align}
But by \eqref{ses}, $L$ is injective when acting on $\Omega^{j}$ for $j\leq n-1\,$.  Hence, \eqref{Dc7} implies that $w_1\wedge \xi ~|_{\partial M}=0\,$ or $\xi \in \Omega^{k-2}_D\,$.

To see the exactness of the third and the fourth set of sequences, we only need to show that $\ker L|_{\Omega^{2n-k}_D} \subset\ast_r(P^k_{D_{-}})$ when $k\leq n$. Let now $\eta \in \Omega^{2n-k}_D$ for $k\leq n$ such that $\omega \wedge \eta=0\,$.  Then by the third exact sequence of \eqref{ses}, there exists an $\xi \in P^k$ such that $\eta= \ast_r \, \xi = \omega^{n-k}\wedge \xi$. Here, it is convenient to express the $D$ boundary condition on $\eta$ differentially as $d(\rho\, \eta)~|_{\partial M}=0\,$ as described in Remark \ref{diffequiv}. This implies
\begin{align*}
0&= d(\rho\, \eta)~|_{\partial M} =  d(\rho\, [\omega^{n-k} \w \xi])~|_{\partial M} = \omega^{n-k} \w d(\rho\, \xi)~|_{\partial M} \\
&= \omega^{n-k} \w \left[ \dpp(\rho\, \xi) +  \omega \w \dpm(\rho\, \xi) \right]~|_{\partial M}\\
&=  \omega^{n-k+1}\w \dpm(\rho\, \xi) ~|_{\partial M} = \ast_r \,\dpm(\rho\,\xi)~|_{\partial M}
\end{align*}
Hence, we obtain $\xi \in P^k_{D_{-}}\,$.
\end{proof}
\begin{remark}\label{mnote} With Proposition \ref{pres}, we have reproduced with boundary conditions the top and the bottom exact sequences of \eqref{ses}.   However, for the middle sequence, Lemma \ref{rl} tells us that
\begin{align*}
L\!:\, \Omega_D^{n-1} \rightarrow \Omega_D^{n+1}~,
\end{align*}
is injective, but not surjective in general.  We will see this in the discussion of examples in next section.
\end{remark}

That the Lefschetz operator $L$ has a well-defined action on  $\Omega_D^*$ allows us to consider the action of Lefschetz maps on relative de Rham cohomologies which are defined over $\Omega_D^*$: 
\[
L: H^{k}(M, \partial M) \rightarrow H^{k+2}(M, \partial M). 
\] 
These Lefschetz maps turn out to be related to the relative primitive cohomologies $PH^*(M,\partial M)$ analogous to the absolute case.  
Immediately, from the short exact sequences of Proposition \ref{pres}, we can write down two  commutative diagrams:
$$\begin{CD}\label{CD1}
   @.           @.\vdots                     @. \vdots                                   @.\\
0@>>> \Omega^0_D@>L>>\Omega^2_D@>\Pi>>     P^{2}_{D_{+}}@>>>0\\
  @.                         @VVdV                         @VVdV                                   @VV\partial_{+}V      @.\\
   @. \vdots            @.\vdots                   @. \vdots                                 @.\\
@.                         @VVdV                         @VVdV                                   @VV\partial_{+}V      @.\\
0@>>> \Omega^{n-3}_D@>L>>\Omega^{n-1}_D@>\Pi>>P^{n-1}_{D_{+}}@>>>0\\
@.                         @VVdV                         @VVdV                                   @VV\partial_{+}V      @.\\
0@>>> \Omega^{n-2}_D@>L>>\Omega^{n}_D@>\Pi>>P^{n}_{D_{+-}}@>>>0
\end{CD}$$ and
\vspace{1cm}
 $$\begin{CD}\label{CD2}
0@>>>P^n_{D_{-}}@>\ast_r>>\Omega^n_D@>L>> \Omega^{n+2}_{D}@>>>0\\
@. @VV\partial_{-}V                                 @VVdV                             @VVdV@.\\
0@>>>P^{n-1}_{D_{-}}@>\ast_r>>\Omega^{n+1}_D@>L>> \Omega^{n+3}_{D}@>>>0\\
@. @VV\partial_{-}V                                 @VVdV                             @VVdV@.\\
 @. \vdots             @.\vdots                     @. \vdots                                   @.\\
 @. @VV\partial_{-}V                                 @VVdV                             @VVdV@.\\
0@>>>P^{2}_{D_{-}}@>\ast_r>>\Omega^{2n-2}_D@>\Pi>>\Omega^n_D@>>>0\\
 @. \vdots             @.\vdots                     @.                                  @.\\
 \end{CD}$$
These two commutative diagrams imply two long exact sequences of cohomologies linking $PH^k_{\pm}(M, \partial M)$ with Lefschetz maps on $H^*(M, \partial M)$ for $k<n$.  However, by Remark \ref{mnote}, we are not able to extend the long exact sequence of cohomologies through $PH^n_{\pm}(M, \partial M)$ with Lefschetz maps.  To relate $PH^*_{\pm}(M,\partial M)$ with Lefschetz maps on $H^{*}(M, \partial M)$ for all $k=0, 1, \ldots, n$, we will make use of harmonic fields as in the proof of the theorem below.

\begin{thm}\label{relLef}
On a symplectic manifold $(M^{2n}, \om)$ with non-trivial boundary $\partial M$, we have the following isomorphisms:
\begin{align*}
PH^k_+(M, \partial M)&\cong {\rm coker}[L\!: H^{k-2}(M, \partial M)\rightarrow H^k(M, \partial M)]  \\ &\qquad\oplus \ker [L\!: H^{k-1}(M, \partial M)\rightarrow H^{k+1}(M, \partial M)]\,,\quad k=0,1,\ldots, n,\\
PH^k_-(M, \partial M)&\cong {\rm coker}[L\!: H^{2n-k-1}(M, \partial M)\rightarrow H^{2n-k+1}(M, \partial M)] \\& \qquad\oplus \ker [L\!: H^{2n-k}(M, \partial M)\rightarrow H^{2n-k+2}(M, \partial M)]\,,~~~ k=0,1,\ldots, n.
\end{align*}
\end{thm}
\begin{proof} 
From \eqref{pai} and \eqref{fva}-\eqref{fvb}, we have 
\begin{align*}
PH^k_+(M, \partial M) 
\cong PH^k_-(M) &\cong {\rm coker}[L\!: {H^{2n-k-1}(M)\to H^{2n-k+1}(M)}]  \\
&\qquad\quad\oplus \ker [L\!: {H^{2n-k}(M)} \to H^{2n-k+2}(M)]~,\\
PH^k_-(M, \partial M) \cong PH^k_+(M) &\cong {\rm coker}[L\!: {H^{k-2}(M) \to H^k(M)}] \\
&\qquad\quad\oplus \ker [L\!: {H^{k-1}(M)}\to H^{k+1}(M)]~.
\end{align*}
Thus, it suffices to show that 
\begin{align}
\ker [L\!: {H^{k}(M)} \to H^{k+2}(M)]& \cong {\rm coker}[L\!: {H^{2n-k-2}(M, \partial M) \to H^{2n-k}(M, \partial M)}]\label{kercok1}\\
{\rm coker}[L\!: {H^{k}(M) \to H^{k+2}(M)}] & \cong \ker [L\!: {H^{2n-k-2}(M, \partial M)} \to H^{2n-k}(M, \partial M)]\label{kercok2}
\end{align}
for all $k$.  To obtain such relations, we recall that by Lefschetz duality, $H^k(M) \cong H^{2n-k}(M,\partial M)$.  A way to see this follows from the equivalence of $H^k(M) \cong \mathcal{H}^k_N(M)$ and $H^k(M,\partial M)\cong \mathcal{H}^k_D(M)$ and that the map by the Hodge star, $\ast: \mathcal{H}^k_N(M) \to \mathcal{H}^{2n-k}_D(M)\,$, is an isomorphism (see, for example \cite{S}).  There is also a non-degenerate pairing that is well-defined on cohomology: 
\begin{alignat}{4}
& H^{k}(M) &&\otimes\, H^{2n-k}(M, \partial M) &&\longrightarrow ~~\quad\mathbb{R}\\
&~\,\quad [\eta]\quad &&\otimes\quad\quad [\xi] &&\longrightarrow ~~(-1)^k\!\!\int_M   \eta \wedge \xi  ~.\nonumber
\end{alignat}
With $\ast \,\ast = (-1)^k$ acting on $\Omega^k(M)$, we can express this pairing in terms of the usual inner product
\begin{align*}
(-1)^k\!\!\int_M   \eta \wedge \xi = \int_M   \eta \wedge \ast\,(\ast \xi) = (\eta, \ast\, \xi)    ~.
\end{align*}
And since the adjoint $L^*= (-1)^k \ast L \,\ast\,$, we have 
\begin{align*}
(L\,\phi\,, \ast\, \xi ) = (\phi\,, \ast\, L\, \xi )
\end{align*}
where $\phi\in \Omega^{k-2}(M)$.  It is then clear that for every $[\phi]\in \ker L|_{H^{k-2}(M)}$, there exists a corresponding $[\psi]\in H^{2n-k+2}(M, \partial M)$ such that $[\psi] \in {\rm coker}\, L|_{H^{2n-k}(M, \partial M)}$.  Such a cohomology pair, $([\phi], [\psi])$, is related as follows: let $\tilde{\phi} \in [\phi]$ be the harmonic representative i.e. ${\tilde\phi} \in \mathcal{H}^{k-2}_N(M)$, then $\ast\, {\tilde\phi}\in \mathcal{H}^{2n-k+2}_D(M, \partial M)$ and $\ast\, {\tilde\phi} = [\psi]$.  This gives an isomorphism between $\ker L|_{H^{k-2}(M)}$ and ${\rm coker}\, L|_{H^{2n-k}(M, \partial M)}$.

Likewise, if $[\xi]\in\ker L|_{H^{2n-k-2}(M, \partial M)}$ and ${\tilde \xi}\in [\xi]$ is the harmonic representative, i.e. ${\tilde \xi} \in\mathcal{H}^{2n-k-2}_D(M, \partial M)$, then $\ast\, {\tilde \xi}\in \mathcal{H}_N^{k+2}(M)$ and the associated cohomology class $[\ast\, {\tilde \xi}] \in {\rm coker}\, L|_{H^{k}(M)}\,$.  This gives an isomorphism between $\ker L|_{H^{2n-k-2}(M, \partial M)}$ and ${\rm coker}\, L|_{H^{k}(M)}$.
\end{proof}

\section{Examples}\label{Exsec}

We calculate here the absolute and relative primitive cohomologies for two symplectic manifolds with boundary: (i) an interval times a five-torus, $I\times T^5$; (ii) a three ball times a three-torus, $B^3 \times T^3$. For each case, we write down the basis of harmonic fields satisfying certain specific boundary conditions. These two simple examples will allow us to make evident some of the differences between primitive cohomology and de Rham cohomology on symplectic manifolds with boundary.

We note that the two examples we study are both K\"ahler.  However, in the case of a non-vanishing boundary, standard properties of closed K\"ahler manifolds may no longer hold.  For instance, the symplectic structure need not be in a non-trivial class and the Hard Lefschetz property may not hold.  Interestingly, in example (ii), we demonstrate clearly the dependence of the absolute and relative cohomologies on the symplectic structure.  In short, different symplectic structures on a manifold can give different dimensions for the absolute and relative cohomologies.  This is in contrast to the case of closed K\"ahler manifold where it was shown in \cite{TTY} that the dimension of primitive cohomologies are invariant under change of the K\"ahler class.

\subsection{$I \times T^5$}

Let $M= [0,1] \times T^5 $, the direct product of the $5-$torus and the interval. To set notation, let us define $M$ by modding out the following identification from $[0,1]\times \mathbb{R}^5$ :
\begin{equation*}
(x_1, y_1, x_2, y_2, x_3, y_3) \sim (x_1, y_1+a,x_2+b, y_2+c, x_3+d, y_3+e),
\end{equation*}  with $a, b, c, d,e \in \mathbb{Z}$. 
 We choose $\{dx_{i}, dy_{i}\}$ as the generating basis for $\Omega^{\ast}(M)$. The boundary is given by 
 \[
 \partial M= \{0\} \times T^5 \cup \{1\}\times T^5\, \text{with $d\rho= \pm dx_1, \vec{n}=\pm \frac{\partial}{\partial x_1}$},
 \]
where plus sign is for the $\{0\} \times T^5$ boundary and the minus sign for $\{1\}\times T^5$.  We consider the standard symplectic structure and Riemannian metric with
\[
\omega=\underset{i}{\sum}dx_i \w dy_i\,, \quad \mathcal{J} dx_i=dy_i\,.
\]  
The de Rham cohomology and primitive cohomology can be straightforwardly calculated and expressed in a basis of harmonic fields satisfying Neumann-type boundary conditions.  (For the tables in this section, the roman indices $\{i, j, l\}$ can take any value from $1$ to $3$ except as indicated, and we have suppressed the wedge product symbol  ``$\wedge$" in all the forms for notational simplicity.) 
 \begin{center}
\begin{tabular}{|c|c|c|}
\hline
$k$& $\dim H^{k}(M)$&  Basis in $\mathcal{H}^{k}_N(M)$ \\
\hline
$0$& $ 1$&$1$\\
\hline
$1$& $ 5$&  $dx_i, dy_j, i\neq 1$\\
  \hline
 $2$& $ 10$ & $dx_2 dx_3, dx_i  dy_j,  dy_j dy_l, i \neq 1$\\
 \hline
  $3$&$ 10$ & $d x_2 dx_3 dy_k, dx_i dy_j dy_l, dy_1dy_2dy_3, i \neq 1,$\\
  \hline
$4$ & $ 5$ & $dx_2 dx_3dy_j dy_l, d x_idy_1dy_2dy_3, i \neq 1$\\
  \hline
 $5$&$ 1$& $ dx_2 dx_3 dy_1 dy_2 dy_3$\\
\hline
$6$&$ 0$& $\emptyset$  \\
\hline
\end{tabular}\\
\end{center}
\begin{center}
\begin{tabular}{|c|c|c|}
\hline
$k$ & $\dim PH^{k}_+(M)$&  Basis in $P\mathcal{H}^{k}_{+, N_{+}}\!(M)$ \\
\hline
$0$&$1$& $ 1$\\
\hline
$1$& $5$ & $dx_i, dy_j, i\neq 1$\\
\hline
 $2$&$9$ & $dx_2 dx_3, dx_i dy_j, i \neq 1, i\neq j$\\
 &&$ dx_2 dy_2 - dx_3 dy_3 , dy_j dy_l  $\\
 \hline
 $3$&$10$ & $ dx_2 dx_3 dy_1,  dx_2 dy_1 dy_3, dx_3 dy_1 dy_2 , dy_1dy_2 dy_3,$\\
  & & $dy_1(dx_2 dy_2 - dx_3 dy_3), x_1dy_1(dx_2 dy_2 - dx_3 dy_3), $ \\
  & & $x_1dx_2 dx_3 dy_1, x_1 dx_2 dy_1 dy_3, x_1dx_3 dy_1 dy_2, x_1dy_1dy_2 dy_3$\\
 \hline\hline
$k$ & $\dim PH^{k}_-(M)$&  Basis in $P\mathcal{H}^{k}_{-, N_{-}}\!\!(M)$ or $P\mathcal{H}^3_{-, N_{--}}\!\!\!(M)$ \\
\hline
$0$&$ 0$& $\emptyset$\\
\hline
$1$& $ 1$ & $dy_1$\\
\hline
 $2$&$ 5$ & $dy_1dx_i, dy_1dy_i,  i \neq 1$\\
 && $dx_1 dy_1 -  \tfrac{1}{2}(dx_2dy_2-dx_3dy_3)$ \\
 \hline
$3$&$ 9$ & $dx_2 dx_3 dy_1, dx_2 dy_1 dy_3, dx_3 dy_1 dy_2 , dy_1dy_2 dy_3$\\
  & & $dx_2 (dx_1dy_1 - dx_3dy_3), dx_3 (dx_1dy_1-dx_2dy_2),  $ \\
  & & $ (dx_2 dy_2 - dx_3 dy_3)dy_1, (dx_1dy_1 - dx_3dy_3)dy_2, (dx_1dy_1-dx_2dy_2)dy_3$\\
\hline
\end{tabular}\\
\end{center}

The absolute primitive cohomology can be most easily calculated by Lefschetz maps as in \eqref{fva}-\eqref{fvb}. From the tables above, we find certain relations between de Rham cohomology and primitive cohomology. For instance, notice that the basis for $PH^{k}_+(M)$ are exactly the primitive subset of the basis of $H^{k}(M)$, for $k<3$. 
 
For relative cohomology, we find the following:

\begin{center}
\begin{tabular}{|c|c|c|}
\hline
$k$& $\dim H^{k}(M, \partial M)$& Basis in $\mathcal{H}^{k}_{D}(M)$\\
\hline
$0$&$ 0$& $\emptyset$\\
\hline
$1$& $ 1$ & $dx_1$\\
  \hline
$2$& $ 5$ & $dx_1 dx_i, dx_1  dy_j  $\\
 \hline
$3$&  $ 10$ & $d x_1 dx_2 dx_3, dx_1d x_i dy_j,  dx_1 dy_j dy_l $\\
\hline
 $4$& $ 10$ & $dx_1dx_2dx_3 dy_j, dx_1 dx_i dy_j dy_l, dx_1dy_1 dy_2 dy_3,$\\
  \hline
$5$&$ 5$& $ dx_1dx_2dx_3 dy_j dy_l, dx_1 dx_i dy_1dy_2 dy_3$\\
\hline
$6$&$ 1$&  $dx_1 dx_2 dx_3 dy_1dy_2  dy_3$ \\
\hline
\end{tabular}\\
\end{center}

\begin{center}
\begin{tabular}{|c|c|c|}
\hline
$k$& $\dim PH^{k}_+(M, \partial M)$ &  Basis in $P\mathcal{H}^{k}_{+, D_{+}}\!\!(M)$ or $P\mathcal{H}^{3}_{+, D_{++}}\!\!(M)$ \\
\hline
$0$& $ 0$& $\emptyset$\\
\hline
$1$& $ 1$ & $dx_1$\\
  \hline
$2$& $ 5$ & $dx_1 dx_i, dx_1  dy_i, \, i\neq 1$\\
 &&$dx_1 dy_1-\frac{1}{2}(dx_2 dy_2-dx_3 dy_3)$\\ 
 \hline
$3$& $9$ & $dx_1 dx_2 dx_3, dx_1 dx_2 dy_3, dx_1 dx_3 dy_2, $\\
  & & $dx_1dy_2 dy_3,  (dx_2 dy_2 - dx_3 dy_3)dx_1,$\\
  & & $ (dx_1dy_1 - dx_3dy_3)dx_2, (dx_1dy_1-dx_2dy_2)dx_3$\\
  & & $(dx_1dy_1 - dx_3dy_3)dy_2, (dx_1dy_1-dx_2dy_2)dy_3 $\\
  \hline\hline
 $k$& $\dim PH^{k}_-(M, \partial M)$ &  Basis in $P\mathcal{H}^{k}_{-, D_{-}}\!(M)$ \\
 \hline
$0$&$ 1$ & $1$\\
  \hline
$1$& $ 5$ & $dx_j, dy_i, i\neq 1$\\
\hline
 $2$&$ 9$ & $dy_2 dy_3, dx_i dy_j,  i \neq 1,i\neq j$\\
 &&$ dx_j dx_k, dx_2dy_2-dx_3 dy_3,  $\\
\hline
$3$&$ 10$ & $dx_1 dx_2 dx_3, dx_1 dx_2 dy_3, dx_1 dx_3 dy_2, dx_1dy_2 dy_3$\\
  & & $dx_1(dx_2 dy_2 - dx_3 dy_3), x_1dx_1(dx_2 dy_2 - dx_3 dy_3), $ \\
  & & $x_1dx_1 dx_2 dx_3, x_1dx_1 dx_2 dy_3, x_1dx_1 dx_3 dy_2 , x_1 dx_1dy_2 dy_3$\\
\hline
\end{tabular}\\
\end{center}
Here, the relative de Rham cohomology can be obtained by the standard long exact sequence $$ \ldots \longrightarrow H^k(M, \partial M) \longrightarrow H^k(M) \longrightarrow H^k(\partial M) \longrightarrow \ldots $$ 
while the relative primitive cohomology can be calculated using the Lefschetz map relations in Theorem \ref{relLef}.  

Clearly, the elements of the absolute cohomology are different from those of the relative ones.  For example, $dx_1$ is certainly $d$-exact and so is trivial in absolute cohomology.  However, it is a non-trivial element of $H^1(M, \partial M)$ and $PH^1_+(M, \partial M)$ since there is no linear function of $x_1$ that satisfies the Dirichlet condition at both ends of the interval, $x_1=0$ and $x_1=1$.    Notice also that the results of the above tables satisfy the pairing isomorphism of Theorem \ref{pairing}.  The pair of cohomologies - $\{PH^k_+(M), PH^k_-(M, \partial M)\}$ and $\{PH^k_-(M), PH^k_+(M, \partial M)\}$ - are related by a $\mathcal{J}$-conjugation.  Regarding Lefschetz maps on $\Omega^*_D$, it is clear that that $L\!: \Omega^{n-1}_D \to \Omega^{n+1}_D$ is not surjective (as noted in Remark \ref{mnote}) as, for example, the element $dx_1dy_1dy_2 dy_3 \in \Omega^4_D$ in the table above does not have a pre-image in $\Omega^2_D$ under the Lefschetz map.

\subsection{$B^3 \times T^3$}
Now consider $M=B^3 \times T^3$, the direct product of the unit ball in $R^3$ and a three-torus. Again to set notation, we define $M$ by modding out the following identification from $ B^3 \times R^3$:
\[
(x_1, x_2, x_3, y_1, y_2, y_3) \sim (x_1, x_2, x_3, y_1+a, y_2+b, y_3+c), a, b, c \in \mathbb{Z} 
\] with $x_1^2+x_2^2+x_3^2 \leq 1$. The boundary is given by 
\[ 
\partial M= S^2 \times T^3: \{x_1^2+x_2^2+x_3^2 =1\},~ {\rm with}~ d \rho= -\underset{i}{\sum}x_i dx_i, \, \vec{n}=-\underset{i}{\sum} x_i \frac{\partial }{\partial x_i}.
\]
We consider first the standard symplectic form and Riemannian metric: 
\[
\omega= \underset{i}{\sum}dx_i \w dy_i\,, \qquad \mathcal{J} dx_i =dy_i\,.
\]  
Then $\mathcal{J} d\rho= -\underset{i}{\sum}x_i dy_i$. Moreover, the symplectic form here is exact since $\omega= d \alpha$ with $\alpha= \underset{i}{\sum}x_i dy_i$. The boundary in this case is said to be of contact type and the Reeb vector field is given by  $\underset{i}{\sum}x_i \frac{\partial}{\partial y_i}$.

With $\om$ being exact, the Lefschetz map  $L\!: H^k(M) \rightarrow H^{k+2}(M)$ trivially maps all elements to zero. This leads to the following isomorphisms for $1 \leq k \leq n$: 
\[
PH^k_+(M) \cong H^{k-1}(M) \oplus H^{k}(M)\,,\quad PH^{k}_-(M) \cong H^{2n-k}(M)\oplus H^{2n-k+1}(M)\,.
\]
In particular, we find the following for the de Rham and primitive cohomology in the absolute case:
\begin{center}
\begin{tabular}{|c|c|c|}
\hline
$k$& $\dim H^{k}(M)$ &  Basis in $\mathcal{H}^k_N\!(M)$ \\
\hline
$0$&$ 1$& $ 1$\\
\hline
$1$&$ 3$ & $dy_1, dy_2, dy_3$\\
  \hline
$2$& $ 3$ & $dy_i \, dy_j $\\
 \hline
$3$& $ 1$ & $d y_1dy_2 \,dy_3$\\
  \hline
$4, 5, 6$&$ 0$ & $\emptyset$ \\
  \hline
\end{tabular}\\
\end{center}

\begin{center}
\begin{tabular}{|c|c|c|}
\hline
$k$& $\dim PH^{k}_+(M)$& Basis in $P\mathcal{H}^k_{+, N_{+}}\!(M)$ \\
\hline
$0$&$ 1$& $ 1$\\
\hline
$1$&$ 4$ & $dy_1, dy_2, dy_3, \alpha$\\
  \hline
$2$& $ 6$ & $dy_i dy_j,\alpha \, dy_i $\\
 \hline
 $3$&$ 4$ & $dy_1 dy_2 dy_3,\alpha \, dy_idy_j $\\
\hline\hline
 $k$& $\dim PH^{k}_-(M)$& Basis in $P\mathcal{H}^k_{-, N_{-}}\!\!(M)$ or $P\mathcal{H}^3_{-, N_{--}}\!\!(M)$ \\
 \hline
$0, 1, 2$& $ 0$ & $\emptyset$ \\
\hline
$3$&$ 1$& $dy_1dy_2dy_3$ \\
\hline
\end{tabular}\\
\end{center}
Of note here is the presence of $\alpha$ as a non-trivial element of $PH^1_+(M)$.  Since $d\alpha =\omega$, $\alpha$ is $\dpp$-closed but not $d$-closed.  
For relative cohomologies, we obtain the following:
\begin{center}
\begin{tabular}{|c|c|c|}
\hline
$k$& $\dim H^{k}(M, \partial M)$&Basis in $\mathcal{H}^k_{D}(M)$\\
\hline
$0, 1, 2$&$ 0$& $\emptyset$\\
\hline
$3$& $ 1$ & $d x_1 dx_2 dx_3$\\
  \hline
$4$&$ 3$ & $dx_ 1dx_2dx_3 dy_j$\\
  \hline
$5$&$ 3$& $ dx_1dx_2dx_3dy_i dy_j$\\
\hline
$6$&$ 1$& $dx_1 dy_1dx_2dy_2 dx_3 dy_3$ \\
\hline
\end{tabular}\\
\end{center}

\begin{center}
\begin{tabular}{|c|c|c|}
\hline
$k$& $\dim PH^{k}_+(M, \partial M)$ & Basis in $P\mathcal{H}^{k}_{+, D_{+}}\!\!(M)$ or $P\mathcal{H}^{3}_{+, D_{++}}\!\!(M)$\\
\hline
$0, 1, 2$&$ 0$& $\emptyset$\\
\hline
$3$&$ 1$ & $dx_1 dx_2 dx_3$\\
  \hline
\hline
 $k$& $\dim PH^{k}_-(M, \partial M)$ & Basis in $P\mathcal{H}^{k}_{-, D_{-}}\!(M)$ \\
\hline
 $0$& $ 1$ & $1$\\
  \hline
$1$&$ 4$ & $dx_1, dx_2, dx_3, d\rho$\\
  \hline
 $2$&$ 6$& $ dx_idx_j, d\rho \, dx_i, $\\
\hline
$3$ & $4$ & $ dx_1 dx_2 dx_3, d\rho \, dx_i dx_j$\\
\hline
\end{tabular}\\
\end{center} 
Here, the dimension of $PH^k_-(M, \partial M)$ is greater than that of $H^{2n-k}(M, \partial M)$, again in contrast to that in the first example.

For closed K\"ahler manifold, it is known that the dimension of $PH^{k}_+(M)$ is a constant with respect to different K\"ahler structures \cite{TTY}. This is due to the existence of the hard Lefschetz property which implies a Lefschetz decomposition of the de Rham cohomology. However the hard Lefschetz property do not in general hold when the boundary is not vanishing. Hence, in the case of manifold with boundary, the dimension of the cohomology $PH^k_{\pm}(M)$ may vary as the symplectic structure varies.  To demonstrate this, let us consider  again $M^6 = B^3 \times T^3$ but now with a different symplectic form and complex structure:
\begin{align*}
{\tomega}&=dx_1\w dx_2+dy_1\w dy_2+dy_3\w dx_3\, ,\\
\tmJ &dx_1=dx_2\,,\quad \tmJ dy_1=dy_2\,,\quad \tmJ dy_3=dx_3\,.
\end{align*}
Though this symplectic form is not exact, it still represents a K\"ahler structure. Moreover, $\tmJ d\rho=-x_1 dx_2+x_2dx_1+x_3dy_3$ whose corresponding vector is $\vec{v}=-x_1 \frac{\partial}{\partial x_2}+x_2\frac{\partial}{\partial x_1}+x_3 \frac{\partial}{\partial y_3}$. 
Of course, the de Rham cohomology and the relative de Rham cohomology being topological remains unchanged. However, the primitive cohomology and relative primitive cohomology are now different.
\begin{center}
\begin{tabular}{|c|c|c|}
\hline
$k$& $\dim PH^{k}_+(M)$& Basis in $P\mathcal{H}^k_{+, N_{+}}\!(M)$ \\
\hline
$0$&$ 1$& $1$\\
\hline
$1$&$ 3$ & $dy_1, dy_2, dy_3$\\
  \hline
$2$& $ 4$& $dy_1dy_3, dy_2dy_3, $ \\
&&$ (x_2dx_1-x_1dx_2+2x_3dy_3)dy_i,~~ i=1, 2$\\
\hline
$3$& $ 3$ & $ (x_1 dx_2 - x_2 dx_1)( dy_1 dy_2 - dy_3 dx_3) + x_3 dy_3 (dx_1 dx_2 - dy_1 dy_2),$\\
&& $(x_1 dx_2 - x_2 dx_1)dy_3 dy_i,~~ i=1, 2$ \\
\hline\hline
 $k$& $\dim PH^{k}_-(M)$& Basis in  $P\mathcal{H}^3_{-, N_{+-}}\!\!(M)$\\
 \hline
$0, 1, 2$& $ 0$& $\emptyset$\\
\hline
$3$&$ 1$& $(dx_1 dx_2 - dy_1 dy_2) dy_3$ \\
\hline
\end{tabular}\\
\end{center}
\begin{center}
\begin{tabular}{|c|c|c|}
\hline
$k$& $\dim PH^{k}_+(M, \partial M)$&Basis in Basis in $P\mathcal{H}^{k}_{+, D_{+}}\!\!(M)$ or $P\mathcal{H}^{3}_{+, D_{++}}\!\!(M)$\  \\
\hline
$0,1, 2$&$ 0$& $\emptyset$\\
\hline
$3$&$ 1$& $(dx_1 dx_2 - dy_1 dy_2) dx_3$ \\
 \hline\hline
 $k$& $\dim PH^{k}_-(M, \partial M)$ &Basis in $P\mathcal{H}^k_{-, D_{-}}\!\!(M)$  \\
\hline
 $0$& $ 1$& $1$ \\
  \hline
$1$&$ 3$& $dy_1, dy_2, dx_3$\\
  \hline
 $2$&$ 4$&$dy_1dx_3, dy_2dx_3,$\\
 &&$(x_1dx_1+x_2dx_2-2x_3dx_3)dy_i, ~~i=1,2$\\
 \hline
 $3$ &  $3$ & $(x_1 dx_2 + x_2 dx_1)( dy_1 dy_2 - dy_3 dx_3)$ \\
 && \quad$ - x_3 dy_3 (dx_1 dx_2 - dy_1 dy_2),$\\
 & & $(x_1 dx_1 + x_2 dx_2)dy_3 dy_i,~~i=1, 2$\\
\hline
\end{tabular}\\
\end{center}
Clearly, the dimensions of $PH^{k}_+(M)$ and $PH^{k}_-(M, \partial M)$ differ for the symplectic structure $\tomega$ as compared to those for ${\omega}$. 

\section{Discussion}

In this paper, we established Hodge theory for primitive cohomologies on symplectic manifolds with boundary.  In order to obtain a unique harmonic representative in each primitive cohomology class, we are required to impose on harmonic fields new Dirichlet- and Neumann-type boundary conditions that are dependent on the symplectic structure.  For those cohomologies associated with fourth-order symplectic Laplacians, the natural boundary conditions additionally involve derivatives.

We associated harmonic fields with Dirichlet-type symplectic boundary conditions with what we have called relative primitive cohomologies.  In differential topology, relative de Rham cohomology is well-defined for any submanifold $N$ embedded in $M$.  Let $i\!:\! N\hookrightarrow M$ be the inclusion map.  Then, there is a relative de Rham complex  defined by elements $\Omega^k_R(M, N)= \Omega^k(M) \oplus \Omega^{k-1}(N)$ with the differential $d$ given by 
\begin{align*}
d(\eta, \xi) = (d\eta, i^*\eta - d\xi)~.
\end{align*} 
Such a differential squares to zero and results in the relative de Rham cohomology, which we shall denote here by $H^k_R(M, N)$.  (For a reference, see \cite{BT}.) In the case of $N=\partial M$, it is well-known that  
\begin{align*}
H^k_R(M, \partial M) \cong H^k(M, \partial M)
\end{align*}
with $H^k(M, \partial M)$ being the standard de Rham cohomology defined over $\Omega^k_D(M)$, i.e. forms satisfying the Dirichlet boundary conditions.  

The isomorphism above begs the question whether the relative primitive cohomologies defined over forms with $\{D_+, D_{++}, D_-\}$ boundary conditions in Section \ref{RelSec} also have a description in terms of a ``relative" complex similar to the de Rham case.  To just generalize the relative de Rham complex by restricting $\Omega^*$ to primitive forms and replacing the differential with the appropriate symplectic operator from the triplet $(\dpp, \dpm, \,\dpp\dpm)$ that appear in the primitive elliptic complex of \eqref{pecomp} would run into an immediate obstacle: $N=\partial M$ is odd-dimensional, and hence, there is no general notion of a primitive form defined on $\partial M$.  (If $N$ happens to be a symplectic submanifold of $M$, then such a relative complex would make sense \cite{TTY2}.)  

To side-step this issue, we propose here considering a relative complex not with respect to $N$, but instead with respect to a {\it closed} tubular neighborhood of $N$ which we will label by $\tN$.  With the map $i: \tN \hookrightarrow M$ be the inclusion, the pullback $i^{\ast} \omega$ then defines a symplectic structure on $\tN$.   This would allow us to proceed to define a relative complex $(\CP_R(M, \tN), \partial)$ with elements  $\CP^l_R(M, \tN)= \CP^l(M) \oplus \CP^{l-1}(\tN)\,$. Here, the vector space $\CP^l$ with $l=0, 1, 2, \ldots, 2n+1,$ are just primitive spaces but sequenced by the order of their appearance in the primitive elliptic complex in \eqref{pecomp}.  Specifically, 
\begin{align}\label{pplexdef}
\CP^l= \begin{cases}
P^l &  \text{if~~ } 0\leq l \leq n~,\\
P^{2n+1-l} & \text{if ~~}  n+1 \leq l \leq 2n +1 ~.
\end{cases}
\end{align} 
which following \eqref{pecomp} is acted upon by the differential operator 
\begin{align}\label{pplexdif}
\partial_l= \begin{cases}
~\dpp &  \text{if~~ } 0\leq l < n-1~,\\
-\dpp\dpm & \text{if ~~} l = n~,\\
-\dpm & \text{if~~ } n+1 \leq j \leq 2n +1 ~.
\end{cases}
\end{align} 
(The extra minus signs make $(\CP^*, \partial_l)$ coincide with the algebra $\mathcal{F}^{p=0}$ in \cite{TTY}.)  The differential $\partial$ acting on the relative element $(\beta, \gamma)\in\CP^l_R(M, \tN)$ would then be standardly given by
\begin{align*}
\partial\,(\beta, \gamma) = (\partial_l\,\beta\,,\, i^*\beta\, - \,\partial_{l-1}\gamma\,)~.
\end{align*}
We will denote the resulting relative cohomology by $PH_{R}^*(M, \tN)$.  In the case, where $N=\partial M$, $N_T=\tpM $ would be a closed collar neighborhood of $\partial M$.  We then expect that  $PH_{R}^*(M, \tpM)$ is isomorphic to the relative cohomology $PH^*(M, \partial M)$ defined in Section \ref{RelSec}. 

We emphasize that the above relative primitive cohomology $PH_{R}^*(M, \tN)$ can be defined for any embedded submanifold $N$ of $M$ and this includes the interesting case where $N$ is a Lagrangian submanifold.  This is of particular relevance for a system of equations that arose in physics which constrains six-dimensional, symplectic Calabi-Yau manifolds with special Lagrangians playing the role of source charges \cite{Tomasiello, TY3}.   (Here, we follow the usage of the term ``Calabi-Yau"  to mean the existence of an $S\!U(3)$ holonomy structure with respect to a connection that may have torsion.). A six-dimensional, symplectic Calabi-Yau can be labelled by $(M^6, \om, \Om)$, where $\Om$ here is a non-vanishing $(3,0)$-form that defines an almost complex structure on $M^6$ and $\om$ is a symplectic $(1,1)$-form.  The physical system requires that the $(3,0)$ form $\Omega$ satisfies:
\begin{align*}
d \,\text{Re}~ \Om &= 0\\
dd^{\Lambda} e^{-2f} \text{ Im}~ \Om &= \rho_L
\end{align*}
with $\rho_L$ being the Poincar\'e-dual current of a special Lagrangian submanifold $L \subset M$ and 
$$e^{-2f}= \dfrac{3}{4} \dfrac{i\,\Omega \w {\overline{\Omega}} }{ \om^3}\,.$$  In \cite{TY3}, the above system was related to a Maxwell type system for (Re $\Om$).  Hence, in analogy with the relationship between Maxwell's equations and relative de Rham cohomology, we expect that the relative primitive cohomology $PH^4_{R}(M, L_T)$ should be relevant for measuring the source charges of the physical system and in understanding its space of solutions.  It is also an interesting question whether $PH^*_R(M, L_T)$ can be described by forms with certain prescribed boundary conditions when asymptotically close to $L$.  (When $\rho_L=0$, a geometric flow for this symplectic system was recently introduced in \cite{Fei}.) 

Lastly, primitive forms and their cohomologies are the special ($p=0$) case of the more general $p$-filtered forms and their filtered cohomologies described in Tsai-Tseng-Yau \cite{TTY}.  The description here should be straightforwardly generalizable to the $p$-filtered case by replacing the $(\dpp, \dpm,\, \dpp\dpm\,)$ operators with the more general $(d_+, d_-, \,\dpp\dpm\,)$ operators defined in \cite{TTY}.

\appendix

\section{Form Decomposition}
At times, it is useful to show explicitly a primitive form's dependence on certain directional components.  For instance, around a neighborhood $U\subset M$ of a point $x\in \partial M$, we can choose to work in a local Darboux basis of one-forms, $\{w_i\}$, for $i=1,\ldots, 2n$, such that $w_1=d\rho$.  Then, with respect to this Darboux basis, we can make explicit the dependence on the $\{w_1, w_2\}$ components for any $\beta_k \in P^k(U)$ by means of the decomposition in \eqref{ldecomp}:
 \begin{align}\label{ldecompa}
\beta_k= w_1 \w \beta_{k-1}^1 + w_2 \w \beta_{k-1}^2 + \Theta_{12}\w\beta_{k-2}^3 + \beta_k^4\,,
\end{align}
where $\beta^i$'s are primitive forms without any components in $w_1$ and $w_2$, and 
\begin{align}\label{Thdef}
\Theta_{12}=  w_1 \w w_2 - \dfrac{1}{H+1} \sum_{i=2}^{n} w_{2i-1} \w w_{2i}\,.
\end{align}  
The proof of this decomposition, which is algebraic in nature, is given in \cite[Lemma 2.3]{TY2}.  Note that when $k=n$, the term $\beta_{k=n}^4$ is identically zero as there is no primitive $n$-form without a component in either $w_1$ or $w_2$.

The decomposition above can be applied repeatedly.  Specifically, we can further extract out the dependence on the $\{w_3, w_4\}$ components for each $\beta^i$'s on the right hand side of \eqref{ldecompa}.  Each $\beta^i$ would in general break up into four terms, and we would end up in all with the following decomposition of 16 terms:
\begin{align}\label{ldecompb}
\beta_k 
 &=w_1 \w \beta_{k-1}^1 + w_2 \w \beta_{k-1}^2 + \Theta_{12}\w\beta_{k-2}^3 + \beta_k^4\nonumber\\
&= w_1\w \left[w_3 \w \ga^{13}_{k-2} + w_4\w \ga^{14}_{k-2} + \Thp \w\ga^{134}_{k-3} + \ga^1_{k-1} \right]\nonumber\\
& \quad +w_2 \w \left[ w_3 \w \ga^{23}_{k-2} + w_4\w \ga^{24}_{k-2} + \Thp \w\ga^{234}_{k-3} + \ga^2_{k-1} \right]\nonumber\\
& \quad +\Th \w \left[ w_3 \w\ga^{123}_{k-3} + w_4 \w\ga^{124}_{k-3} + \Thp \w\ga^{1234}_{k-4} + \ga^{12}_{k-2}\right] \\
&\qquad \ + \left[w_3 \w\ga^3_{k-1} + w_4 \w\ga^4_{k-1} + \Thp \w\ga^{34}_{k-2} + \ga^0_k\right] \nonumber
\end{align}
where the $\ga$'s in each of the sixteen terms above are primitive forms that do not have any components involving $\{\wa,\wb, \wc, \wdd\}$, and also
\begin{align}\label{Thpdef}
\Theta'_{34}=  w_3 \w w_4 - \dfrac{1}{H} \sum_{i=3}^{n} w_{2i-1} \w w_{2i}\,.
\end{align}  
Here, $\Thp$ is the analogous object to $\Th$ in \eqref{Thdef} but in dimension $d=2n-2$ with $\{w_3, w_4\}$ components singled out.   
Let us further point out that when $k=n-1$, the last term $\ga^0_{n-1}$ in \eqref{ldecompb} is identically zero.  And when $k=n$, the primitive forms $\{\ga^0_n, \ga^1_{n-1}, \ga^2_{n-1}, \ga^3_{n-1}, \ga^4_{n-1}, \ga^{34}_{n-2}\}$  are all identically zero in \eqref{ldecompb}.

\section{Symbols of Symplectic Differential Operators}\label{AppB}
For  $\zeta \in T^*_xM\!\setminus\!\{0\}$, the principal symbols $\sigma_{\mathcal P}(\zeta)\beta_k$ for any symplectic differential operator $\mathcal{P}$ discussed in this paper can be computed straightforwardly, though the calculations can be rather long and tedious.  We here write down some of the principal symbols that are relevant for the proof of the ellipticity of BVPs in Propositions \ref{2ndE} and \ref{4thE}.  In our principal symbols, we will not include any $i=\sqrt{-1}$ factors that are often included in the definition of the principal symbols, as these imaginary factors do not make a difference in our calculations.

Since calculating the symbol is a point-wise computation, we can without loss of generality, choose to work in the standard basis $\{w_i\}$ with $\om = w_1 \w w_2 + w_3 \w w_4 + \ldots$ and  $g_{ij}=\delta_{ij}$, assuming as we do that the metric is compatible to $\om$.  Of relevance for the proofs of Propositions \ref{2ndE} and \ref{4thE} is calculating the principal symbol when the covector $\zeta \in T^*_xM\!\setminus\!\{0\}$ takes the form $\zeta = \xa \wa + \xb \wb + \xc \wc$ which will be our focus here. 

{\it Notation}: In the expressions for the principal symbols below, the integer $h=n-k$, where $n=(\dim M)/2$ and $k$ is the degree of the primitive form acted upon by the principal symbol.  We will use the notation $w_{ij}=w_i \w w_j$ to denote the wedge product of two basis 1-forms.  Furthermore, to simplify expressions, we will drop the wedge product symbol (exterior product between forms will be assumed) and also leave out the subscript that denotes the degree of the forms $\{\beta_{k-1}^1, \ldots, \beta_k^4\}$ in \eqref{ldecompa} and the 16 $\gamma$'s in \eqref{ldecompb}.

The symbols of the first order operators $\{\dpp, \dpm, \dpps, \dpms\}$ in terms of $\{\beta^1, \ldots, \beta^4\}$ in \eqref{ldecompa} and the 16 $\gamma$'s in \eqref{ldecompb} can be expressed as follows:
\begin{align*}
\sigma_{\dpp}&(\xa \wa + \xb \wb + \xc \wc) \beta_k = 
 \zeta_1( w_1 \beta^4 + \tfrac{h}{h+1} \Theta_{12} \beta^2) + \zeta_2 (w_2 \beta^4 - \tfrac{h}{h+1} \Theta_{12} \beta^1)\\
& -\zeta_3 (w_{13}\ga^1 + w_{23}\ga^2 - \Theta_{12}w_3(\tfrac{h(h+2)}{(h+1)^2}\ga^{12}+\tfrac{1}{h-1}\ga^{34})- w_3 \ga^0) \\
&- \zeta_3 \tfrac{h}{h+1}(w_1\Thp\ga^{14}+w_2\Thp\ga^{24} -\Th\Thp\ga^{124} + \tfrac{1}{h}\Th\ga^4 - \tfrac{(h-1)(h+1)}{h^2}\Thp\ga^4)
\\
\sigma_{\dpps}&(\xa \wa + \xb \wb + \xc \wc) \beta_k = -\zeta_1(w_2 \beta^3 + \beta^1)+ \zeta_2 (w_1 \beta^3  - \beta^2)\\
&+\zeta_3(w_1\ga^{13}+w_2\ga^{23}-(\Th+\tfrac{1}{h+2}\Thp)\ga^{123}-\ga^3)\\
&+\zeta_3(w_{14}\ga^{134}+ w_{24}\ga^{234}-\Th w_4\ga^{1234}-w_4(\ga^{34}-\tfrac{1}{h+1}\ga^{12}))
\\
\sigma_{\dpm}&(\xa \wa + \xb \wb + \xc \wc) \beta_k =-\tfrac{1}{h+1}\left[\zeta_1( w_1 \beta^3-\beta^2) + \zeta_2 (w_2  \beta^3+\beta^1) \right]\\
&-\zeta_3\tfrac{1}{h+1}(w_1\ga^{14}+w_2\ga^{24}-(\Th+\tfrac{1}{h+2}\Thp)\ga^{124}-\ga^4)\\
&+\zeta_3\tfrac{1}{h+1}(w_{13}\ga^{134}+w_{23}\ga^{234}-\Th w_3\ga^{1234}-w_3(\ga^{34}-\tfrac{1}{h+1}\ga^{12}))
\\
\sigma_{\dpms}&(\xa \wa + \xb \wb + \xc \wc) \beta_k = 
\zeta_1(- \tfrac{1}{h} w_2 \beta^4 + \tfrac{1}{h+1} \Theta_{12} \beta^1) + \zeta_2 (\tfrac{1}{h} w_1\beta^4  + \tfrac{1}{h+1} \Theta_{12}\beta^2)\\
&-\zeta_3\tfrac{1}{h+1}(w_1\Thp\ga^{13}+w_2\Thp\ga^{23}-\Th\Thp\ga^{123}+\tfrac{1}{h}\Th\ga^3-\tfrac{(h-1)(h+1)}{h^2}\Thp\ga^3)\\
&+\zeta_3\tfrac{1}{h}(w_{14}\ga^1+w_{24}\ga^2-\Th w_4 (\tfrac{h(h+2)}{(h+1)^2}\ga^{12}+\tfrac{1}{h+1}\ga^{34})-w_4\ga^0)
\end{align*}
Note that the expressions for the symbols above take a simpler form when $\xc=0$.  Proceeding further, we can compose the first-order symbols above to obtain the symbols for the Laplacians and their boundary conditions.  The calculations below were performed with the aid of the symbolic computation program Mathematica.

For the second order Laplacian $\Delta_+ = \dpp\dpps + \dpps\dpp$, the principal symbol can be expressed relatively simply when $\zeta=\xa \wa + \xb \wb$.
\begin{align}\label{Dpsyma}
\sigma_{\Delta_+} &(\zeta_1w_1+\zeta_2 w_2) \beta_k 
 = -w_1 \left[\left(\zeta_1^2 +\tfrac{h}{h+1}\zeta_2^2\right) \beta^1 + \tfrac{1}{h+1}(\zeta_1\zeta_2) \beta^2\right] \\
& - w_2  \left[ \tfrac{1}{h+1}(\zeta_1\zeta_2) \beta^1 +\left(\zeta_2^2 +\tfrac{h}{h+1}\zeta_1^2\right) \beta^2\right]- (\zeta_1^2 + \zeta_2^2) \left(\tfrac{h+1}{h+2} \Th \beta^3 + \beta^4\right).\nonumber
\end{align}
More generally, for $\zeta=\xa \wa + \xb \wb+\xc\wc$ we need to use the 16-term decomposition for $\beta_k$ in \eqref{ldecompb}.   The symbol  $\sigma_{\Delta_+}(\xa \wa + \xb \wb + \xc \wc)$ acting on the 16 $\ga$'s in \eqref{ldecompb} can be separated into four independent parts:
\begin{align*}
&\sigma_{\Delta_+} (\xa \wa + \xb \wb + \xc \wc)\beta_k=\\
&~\sigma_{\Delta_+} (\xa \wa + \xb \wb + \xc \wc)\left(\ga^0+\wc \ga^3+ \Thpp\ga^{1234}\right)\\
&+\sigma_{\Delta_+} (\xa \wa + \xb \wb + \xc \wc)\left(\wa \ga^1+\wb \ga^{2}+ \wdd \ga^{4}\right)\\
&+\sigma_{\Delta_+} (\xa \wa + \xb \wb + \xc \wc)\left(\wa\Thp \ga^{134}+ \Th \wc \ga^{123}+\wb\Thp \ga^{234}+ \Th \wdd \ga^{124}\right)\\
&+\sigma_{\Delta_+} (\xa \wa + \xb \wb + \xc \wc)\left(w_{13}\ga^{13}+ w_{24}\ga^{24}+w_{14} \ga^{14}+ w_{23} \ga^{23}+ \Thp \ga^{34}+ \Th \ga^{12}\right)
\end{align*}
with the action on each of the four parts given as follows: 
\begin{align*}
\sigma_{\Delta_+} (\xa \wa + \xb \wb + \xc \wc)(\ga^0 + \wc \ga^3 + \Thpp\ga^{1234})
= -(\xas + \xbs+\xcs) \left[\ga^0 + \wc \ga^3  -\tfrac{h+1}{h+2}\Thpp\ga^{1234}\right]
\end{align*}
\begingroup
\renewcommand*{\arraystretch}{1.25}
\begin{align*}
\sigma_{\Delta_+} (\xa \wa + &\xb \wb + \xc \wc)
\begin{pmatrix} 
\wa \ga^1\\
\wb \ga^{2} \\
\wdd \ga^{4} 
\end{pmatrix}
=
\begin{pmatrix} 
-\xas -\frac{h}{h+1}\xbs -\xcs& -\frac{\xa\xb}{h+1}& -\frac{\xb\xc}{h+1} \\
-\frac{\xa\xb}{h+1}&-\frac{h}{h+1}\xas -\xbs -\xcs& \frac{\xa\xc}{h+1} \\
-\frac{\xb\xc}{h+1}&\frac{\xa\xc}{h+1}&-\xas -\xbs -\frac{h}{h+1}\xcs
\end{pmatrix}
\begin{pmatrix} 
 \wa \ga^1\\
 \wb\ga^{2} \\
\wdd \ga^{4} 
\end{pmatrix}
\end{align*}
\begin{align*}
\sigma_{\Delta_+} (\xa \wa + \xb \wb + \xc \wc)
\SmallMatrix{
\wa\Thp \ga^{134} \\
\Th \wc \ga^{123}\\
 \wb\Thp \ga^{234} \\
\Th \wdd \ga^{124}}
=
\SmallMatrix{
-\xas -\frac{h}{h+1}\xbs-\frac{h+1}{h+2}\xcs& -\frac{\xa\xc}{h+2} & -\frac{\xa\xb}{h+1}&-\frac{\xb\xc}{(h+1)(h+2)}\\
-\frac{\xa\xc}{h+2}& -\frac{h+1}{h+2}(\xas+\xbs) -\xcs & -\frac{\xb\xc}{h+2}& 0\\
-\frac{\xa\xb}{h+1}& -\frac{\xb\xc}{h+2}& -\frac{h}{h+1}\xas -\xbs-\frac{h+1}{h+2}\xcs& \frac{\xa\xc}{(h+1)(h+2)} \\
-\frac{\xb\xc}{(h+1)(h+2)}& 0 &  \frac{\xa\xc}{(h+1)(h+2)}& -\frac{h+1}{h+2}(\xas+\xbs) -\frac{h}{h+1}\xcs}
\SmallMatrix{
\wa\Thp \ga^{134} \\
\Th \wc \ga^{123}\\
 \wb\Thp \ga^{234} \\
\Th \wdd \ga^{124}}
\end{align*}
\begin{align*}
&\sigma_{\Delta_+} (\xa \wa + \xb \wb + \xc \wc)
\SmallMatrix{
w_{13}\ga^{13}\\ 
w_{24}\ga^{24}\\
 w_{14}\ga^{14}\\ w_{23}\ga^{23}\\  \Thp \ga^{34}\\ \Th\ga^{12}}
=\\
&\SmallMatrix{
-\xas -\frac{h}{h+1}\xbs-\xcs& 0 & 0 &-\frac{\xa\xb}{h+1}&\frac{\xb\xc}{h+1}&-\frac{\xb\xc}{(h+1)^2}\\
0& -\frac{h}{h+1}(\xas+\xcs)-\xbs &-\frac{\xa\xb}{h+1} &0&0& \frac{\xb\xc}{(h+1)}   \\
0& -\frac{\xa\xb}{(h+1)} & -\xas -\frac{h}{h+1}(\xbs+\xcs)& 0 & 0 & \frac{\xa\xc}{h+1} \\
-\frac{\xa\xb}{(h+1)}& 0&0& -\frac{h}{h+1}\xas -\xbs-\xcs & -\frac{\xa\xc}{h+1} & \frac{\xa\xc}{(h+1)^2}\\
\frac{h\, \xb\xc}{(h+1)^2} &0&0& -\frac{h}{(h+1)^2}\xa\xc & -\xas -\xbs -\frac{h^2+h+1}{(h+1)^2}\xcs & -\frac{h \xcs}{(h+1)^3}\\
-\frac{\xb\xc}{(h+1)(h+2)} & \frac{\xb\xc}{h+2} & \frac{\xa\xc}{h+2} & \frac{\xa\xc}{(h+1)(h+2)} & -\frac{\xcs}{(h+1)(h+2)}&
 -\frac{h+1}{h+2}(\xas+\xbs) -\frac{h^3+4h^2+5h+1}{(h+1)^2(h+2)}\xcs
}
\SmallMatrix{
w_{13}\ga^{13}\\ 
w_{24}\ga^{24}\\
w_{14}\ga^{14}\\ w_{23}\ga^{23}\\  \Thp \ga^{34}\\ \Th\ga^{12}
}
\end{align*}

Concerning the fourth-order Laplacian, $\Delta_{++}$, it acts on the middle-degree primitive form, $\beta_n$.  When $\zeta=\xa \wa + \xb \wb$, (i.e. $\xc=0$), the symbol can be expressed simply,
\begin{align}\label{Dppsyma}
\sigma_{\Delta_{++}}(\zeta_1w_1+\zeta_2 w_2) \beta_n = (\zeta_1^2 + \zeta_2^2)^2\left[w_1  \beta^1 + w_2  \beta^2  + \tfrac{1}{4}\Theta_{12} \beta^3\right].
\end{align}

More generally, we use the decomposition \eqref{ldecompb} for $\beta_n$, which has only 10 non-zero $\ga$ terms.  The symbol $\sigma_{\Delta_{++}} (\xa \wa + \xb \wb + \xc \wc)$ acting on the ten terms can be separated into three independent parts. 
\begin{align*}
&\sigma_{\Delta_{++}} (\xa \wa + \xb \wb + \xc \wc)\beta_n\\
&=\sigma_{\Delta_{++}} (\xa \wa + \xb \wb + \xc \wc)\left(\Thpp\ga^{1234}\right) \\
&+\sigma_{\Delta_{++}} (\xa \wa + \xb \wb + \xc \wc)\left(\wa\Thp \ga^{134}+ \Th \wc \ga^{123}+\wb\Thp \ga^{234}+ \Th \wdd \ga^{124}\right)\\
&+\sigma_{\Delta_{++}} (\xa \wa + \xb \wb + \xc \wc)\left(w_{13}\ga^{13}+w_{24}\ga^{24}+w_{14} \ga^{14}+ w_{23} \ga^{23}+ \Th \ga^{12}\right)
\end{align*}
with the action 
on each of the three parts given as follows:
\begin{align*}
\sigma_{\Delta_{++}} (\xa \wa + \xb \wb + \xc \wc)\left(\Thpp\ga^{1234}\right)  = \tfrac{1}{4}(\xas + \xbs +  \xcs)^2  \left(\Thpp\ga^{1234}\right)\qquad\quad
\end{align*}
\begin{align*}
&\sigma_{\Delta_{++}} (\xa \wa + \xb \wb + \xc \wc)
\begin{pmatrix} 
\wa\Thp \ga^{134} \\
\Th \wc \ga^{123}\\
\wb\Thp \ga^{234}\\
\Th \wdd \ga^{124}
\end{pmatrix}
=\\
&\tfrac{1}{4}\left(\xas+\xbs+\xcs\right)\begin{pmatrix} 
4(\xas+\xbs) +\xcs& 3\xa\xc &0& -3\xb\xc \\
3\xa\xc& \xas + \xbs +4 \xcs & 3 \xb\xc & 0 \\
0& 3\xb\xc& 4(\xas + \xbs) +\xcs& 3\xa\xc \\
-3\xb\xc&0&3\xa\xc&\xas+\xbs +4\xcs
\end{pmatrix}
\begin{pmatrix} 
\wa\Thp \ga^{134} \\
\Th \wc \ga^{123}\\
\wb\Thp \ga^{234}\\
\Th \wdd \ga^{124}
\end{pmatrix}
\end{align*}
\begin{align*}
&\sigma_{\Delta_{++}} (\xa \wa + \xb \wb + \xc \wc)
\SmallMatrixE{
w_{13}\ga^{13}\\ 
w_{24}\ga^{24}\\w_{14} \ga^{14}\\ w_{23} \ga^{23}\\ \Th \ga^{12}
}=\\
&
\SmallMatrixE{(\xas+\xbs)^2 + \tfrac{1}{2}(4\xas+\xbs)\xcs + \xcf& -\tfrac{3}{2}\xbs\xcs &-\tfrac{3}{2}\xa\xb\xcs &\tfrac{3}{2}\xa\xb\xcs &-\tfrac{3}{2}\xb\xc(\xas+\xbs-\xcs)\\
-\tfrac{3}{2}\xbs\xcs &(\xas+\xbs)^2 + \tfrac{1}{2}(4\xas+\xbs)\xcs + \xcf&-\tfrac{3}{2}\xa\xb\xcs &\tfrac{3}{2}\xa\xb\xcs &-\tfrac{3}{2}\xb\xc(\xas+\xbs-\xcs)\\
-\tfrac{3}{2}\xa\xb\xcs & -\tfrac{3}{2}\xa\xb\xcs & \xaf +\tfrac{1}{2}\xas(4\xbs+\xcs)+(\xbs+\xcs)^2& \frac{3}{2}\xas\xcs &-\tfrac{3}{2}\xa\xc(\xas+\xbs-\xcs) \\
\tfrac{3}{2}\xa\xb\xcs &\tfrac{3}{2}\xa\xb\xcs &\tfrac{3}{2}\xas\xcs & \xaf +\tfrac{1}{2}\xas(4\xbs+\xcs)+(\xbs+\xcs)^2&\tfrac{3}{2}\xa\xc(\xas+\xbs-\xcs)\\
-\tfrac{3}{4}\xb\xc(\xas+\xbs-\xcs) &
-\tfrac{3}{4}\xb\xc(\xas+\xbs-\xcs)  & -\tfrac{3}{4}\xa\xc(\xas+\xbs-\xcs) &\frac{3}{4}\xa\xc(\xas+\xbs-\xcs)& \frac{1}{4}(\xas+\xbs)^2+14(\xas+\xbs)\xcs + \xcf
}
\SmallMatrixE{
w_{13}\ga^{13}\\ 
w_{24}\ga^{24}\\
w_{14} \ga^{14}\\ w_{23}\ga^{23}\\ \Th \ga^{12}
}
\end{align*}
\endgroup

\section{Injectivity Proof of Case (A2) for Propositions \ref{2ndE} and \ref{4thE}}\label{AppC}

For Proposition \ref{2ndE} and Proposition \ref{4thE}, we here give the proof of injectivity of the $\Psi_{x,\xi}$ map in the (A2) case where $\xi= pw_2 + qw_3$ with both $p\neq0$ and $q\neq0$.  As with in Appendix B, the calculations below were performed with the aid of Mathematica.

\subsection{Proposition \ref{2ndE}}\label{AppC1}
To write out the ordinary differential system \eqref{ode2}, we use the calculation of the principal symbol $\sigma_{\Delta_+}(\zeta_1 w_1 + \zeta_2 w_2 + \zeta_3 w_3)$ in  Appendix \ref{AppB}, substituting $\zeta_1 = -i\,\partial_t$, $\zeta_2=p$ and $\zeta_3=q$.  The general $\mathbb{R}_+$-bounded solution of  $\sigma(\Delta_+) ( -i \, w_1 \partial_t + p\, w_2 + q\, w_3)\, \phi_k(t)=0$ with both $p$ and $q$ non-zero can be expressed in terms of 16 constant primitive forms, labelled by $c^l$ for $l=1, \ldots, 16$.  Each $c^l$ has no components in $\{\wa, \wb, \wc, \wdd\}$.  In the solutions below, $r=\sqrt{p^2+q^2}>0$ and we also express $\phi_k(t)$ in terms of the 16 $\ga(t)$'s as in \eqref{bexpand2}, which is identical to the decomposition of \eqref{ldecompb}. 
\begingroup
\renewcommand*{\arraystretch}{1.25}
\begin{align*}
&\ga_k^0(t)= c_k^1 e^{-rt} \,,\quad  \ga^3_{k-1}(t)=  c_{k-1}^2 e^{-rt}\,, \quad \ga_{k-4}^{1234}(t)= c_{k-4}^3 e^{-rt}\,, \\
\end{align*}
\begin{align*}
\begin{pmatrix} 
\ga_{k-1}^1\\
\ga_{k-1}^{2} \\
 \ga_{k-1}^{4} 
\end{pmatrix}&=
c_{k-1}^4\begin{pmatrix}-\frac{(2h+1)r}{p}-p t\\ i r t\\ q t\end{pmatrix} e^{-rt} + c_{k-1}^5 \begin{pmatrix} q\\ 0\\ p\end{pmatrix}e^{-rt}
+ c_{k-1}^6 \begin{pmatrix} i\,r\\ p \\ 0\end{pmatrix}e^{-rt},\\
\end{align*}
\begin{align*}
\begin{pmatrix} 
 \ga_{k-3}^{134} \\
 \ga_{k-3}^{123}\\
 \ga_{k-3}^{234} \\
 \ga_{k-3}^{124}
\end{pmatrix}&=
c_{k-3}^7\begin{pmatrix}-2 (h+1)^2 q r + q^3 t \\  i q^2+ 2i (h+1)(h+2)r^2-i  q^2 r t\\ 0\\  p q^2 t\end{pmatrix} e^{-rt} 
+ c_{k-3}^8\begin{pmatrix}-iq + i q r t \\ - 2(h+2)r + q^2 t\\ p q t\\ 0\end{pmatrix} e^{-rt} \\
&\qquad
+ c_{k-3}^9\begin{pmatrix}q \\ -i r\\ 0 \\ p\end{pmatrix} e^{-rt} 
+ c_{k-3}^{10} \begin{pmatrix} ir \\q\\ p\\ 0\end{pmatrix}e^{-rt},\\
\end{align*}
\begin{align*}
\begin{pmatrix} 
\ga_{k-2}^{13}\\ 
\ga_{k-2}^{24}\\  
\ga_{k-2}^{14}\\ 
\ga_{k-2}^{23}\\  
\ga_{k-2}^{34}\\ 
\ga_{k-2}^{12}
\end{pmatrix}&=
c_{k-2}^{11}\begin{pmatrix}-\frac{(2h+1)r}{p}  - p\,t \\ \frac{(2h+1)r}{p}  - pt\\ -i r t \\ i r t\\ -\frac{h}{h+1}qt\\ q t \end{pmatrix}e^{-rt}
+c_{k-2}^{12}\begin{pmatrix}\frac{q}{h+1} \\ q\\ 0 \\ 0 \\0 \\ p \end{pmatrix}e^{-rt}
+c_{k-2}^{13}\begin{pmatrix}-q\\0\\0\\0\\p\\0\end{pmatrix}e^{-rt}
+c_{k-2}^{14}\begin{pmatrix}ir \\0\\0\\p\\0\\0\end{pmatrix}e^{-rt}
+c_{k-2}^{15}\begin{pmatrix}0\\-ir\\p\\0\\0\\0\end{pmatrix}e^{-rt}\\
&\qquad+c_{k-2}^{16}\begin{pmatrix}
4h(h+1)(2h+1)p + \frac{2h (4h^2 +6h +3) q^2}{p} + 4h(h+1)pr t - p q^2 t^2\\
\frac{2(h+1)(p^2 + (4h+1)r^2)}{p}-4(h+1)prt + p q^2 t^2 \\ 
-2i(p^2+ (2h+1)r^2)t -  i q^2 r t^2\\ 
-2i(q^2+ 2 h (h+1) r^2) t +  i q^2 r t^2\\ 
4h^2 q r t -\frac{h}{h+1} q^3t^2\\
q^3 t^2 \end{pmatrix}e^{-rt}.
\end{align*}  
\endgroup      
For injectivity, we study the kernel of $\Psi_{x, \xi=pw_2+qw_3}$ which imposes the following conditions:
\begin{align*}
b_1(x, -i\, w_1 \partial_t + p w_2 + qw_3)\phi_k(t)\big|_{t=0}\,=0
&\Longrightarrow  \left\{\begin{aligned}
&\ga^{0}(0) =0, \ga^{2}(0) =0, \,\ga^{3}(0) =0, \,\ga^{4}(0) =0,\\
& \ga^{234}(0) =0, \,  \ga^{23}(0) =0, \,\ga^{24}(0) =0, \ga^{34}(0) =0, 
\end{aligned}
\right.\\
b_2(x, -i\, w_1\partial_t + p w_2 + q w_3)\phi_k(t)\big|_{t=0}\,=0
&\Longrightarrow\left\{
\begin{aligned}
&\pa_t\ga^{123}(0)=0,\, \pa_t\ga^{1234}(0)=0,\\
&i\pa_t \ga^{1}(0) - p \ga^2(0) -  q \ga^3(0) =0,\\
&i \pa_t \ga^{124}(0) + q \ga^{234}(0) =0,\\
&i \pa_t \ga^{134}(0) -p \ga^{234}(0)- \tfrac{q}{h+2} \ga^{123}(0) =0, \\
&i\pa_t \ga^{13}(0) - p \ga^{23}(0) =0,\\
&i\pa_t \ga^{12}(0) + q \ga^{23}(0) =0,\\
&i \pa_t \ga^{14}(0) -p \ga^{24}(0)- q \ga^{34}(0) + \tfrac{q}{h+1}\ga^{12}(0) =0.\\
\end{aligned}
\right. 
\end{align*}
It can then be checked that imposing the above 16 boundary condition equations on the 16 $\ga$'s at $t=0$ requires that all 16 constants $c^{l}$, for $l=1,\ldots, 16$, in the general $\mathbb{R_+}$-bounded solutions are identically zero.  Therefore, $\Psi_{x, \xi=pw_2+qw_3}$ is injective when both $p\neq0$ and $q\neq0$.

\subsection{Proposition \ref{4thE}}\label{AppC2}

Using the calculation of the principal symbol of $\Delta_{++}$ in Appendix \ref{AppB}, the general $\mathbb{R}_+$-bounded solution $\phi_n(t) \in \mathcal{M}^{+}_{x, \xi = pw_2+qw_3}$ for  $\sigma_{\Delta_{++}} ( -i \, w_1 \partial_t + p \,w_2 + q\, w_3)\, \phi_n(t)=0$ with both $p\neq 0$ and $q\neq0$ can be solved and expressed in terms of the 10 $\ga(t)$'s of \eqref{bexpand4}, which is identical to the decomposition of \eqref{ldecompb} when $k=n$.
\begingroup
\renewcommand*{\arraystretch}{1.2}
\begin{align*}
\ga_{n-4}^{1234}= (c_{n-4}^{1}+c_{n-4}^{2}t) e^{-rt},
\end{align*}
\begin{align*}
\begin{pmatrix} 
 \ga_{n-3}^{134} \\
\ga_{n-3}^{123}\\
 \ga_{n-3}^{234}\\
 \ga_{n-3}^{124}
\end{pmatrix}&= c_{n-3}^3\begin{pmatrix}1 \\ 0 \\ 0 \\ 0 \end{pmatrix}e^{-rt}
+c_{n-3}^4\begin{pmatrix}0 \\ 1 \\ 0 \\ 0 \end{pmatrix}e^{-rt}
+c_{n-3}^5\begin{pmatrix}0 \\ 0 \\ 1 \\ 0 \end{pmatrix}e^{-rt}
+c_{n-3}^6\begin{pmatrix}0 \\ 0 \\ 0 \\ 1 \end{pmatrix}e^{-rt}
+c_{n-3}^7\begin{pmatrix}i r t \\ qt\\p t\\0 \end{pmatrix}e^{-rt}\\
&\quad
+c_{n-3}^8\begin{pmatrix}qt \\ -i r t\\0\\pt \end{pmatrix}e^{-rt}
+c_{n-3}^{9}\begin{pmatrix}-2iq t+irqt^2\\-\frac{16r}{3}t+q^2t^2\\pqt^2\\0\end{pmatrix}e^{-rt}
+c_{n-3}^{10}\begin{pmatrix} \frac{4}{3}rt+q^2t^2\\ 2iq t - ir qt^2\\0\\ pqt^2\end{pmatrix}e^{-rt},\\
%
\begin{pmatrix} 
 \ga_{n-2}^{13}\\ 
 \ga_{n-2}^{24}\\
 \ga_{n-2}^{14}\\ 
 \ga_{n-2}^{23}\\ 
 \ga_{n-2}^{12}
\end{pmatrix}&=
(c_{n-2}^{11}+c_{n-2}^{12}t)\begin{pmatrix}-1 \\ 1\\ 0\\0\\0 \end{pmatrix}e^{-rt} 
+c_{n-2}^{13}\begin{pmatrix}-ir \\ 0\\ p\\0\\0 \end{pmatrix}e^{-rt} 
+c_{n-2}^{14}\begin{pmatrix}i-irt \\ 0\\ pt\\0\\0 \end{pmatrix}e^{-rt} \\
&\,
+c_{n-2}^{15}\begin{pmatrix}ir \\ 0\\ 0\\p\\0 \end{pmatrix}e^{-rt} 
+c_{n-2}^{16}\begin{pmatrix}-i+ irt \\ 0\\ 0\\pt\\0 \end{pmatrix}e^{-rt}
+c_{n-2}^{17}\begin{pmatrix}2q \\ 0\\ 0\\0\\p \end{pmatrix}e^{-rt} 
+c_{n-2}^{18}\begin{pmatrix}-2r+2q^2t \\ 0\\ 0\\0\\pqt \end{pmatrix}e^{-rt}  \\
&\, 
+c_{n-2}^{19}\begin{pmatrix}-\frac{6p^2+10r^2}{3p} - pq^2t^2 \\ -pq^2 t^2\\-irq^2 t^2\\ irq^2 t^2\\ q^3 t^2 \end{pmatrix}e^{-rt}
+c_{n-2}^{20}\begin{pmatrix}
-16pr-\frac{2q^2(2p^2 + 5q^2)}{p} -3pq^2rt^2 - pq^4t^3\\
-3pq^2rt^2 - pq^4t^3\\
-3ip^2q^2t^2- i r q^4 t^3\\
3ip^2q^2t^2+ i r q^4 t^3\\
q^5t^3
\end{pmatrix}e^{-rt},
\end{align*}
\endgroup
where all twenty primitive forms $c^l$ for $l=1, \ldots, 20$ are constant  forms without components in $\{w_1, w_2, w_3, w_4\}$ and $r=\sqrt{p^2+q^2}>0$.

For injectivity, the image of the $\Psi_{x, \xi=pw_2+qw_3}$ map vanishes imposes the following conditions: 
\begin{align*}
b_1(-i\, w_1\partial_t + p w_2 + q w_3)\phi_n(t)\big|_{t=0}\,=0
&\Longrightarrow \ga^{234}(0)=0,\ \ga^{23}(0)=0,\ \ga^{24}(0)=0, \\
b_2(-i\, w_1\partial_t + p w_2 + q w_3)\phi_n(t)\big|_{t=0}\,=0
&\Longrightarrow \left\{\begin{aligned}
&\pa_t\ga^{234}(0) - ip \ga^{134}(0) + iq\ga^{124} =0,\\
&\pa_t \ga^{23}(0) -ip \ga^{13}(0) + 2iq\ga^{12}=0,\\
&\pa_t \ga^{24}(0) - ip \ga^{14}(0) =0,
\end{aligned}
\right.
\\
b_3(-i\, w_1\partial_t + p w_2 + q w_3)\phi_n(t)\big|_{t=0}\,=0
&\Longrightarrow \left\{
\begin{aligned}
&\pa_t\ga^{1234}(0)=0,\ \pa_t\ga^{123}=0\\
&i\pa_t\ga^{134}(0)- p \ga^{234}(0)- \tfrac{q}{2} \ga^{123}(0)  =0,\\
&i\pa_t\ga^{124}(0)+q\ga^{234}(0)=0,\\
&i\pa_t\ga^{13}(0)-p\ga^{23}(0)=0,\\
&i\pa_t\ga^{12}(0)+q\ga^{23}(0)=0,\\
&i\pa_t\ga^{14}(0)-p\ga^{24}(0)+q\ga^{12}=0,
\end{aligned}
\right.
\end{align*}
\begin{align*}
&b_4(-i\, w_1\partial_t + p w_2 + q w_3)\phi_n(t)\big|_{t=0}\,=0
\Longrightarrow \\
&\left\{
\begin{aligned}
&\pa_t(\pa_t^2 - r^2)\ga^{1234}\big|_{t=0}=0,\\
&\left\{i(\pa_t^3 -(p^2+\tfrac{3q^2}{4})\pa_t)\ga^{134} -\tfrac{q}{4}(3\pa_t^2-(3p^2+2q^2))\ga^{123} -p(\pa_t^2 -(p^2+\tfrac{3q^2}{4}))\ga^{234}
\right\}\big|_{t=0}=0,\\
&\left\{i(\pa_t^3 -(p^2+2q^2))\ga^{123}-q\pa_t^2\ga^{134}  -ipq\pa_t\ga^{234} 
\right\}\big|_{t=0}=0,\\
&\left\{i(\pa_t^3 -r^2\pa_t)\ga^{124}+q(\pa_t^2-2p^2 -q^2)\ga^{234}  +ipq\pa_t\ga^{134} -pq^2\ga^{123} 
\right\}\big|_{t=0}=0,\\
&\left\{i(\pa_t^3 -r^2\pa_t) \ga^{13} -p(\pa_t^2 - r^2)\ga^{23}
\right\}\big|_{t=0}=0,\\
&\left\{i(\pa_t^3 -(p^2 +\tfrac{q^2}{2})\pa_t)\ga^{14} -p(\pa_t^2 -p^2 -\tfrac{q^2}{2})\ga^{24} + \tfrac{3}{2}q(\pa_t^2-p^2-\tfrac{q^2}{3})\ga^{12} - \tfrac{i}{2} q^2\pa_t^2\ga^{23} - \tfrac{1}{2} pq^2\ga^{13}
\right\}\big|_{t=0}=0,\\
&\left\{i(\pa_t^3 -(p^2 +3q^2)\pa_t)\ga^{12} +q(\pa_t^2 -2r^2)\ga^{23} + q\pa_t^2\ga^{14} +ipq\pa_t\ga^{13} +ipq\pa_t\ga^{24}
\right\}\big|_{t=0}=0.\\
\end{aligned}
\right.
\end{align*}
It can then be checked that the above 20 boundary condition equations are only satisfied if the twenty primitive constants, $c^I$ for $I=1, \dots, 20$, of the general $\mathbb{R_+}$-bounded solutions all vanish, thus proving that the map $\Psi_{x, \xi=pw_2+qw_3}$ is injective when both $p\neq0$ and $q\neq0$.


\end{document}